\documentclass[10pt,a4paper,reqno]{amsart}
\usepackage{vmargin,color}
\usepackage[latin1]{inputenc}

\usepackage{amsmath,amsfonts,amsthm,epsfig,graphicx,mathtools,tikz-cd, amssymb,setspace,enumitem,amsaddr}
\usepackage{caption}
\usepackage[colorlinks,linkcolor=blue,citecolor=blue]{hyperref}
\usepackage[capitalise]{cleveref}
\usepackage{autonum}
\usepackage{tikz-cd}

\def\B{\mathcal{B}}

\def\cE{\mathcal{E}}
\def\cF{\mathcal{F}}

\def\\rho{\mathcal{\rho}}
\def\cP{\mathcal{P}}

\def\N{\mathbb N}
\def\Z{\mathbb Z}

\def\R{\mathbb R}
\def\T{\mathbb T}

\def\\rho\rho{\mathbf{\rho}}
\def\HH{\mathbf{H}}

\def\eps{\varepsilon}

\def\Im{{\rm Im\,}}

\DeclareMathOperator{\esup}{ess\,sup}
\DeclareMathOperator{\einf}{ess\,inf}
\DeclareMathOperator{\eosc}{ess\,osc}

\def\weak{\rightharpoonup}

\DeclareMathAlphabet{\mathup}{OT1}{\familydefault}{m}{n}
\newcommand{\dx}[1]{\mathop{}\!\mathup{d} #1}

\newcommand{\intom}[1]{\int_{\Omega} #1 \dx{x}}
\newcommand{\iintom}[1]{\iint_{\Omega\times \Omega} #1 \dx{x} \dx{y}}
\newcommand{\Sym}{\mathrm{Sym}}
\DeclarePairedDelimiter{\abs}{\lvert}{\rvert}
\DeclarePairedDelimiter{\norm}{\lVert}{\rVert}
\DeclarePairedDelimiter{\bra}{(}{)}
\DeclarePairedDelimiter{\pra}{[}{]}
\DeclarePairedDelimiter{\set}{\{}{\}}
\DeclarePairedDelimiter{\skp}{\langle}{\rangle}

\newcommand{\Leb}{\ensuremath{{L}}}
\newcommand{\SobH}{\ensuremath{{H}}}

\newtheorem{theorem}{Theorem}

\newtheorem{proposition}[theorem]{Proposition}
\newtheorem{lemma}[theorem]{Lemma}
\newtheorem{corollary}[theorem]{Corollary}

\theoremstyle{definition}
\newtheorem{remark}[theorem]{Remark}
\newtheorem{defn}[theorem]{Definition}
\DeclareMathOperator{\Ima}{Im}

\newcommand{\cb}[1]{{#1}}
\newcommand{\cre}[1]{{#1}}

\newenvironment{tenumerate}[1][]
  {\enumerate[label=\textup(\alph*\textup),ref=(\alph*),#1]}
  {\endenumerate}

\newcommand{\Dref}[1]{\dodref\Cref#1\relax}
\def\dodref#1#2!#3\relax{#1{#2}\ref{#2!#3}}

\allowdisplaybreaks

\mathchardef\mhyphen="2D

\numberwithin{equation}{section}
\numberwithin{figure}{section}
\numberwithin{theorem}{section}
\title{Phase transitions for nonlinear nonlocal aggregation-diffusion equations }
\author{Jos\'e A. Carrillo}
\address[A1]{Mathematical Institute, University of Oxford, Oxford OX2 6GG, UK}
\author{Rishabh S. Gvalani}
\address[A2]{Department of Mathematics, Imperial College London, London SW7 2AZ, UK}
\email[A1,A2]{carrillo@maths.ox.ac.uk\textsuperscript{$\ast$}, rg1314@ic.ac.uk}
\thanks{JAC was partially supported by EPSRC grant number EP/P031587/1 and the Advanced Grant Nonlocal-CPD (Nonlocal PDEs for Complex Particle Dynamics: Phase Transitions, Patterns and Synchronization) of the European Research Council Executive Agency (ERC) under the European Union's Horizon 2020 research and innovation programme (grant agreement No. 883363). RSG was funded by an Imperial College President's PhD Scholarship, partially through EPSRC Award Ref. 1676118. Part of this work was carried out at the {\emph ``Junior Trimester Programme in Kinetic Theory''} held at the Hausdorff Research Institute for Mathematics, Bonn. RSG is grateful to the institute for its hospitality. }
\begin{document}
\maketitle
\begin{abstract}
We are interested in studying the stationary solutions and phase transitions of aggregation equations with degenerate diffusion of porous medium-type, with exponent $1 < m < \infty$.  We  first prove the existence of possibly infinitely many bifurcations from the spatially homogeneous steady state. We then focus our attention on the associated free energy proving existence of minimisers and even uniqueness for sufficiently weak interactions. In the absence of uniqueness, we show that the system exhibits phase transitions: we classify values of $m$ and interaction potentials $W$ for which these phase transitions are continuous or discontinuous. Finally, we comment on the limit $m \to \infty$ and the influence that the presence of a phase transition has on this limit.
\end{abstract}


 \setstretch{1.2}

\section{Introduction}

In this work, we deal with the properties of the set of stationary states and long-time asymptotics for a general class of nonlinear aggregation-diffusion equations of the form
\begin{align}
\begin{cases}
\partial_t \rho = \beta^{-1} \Delta \rho^m + \nabla \cdot (\rho \nabla W \star \rho)  & (x,t) \in \Omega \times (0,T] \\ 
\rho(x,0)=\rho_0(x) \in \Leb^2(\Omega) \cap \Leb^m(\Omega) \cap \cP(\Omega) & x \in \Omega
\end{cases} \, ,
\label{eq:agdif1}
\end{align}
where $1 < m <\infty$ is the nonlinear diffusivity exponent of porous medium type \cite{Va}, $\beta>0$ measures the relative strength between repulsion (by nonlinear diffusion) and attraction-repulsion (by the nonlocal aggregation terms), and $W \in C^2(\Omega)$ is the attractive-repulsive interaction potential. Here $\Omega$ denotes the $d$-dimensional torus $\T^d$ having side length $L>0$, with $\cP(\Omega)$ being the set of Borel probability measures on $\Omega$, and $\Leb^m(\Omega)$ the set of $m$-power integrable functions on $\Omega$. Notice that for $m=1$ we recover the linear diffusion case which is related to certain nonlocal Fokker-Planck equations, also referred to as McKean--Vlasov equations in the probability community. These equations also share the feature of being gradient flows of free energy functionals of the form
\cb{
\begin{align}\label{fenergy}
\cF_\beta^m (\rho) := 
\begin{cases}
\displaystyle\frac{\beta^{-1}}{m-1} \intom{\rho^m(x)}- \frac{\beta^{-1}}{m-1} + \frac{1}{2}\iintom{W(x-y) \rho(x) \rho(y)}, & m>1 \\&
\\
\displaystyle\beta^{-1} \intom{(\rho \log \rho)(x)} + \frac{1}{2}\iintom{W(x-y) \rho(x) \rho(y)}, & m=1
\end{cases}
\end{align}
}
for $\rho \in \Leb^m(\Omega) \cap \cP(\Omega)$, as discussed extensively in the literature \cite{JKO,Otto,Villani,cmcv-03,AGS}. We refer to \cite{CCY} for a recent survey of this active field of research. \cre{Note that although we have included the free energy for $m=1$ in~\eqref{fenergy}, we will mostly be dealing with case $m>1$ in this article. We will only discuss the case $m=1$ as a limiting case of the energies $\cF_\beta^m$ as $m \to 1$. The case $m=1$ is treated in more detail in~\cite{CGPS19}.}

Aggregation-diffusion equations such as \eqref{eq:agdif1} naturally appear in mathematical biology \cite{BCM,VS,CMSTT,BDZ,BCDPZ} and mathematical physical contexts \cite{Oe,P07,FP08,BV} as the typical mean-field limits of interacting particle systems of the form
\begin{align}
dX^i_t = -\frac{1}{ N} \sum\limits_{i \neq j}^N\nabla W^N(X^i_t -X^j_t) \dx{t} + \sqrt{2 \beta_2^{-1}}\, dB^i_t \, ,
\label{eq:lang}
\end{align}
where $W^N=\frac1{\beta_1} \varphi^N +  W$ and 
$
 \varphi^N (x) = N^{\xi}  \varphi(N^{\xi/d}x), \mbox{ for all } x\in \R^d \,.
$
Here, $\varphi$ is a the typical localized repulsive potential, for instance a Gaussian, and $0<\xi<1$. Notice that due to the choice of $\xi$, the shape of the potential gets squeezed to a Dirac Delta at 0 slower than the typical relative particle distance $N^{-1/d}$. \cb{Also, $\beta_2^{-1} \geq 0$ is the strength of the independent Brownian motions driving each particle.} We refer to \cite{Oe,P07,BV} for the case of quadratic diffusion  $m=2$ with $\beta_1=\beta$, $\nu=0$, and to \cite{FP08} for related particle approximations for different exponents $m$. The McKean--Vlasov equation $m=1$ is obtained for the particular case $\beta_1=+\infty$ and $\beta_2=\beta$, being the inverse temperature of the system for the linear case, and its derivation is classical for regular interaction potentials $W$, see for instance \cite{Sz}. 

Analysing the set of stationary states of the aggregation-diffusion equation \eqref{eq:agdif1} and their properties depending on $\beta$, the relative strength of repulsion by local nonlinear diffusion and attraction-repulsion by nonlocal interactions, is a very challenging problem. As with the linear case, the flat state 
\cre{
\begin{align}
\rho_\infty:=L^{-d}=\abs{\Omega}^{-1} \, ,
\label{flat}
\end{align}
}
 is always a stationary solution of the system. The problem lies in constructing nontrivial stationary solutions and minimisers. In the linear diffusion case $m=1$, we refer to \cite{CP10,CGPS19} where  quite a complete picture of the appearance of bifurcations and of continuous and discontinuous phase transitions is present, under suitable assumptions on the interaction potential $W$. Bifurcations of stationary solutions depending on a parameter are usually referred in the physics literature as phase transitions \cite{Daw83}. In this work we make a distinction between the two: referring to the existence of nontrivial stationary solutions as bifurcations and the existence of nontrivial minimisers of $\cF_\beta^m$ as phase transtions. Particular instances of phase transitions related to aggregation-diffusion equations with linear diffusion have been recently studied for the case of the Vicsek-Fokker-Planck equation on the sphere \cite{DFL,FL} and the approximated homogeneous Cucker-Smale approximations in the whole space \cite{Tu,BCCD,ABCD}. We also refer to~\cite{Sch85} where the problem was studied on a bounded domain for the Newtonian interaction, and to~\cite{Tam84} where the problem was studied on the whole space with a confining potential.

However, there are no general results in the literature for the nonlinear diffusion case \eqref{eq:agdif1}, $m>1$, except for the particular case of $m=2$, $d=1$, with $W$ given by the fundamental solution of the Laplacian with no flux boundary conditions (the Newtonian interaction) recently studied in \cite{CCWWZ2019}. Despite the simplicity of the setting in \cite{CCWWZ2019}, this example revealed how complicated phase transitions for nonlinear diffusion cases could be. The authors showed that infinitely many discontinuous phase transitions occur for that particular problem. Let us mention that the closer result in the periodic setting is \cite{CKY13}, where the authors showed that no phase transitions occur for small values of $\beta$, when the flat state is asymptotically stable, for $m \in(1,2]$.

Our main goal is thus to develop a theory for the stationary solutions and phase transitions of \eqref{eq:agdif1} for general interactions $W \in C^2(\Omega)$ and nonlinear diffusion in the periodic setting, something that has not been previously studied in the literature. This paper can be thought of as an extension of the results in~\cite{CGPS19} to the setting of nonlinear diffusion.  Considering this, we need to define appropriately the notion of phase transition for the case $m \in (1, \infty)$, as done in~\cite{CP10} for the linear case $m=1$. 

Note that, unlike in the linear setting, the $\Leb^1(\Omega)$ topology is not the natural topology to define phase transitions. It seems that for $m>1$ the correct topology to work in is $\Leb^\infty(\Omega)$ (cf. \cref{def:tp,topology} below). For our results we will often require compactness of minimisers in this topology. One possible way of obtaining this compactness is via control of the H\"older norms of the stationary solutions of \eqref{eq:agdif1}. In \cref{exreg} we briefly comment on the existence of solutions to \eqref{eq:agdif1} before proceeding to the proof of H\"older regularity. Since this is a key element of the subsequent results and the proof of H\"older regularity for such equations is not in the literature we include the proof in full detail in \cref{exreg}. It relies on the so-called method of intrinsic scaling introduced by DiBenedetto for the porous medium equation (cf. ~\cite{DiB79}), which is a version of the De Giorgi--Nash--Moser iteration adapted to the setting of degenerate parabolic equations. We make modifications to the method to deal with the presence of the nonlocal drift term $\nabla \cdot(\rho \nabla W \star \rho)$. We remark here that the proof of this result is completely independent of the rest of the paper. In a first reading, readers more interested in the properties of stationary solutions and phase transitions might choose to skip the proof and continue to~\cref{bif}. \cb{As a consequence of the proof of H\"older regularity, we also obtain uniform-in-time equicontinuity of the solutions away from the initial datum in~\cref{cor:holder}.}

After the proof of the H\"older regularity we proceed to~\cref{bif}, where we discuss the local bifurcations of stationary solutions from the flat state $\rho_\infty$. In~\cref{thm:bif}, we \cre{provide conditions on the interaction potential $W$ and on the parameter  $\beta=\beta_*$}, such that $(\rho_\infty,\beta_*)$ is a bifurcation point using the Crandall--Rabinowitz theorem (cf.~\cref{thm:cr}). In fact for certain choices of $W$ one can show that there exist infinitely many such bifurcation points. We then move on to~\cref{transitions}, where we prove the existence and regularity of minimisers $\cF_\beta^m$. We also show that, for $\beta$ small enough, the flat state is the unique minimiser of the energy for $m \in (1,\infty]$, thus extending the result of~\cite{CKY13}. \cb{In~\cref{ltb}, we use the uniform equicontinuity in time obtained in~\cref{cor:holder} to prove that solutions of~\eqref{eq:agdif1} converge to $\rho_\infty$ in $\Leb^\infty(\Omega)$ whenever it is the unique stationary solution.}  We show that, as in the linear case, the notion of $H$-stability (cf.~\cref{def:Hstab}), provides a sharp criterion for the existence or non-existence of phase transitions. We then proceed, in~\cref{lem:ctp@cs,lem:exnu,prop:strat}, to provide sufficient conditions for the existence of continuous or discontinuous phase transitions, where the proofs rely critically on the H\"older regularity obtained in~\cref{exreg}. We also provide general conditions on $W$ for the existence of discontinuous phase transitions. We conclude the section by showing that $m \in[2,3]$ all non-$H$-stable potentials $W$ are associated with discontinuous phase transitions of $\cF_\beta^m$, while for $m=4$ we can construct a large class of $W$ that lead to continuous phase transitions of $\cF_\beta^m$. 
We summarise our results below:
\begin{enumerate}
\item The proof of H\"older regularity of the weak solutions of \eqref{eq:agdif1} can be found in~\cref{holder} and the preceding lemmas of~\cref{exreg}.
\item The result on the existence  of local bifurcations  of the stationary solutions is contained in~\cref{thm:bif}.
\item The results  on phase transitions are spread out throughout~\cref{transitions}. \cb{The result on the long-time behaviour of solutions before or in the absence of a phase transition can be found in~\cref{ltb}.} The main result on the existence of discontinuous transition points is~\cref{thm:dctp} while the explicit conditions for a continuous transition point can be found in~\cref{m=4}.
\item In~\cref{mesa}, we treat the mesa limit $m \to \infty$. The $\Gamma$-convergence of the sequence of energies $\cF_\beta^m$  to some limiting free energy $\cF^\infty$ as $m \to \infty$ can be found in~\cref{thm:mesa}. We then provide a characterisation of the minimisers of the limiting variational problem in terms of the size of the domain and the potential $W$ in~\cref{thm:mesapt}.
\end{enumerate}
In~\cref{numexp}, we display the results of some numerical experiments which we hope will shed further light on the theoretical results, while also providing us with some conjectures about the behaviour of the system in settings not covered by the theory.


\section{Preliminaries and notation}
 As mentioned earlier, we denote by $\cP(\Omega)$ the space of all Borel probability measures on $\Omega$ with $\rho$ the generic element which we will often associate with its density $\rho(x) \in \Leb^1(\Omega)$, if it exists. We use the standard notation of $\Leb^p(\Omega)$ and $ \SobH^s(\Omega)$ for the Lebesgue and periodic $\Leb^2$-Sobolev spaces, respectively.  We denote by the $C^k(\Omega),C^\infty(\Omega)$ the space of $k$-times ($k \in \N$) continuously differentiable and smooth functions, respectively.  

Given any function in $f \in \Leb^2(\Omega)$ we define its Fourier transform as 
\begin{align}
\hat{f}(k)= \skp{f, e_k}_{\Leb^2(\Omega)}, \qquad k \in \Z^d
\end{align} 
where
\begin{align}\label{e:def:wk}
e_k(x)= N_k\prod\limits_{i=1}^d e_{k_i}(x_i), \qquad\text{ where } \qquad
e_{k_i}(x_i)=
  \begin{cases}
    \cos\left(\frac{2 \pi k_i}{L} x_i\right) &  k_i>0, \\
    1 & k_i=0,  \\
    \sin\left(\frac{2 \pi k_i}{L} x_i\right)  & k_i<0, \\
  \end{cases}
\end{align}
and $N_k$ is defined as
\begin{align}\label{NkTk}
N_k:=\frac{1}{L^{d/2}}\prod\limits_{i=1}^d \left(2-\delta_{k_i , 0} \right)^{\frac{1}{2}}= :\frac{\Theta(k)}{L^{d/2}} \, .
\end{align}
Using this we have the following representation of the convolution of two functions $W, f  \in \Leb^2(\Omega)$ where $W$ is even along every coordinate
\begin{align}
(W \star f)(y)= \sum\limits_{k \in \N^d}\hat{W}(k) \frac{1}{N_k} \sum\limits_{\sigma \in \Sym_k(\Lambda)}\hat{f}(\sigma(k))e_{\sigma(k)}(y)  \, .
\end{align}
where $\Sym_k(\Lambda)=\Sym(\Lambda)/H_k$.  $\Sym(\lambda)$ represents the symmetric group of the product of two-point spaces, $\Lambda=\{1,-1\}^d$, which acts on $\Z^d$ by pointwise multiplication, i.e. $(\sigma(k))_i=\sigma_i k_i, k \in \Z^d, \sigma \in \Sym(\Lambda)$.  $H_k$ is a normal subgroup of $\Sym(\Lambda)$ defined as follows 
\begin{align}
H_k:= \set*{\sigma \in \Sym(\Lambda): \sigma(k)=k} \,.
\end{align}
We need to quotient out $H_k$ as there might be some repetition of terms in the sum $\sum_{\sigma \in \Sym(\Lambda)}$ if $k \in \N^d$ is such that $k_i=0$ for some $i \in \set{1, \dots,d}$. Another expression that we will use extensively in the sequel is the Fourier expansion of the following bilinear form
\begin{align}\label{Fourier:Interaction}
\iint\limits_{\Omega \times \Omega} \! W(x-y)f(x) f(y) \dx{x} \dx{y} = \sum\limits_{k \in \N^d} \hat{W}(k)\frac{1}{N_k}\sum\limits_{\sigma \in \Sym_k(\Lambda)}|\hat{f}(\sigma(k))|^2 \, .
\end{align}
The following notion will play an important role in the subsequent analysis.
\begin{defn}\label{def:Hstab}
A potential $W \in \Leb^2(\Omega)$ is said to be $H$-stable denoted by $W \in \HH_s$ if
\begin{align}
\hat{W}(k) \geq 0,\quad  \forall k \in \Z^d, k \neq0  \,.
\end{align}
If this does not hold, we denote this by $W \in \HH_s^c$. \cre{ The above condition is equivalent to the following  inequality holding true for all $\eta \in \Leb^2(\Omega)$ : 
\begin{align}
\iintom{W(x-y) \eta(x) \eta(y)} \geq 0 \, .
\end{align}
Furthermore, if $\eta ,W \not \equiv 0$, we have that
\begin{align}
\iintom{W(x-y) \eta(x) \eta(y)} > 0 \, .
\end{align}
}
\end{defn}


\section{Existence and regularity of solutions}\label{exreg}
We are interested in solutions of the following nonlinear-nonlocal PDE
\begin{align}
\begin{cases}
\partial_t \rho = \beta^{-1} \Delta \rho^m + \nabla \cdot (\rho \nabla W \star \rho)  & (x,t) \in \Omega \times (0,T] \\ 
\rho(x,0)=\rho_0(x) \in \Leb^2(\Omega)\cap \Leb^m(\Omega) \cap \cP(\Omega) & x \in \Omega
\end{cases} \, ,
\label{eq:agdif}
\end{align}
where \cre{$1 < m <\infty$, $\beta>0$}, and $W \in C^2(\Omega)$ is even along every co-ordinate and has mean zero. It is not immediately clear what the correct notion of solution for the above PDE is, as it need not possess classical solutions.  We introduce the appropriate notion of solution in the following definition.
\begin{defn}\label{thm:weak1}
A weak solution of~\eqref{eq:agdif} is a bounded, measurable function
\[
\rho \in C([0,T]; \Leb^2(\Omega)) 
\]
with
\[
\rho^m \in \Leb^2([0,T]; \SobH^1(\Omega)) \, ,
\]
such that
\begin{align}
&\left.\intom{\rho(x,t) \phi(x,t)}\right\lvert^T_0 \\&+ \int_{0}^T \intom{\bra*{- \rho(x,t) \phi_t(x,t) + \beta^{-1}m\rho^{m-1}(x,t)\nabla \rho(x,t) \cdot \nabla \phi(x,t)+ \rho \nabla( W \star \rho(x,t)) \cdot \nabla \phi(x,t) }} \dx{t} =0 \, , \label{eq:weakform}
\end{align}
for all $\phi \in \SobH^1([0,T];\Leb^2(\Omega)) \cap \Leb^2([0,T]; \SobH^1_0(\Omega))$  and $\rho(x,0)=\rho_0(x)$.
\end{defn}
\begin{theorem}
Given  $\rho_0 \in \Leb^2(\Omega) \cap \Leb^m(\Omega) \cap \cP(\Omega)$, there exists a unique weak solution of~\eqref{eq:agdif}. Furthermore $\rho(\cdot,t) \in \cP(\Omega)$ for all $t \geq 0$. 
\end{theorem} 
The proof of this result is classical and we will not include it. It relies on regularisation techniques which remove the degeneracy in the problem. The meat of the matter is proving estimates uniform in the regularisation parameter. We refer to~\cite{BCL09,BS10} for proofs of this result with $W \in C^2(\Omega)$. 

We turn our attention to the regularity of solutions of~\eqref{eq:agdif}. The proof is based on the method of intrinsic scaling introduced by DiBenedetto for the porous medium equation~\cite{DiB79,UrB08}. It is also similar in spirit to the proof in~\cite{KZ18} where regularity was proved for a degenerate diffusion equation posed on $\R^d$ with a potentially singular drift term. \cb{We also direct the readers to~\cite{HZ19} where H\"older regularity was proven for drift-diffusion equations with sharp conditions on the drift term using a different strategy of proof.} Since we will mainly be concerned with stationary solutions we assume for the time being that there exists some universal constant $M>0$ such that $\norm{\rho}_{\Leb^\infty(\Omega_T)} \leq M$, where $\Omega_T$ is the parabolic domain $\Omega_T:=\Omega \times[0,T]$ and $\Omega_\infty:=\Omega \times[0,\infty)$. We first state the result regarding H\"older regularity.
\begin{theorem}\label{holder}
Let $\rho$ be a weak solution of~\eqref{eq:agdif} with initial datum $\rho_0 \in \Leb^\infty(\Omega) \cap \cP(\Omega)$, such that $\norm{\rho}_{\Leb^\infty(\Omega_T)} \leq M< \infty$. Then $\rho$ is H\"older continuous with exponent $a \in (0,1)$ dependent on the data, $m$, $d$, $W$, and $\beta$. Moreover, the H\"older exponent $a$ depends continuously on $\beta$ for $\beta>0$.
\end{theorem}

We also have the following consequence of the above result:
\cb{
\begin{corollary}\label{cor:holder}
Let $\rho$ be a weak solution of~\eqref{eq:agdif} with initial datum $\rho_0 \in \Leb^\infty(\Omega) \cap \cP(\Omega)$, such that $\norm{\rho}_{\Leb^\infty(\Omega_\infty)} \leq M< \infty$. Then, for some $C>0$, it holds that 
\begin{align}
\abs{\rho(y,t_1)-\rho(x,t_2)} \leq C_{h} \bra*{d_{\T^d}(x,y) + \abs{t_1-t_2}^{1/2}}^a \, ,
\end{align}
for all $x,y \in \T^d$ and $0<C<t_1<t_2<\infty$. Note that the constants $C_h$ and $a$ are independent of $x,y$ and $t_1,t_2$.
\end{corollary}
}
\cb{We remind the reader that the above results are used to obtain the desired regularity and compactness of minimisers in~\cref{compactness} and the equicontinuity in time of solutions for the long-time behaviour result in~\cref{ltb}, although they are of independent interest by themselves. The proof of~\cref{holder,cor:holder} can be found in~\cref{sec:proofofholder}.} 
\section{Characterisation of stationary solutions and bifurcations}\label{bif}

Now that we have characterised the notion of solution for~\eqref{eq:agdif} we study the associated stationary problem which is given by
\begin{align}
\beta^{-1} \Delta \rho^m + \nabla \cdot (\rho \nabla W \star \rho)=0 \,,  \qquad x \in \Omega 
\label{eq:sagdif}
\end{align}
with the notion of solution identical to the one defined in~\cref{thm:weak1}. One can immediately see that 
$\rho_\infty$ (cf.~\eqref{flat}) is a solution to~\eqref{eq:sagdif} for all $\beta>0$. As mentioned earlier, \eqref{eq:agdif} and \eqref{eq:sagdif} are intimately associated to the free energy functional $\cF_\beta^m : \cP(\Omega) \to (-\infty,+\infty]$ which is defined as
\begin{align}
\cF_\beta^m (\rho):= 
\begin{cases}
\displaystyle \frac{\beta^{-1}}{m-1} \intom{\rho^m(x)}- \frac{\beta^{-1}}{m-1} + \frac{1}{2}\iintom{W(x-y) \rho(x) \rho(y)}, & m>1 \\&
\\
\displaystyle \beta^{-1} \intom{(\rho \log \rho)(x)} + \frac{1}{2}\iintom{W(x-y) \rho(x) \rho(y)}, & m=1
\end{cases}
\, ,
\end{align}
whenever the above quantities are finite and as $+\infty$ otherwise. We will often use the shorthand notation $S_\beta^m(\rho):=\frac{\beta^{-1}}{m-1} \intom{\rho^m(x)}- \frac{\beta^{-1}}{m-1}$ and $S_\beta:=\beta^{-1} \intom{(\rho \log \rho)(x)}$ for the entropies and $\cE(\rho):=\frac{1}{2}\iintom{W(x-y) \rho(x) \rho(y)}$
for the interaction energy. We will also drop the superscript $m$ and just use $\cF_\beta(\rho)$ whenever $m=1$. 

Another object that will play an important role in the analysis below is the following self-consistency equation
\begin{align}
\beta^{-1}\frac{m}{m-1}\rho^{m-1} +  W \star \rho =C  \, ,
\end{align}
for some constant $C>0$. We discuss how the above equation,  solutions of~\eqref{eq:sagdif}, and $\cF_\beta^m(\rho)$ are related to each other for the case $m>1$ in the following proposition. (the case $m=1$ is discussed in~\cite{CGPS19} and the proofs are essentially identical)
\begin{proposition}
Let $\rho \in \cP(\Omega) \cap \Leb^m(\Omega)$ and fix $\beta>0,m>1$. Then the following statements are equivalent
\begin{enumerate}
\item $\rho$ is a weak solution of ~\eqref{eq:sagdif}
\item $\rho$ is a critical point of $\cF_\beta^m$, i.e. the metric slope  $\abs{\partial \cF_\beta^m }(\rho)$ is $0$.
\item For every connected component $A$ of its support $\rho$ satisfies the self-consistency equation, i.e.
\begin{align}
\beta^{-1}\frac{m}{m-1}\rho^{m-1} + W \star \rho =C(A, \rho)  \label{eq:sc}
\end{align}
with $C(A,\rho)$ given by
\cre{\[
C(A,\rho)= \beta^{-1}\frac{m}{\abs{A}(m-1)}\norm{\rho}_{L^{m-1}(A)}^{m-1} + \frac{1}{\abs{A}}\int_{A} W \star \rho(x) \dx{x} \, .
\]
}
\end{enumerate}
\label{prop:tfae}
\end{proposition} 
\cre{
\begin{remark}\label{notnorm}
We have used the notation
\begin{align}
\norm{\rho}_{L^{m-1}(A)}= \bra*{\int_A \rho^{m-1}(x) \dx{x}}^{\frac{1}{m-1}} \, ,
\end{align}
for $1 < m < \infty$, even though this is not a norm for $1 < m <2$.
\end{remark}
\begin{remark}\label{mto1}
Note that if a stationary solution $\rho$ is fully supported then the constant $C(A,\rho)=C(\Omega,\rho)$ reduces to
\[
C(\Omega,\rho)= \beta^{-1}\frac{m}{\abs{\Omega}(m-1)}\norm{\rho}_{L^{m-1}(A)}^{m-1} \, ,
\]
where we have used the fact that $W$ has mean zero. We can now formally pass to the limit $m \to 1$ to obtain
\begin{align}
\beta^{-1} \log \rho + W \star \rho = \frac{\beta^{-1}}{\abs{\Omega}}\intom{\log \rho}\, .
\end{align}
The solutions of the above equation are studied in detail in~\cite{CGPS19}.
\end{remark}
}

Now that we have various equivalent characterisations of stationary solutions of~\eqref{eq:agdif}, we proceed to state and prove the main result of this section regarding the existence of bifurcations from the uniform state $\rho_\infty$ (cf.~\eqref{flat}). Before doing this however we need to introduce some relevant notions. We denote by $\SobH^n_0(\Omega)$ the homogeneous $\SobH^n(\Omega)$ space and by $\SobH^n_{0,s}(\Omega)$ the closed subspace of $\SobH^n_0(\Omega)$ consisting of functions which are even along every coordinate (pointwise a.e.). Note that the $\set{e_k}_{k \in \N^d, k \neq 0}$ form an orthogonal basis for $\SobH^n_{0,s}(\Omega)$. We then introduce the following map $F : \SobH^n_{0,s}(\Omega) \times \R_+ \to \SobH^n_{0,s}(\Omega)$ for $n>d/2$ which is given by
\begin{align}
F(\eta,\beta):= \beta^{-1} \frac{m}{m-1}(\rho_\infty+ \eta)^{m-1} + W \star \eta - \beta^{-1}\frac{m}{\abs{\Omega}(m-1)}\norm{\rho_\infty + \eta}_{L^{m-1}(\Omega)}^{m-1} \, \label{eq:scp}.
\end{align} 
Note that if $F(\eta,\beta)=0$ then the pair $(\rho_\infty+ \eta,\beta)$ satisfies~\eqref{eq:sc} on all of $\Omega$. If one can show that $(\rho_\infty + \eta)(x) \geq 0, \forall x \in \Omega$ then we have found a bonafide stationary solution of~\eqref{eq:agdif} by the equivalency established in~\cref{prop:tfae}. Thus, we would like to study the bifurcations of the map $F$ from its trivial branch $(0,\beta)$ . To this order we compute its Fr\'echet derivatives around $0$ as follows:
\begin{align}
D_{\eta}F(0,\beta)(e_1)=&\beta^{-1}m \rho_\infty^{m-2} e_1 +  W\star e_1 \\
 D^2_{\eta \beta}F(0,\beta)(e_1)=&-\beta^{-2}m \rho_\infty^{m-2} e_1 \\
  D^2_{\eta \eta}F(0,\beta)(e_1,e_2)=&\beta^{-1}m(m-2) \rho_\infty^{m-3} e_1 e_2 -\beta^{-1}\frac{m(m-2)}{\abs{\Omega}} \rho_\infty^{m-3} \intom{e_1 e_2 } \\
  D^3_{\eta \eta \eta}F(0,\beta)(e_1,e_2,e_3)=&\beta^{-1}m(m-2) (m-3)\rho_\infty^{m-4} e_1 e_2 e_3-\beta^{-1}\frac{m(m-2)(m-3)}{\abs{\Omega}} \rho_\infty^{m-4} \intom{e_1 e_2 e_3 } \, ,
\end{align}
for some $e_1,e_2,e_3 \in \SobH^n_{0,s}(\Omega)$.  We then have the following result:

\begin{theorem}[Existence of bifurcations]\label{thm:bif}
Consider the map $F : \SobH^n_{0,s}(\Omega) \times \R_+ \to \SobH^n_{0,s}(\Omega)$ for $n>d/2$ as defined in~\eqref{eq:scp} with its trivial branch $(0,\beta)$. Assume there exists $k^* \in \N^d, k^* \not \equiv0$ such that the following two conditions are satisfied
\begin{enumerate}
\item $\hat{W}(k^*) <0$
\item $\operatorname{card}\set*{k \in \N^d, k \not \equiv0: \frac{\hat{W}(k)}{\Theta(k)} = \frac{\hat{W}(k^*)}{\Theta(k^*)}}=1$ \, .
\end{enumerate}
Then, $(0,\beta_*)$ is a bifurcation point of $(0,\beta)$ with
\begin{align}
\beta_*=-\frac{m \rho_\infty^{m-3/2} \Theta(k^*)} { \hat{W}(k^*)} \, ,
\end{align}
i.e. there exists a neigbourhood $N$ of $(0,\beta_*)$ and a curve $(\eta(s),\beta(s)) \in N, s \in(-\delta,\delta), \delta>0$ such that $F(\eta(s),s)=0$. The branch $\eta(s)$ has the form
\begin{align}
\eta(s)= s e_{k^*} + r(s e_{k^*},\beta(s)) \, ,
\end{align}
where $\norm{r}_{\SobH^n_{0,s}(\Omega)}=o(s)$ as $s \to 0$. Additionally, we have that $\beta'(0)=0$ and 
\begin{align}
\beta''(0)= \frac{\beta_* (m-2)(m-3)}{3 \rho_\infty^2} \intom{e_{k^*}^4} \, .
\end{align}
\end{theorem}
\begin{proof}
The proof of this theorem relies on the Crandall--Rabinowitz theorem (cf.~\cref{thm:cr}). Note that $F \in C^2(\SobH^n_{0,s}(\Omega) \times \R_+; \SobH^n_{0,s}(\Omega))$. Thus, we need to show that: (a) $D_{\eta}F(0,\beta_*): \SobH^n_{0,s}(\Omega) \to \SobH^n_{0,s}(\Omega)$ is Fredholm with index zero and has a one-dimensional kernel and (b) for any $e \in \ker(D_{\eta}F(0,\beta_*)), e \neq 0$ it holds that $D^2_{\eta \beta}F(0,\beta_*)(e) \notin \Im(D_{\eta}F(0,\beta_*))$.

For (a) we first note that $D_{\eta}F(0,\beta_*)$ is a compact perturbation of the identity as the operator $W \star e$ is compact on $\SobH^n_{0,s}$.	It follows then that it is a Fredholm operator. Note that the functions $\set{e_k}_{k \in \N^d, k \neq 0}$ diagonalise the operator $D_{\eta}F(0,\beta_*)$. Indeed, we have
\begin{align}
D_{\eta}F(0,\beta_*)\bra{e_k}=& \bra*{\beta_*^{-1} m \rho_\infty^{m-2} + \frac{1}{N_k}\hat{W}(k)}e_k \\
=& \bra*{\beta_*^{-1} m \rho_\infty^{m-2} + \rho_\infty^{-1/2}\frac{\hat{W}(k)}{\Theta(k)}}e_k \, .
\end{align}
Note that if the conditions (1) and (2) in the statement of the theorem are satisfied it follows, using the expression for $\beta_*$, that $D_{\eta}F(0,\beta_*)\bra{e_k}=0$ if and  only if $k = k^*$. Thus, we have that $\ker(D_{\eta}F(0,\beta_*))= \textrm{span}\bra{e_{k^*}}$. This completes the verification of the condition (1) in~\cref{thm:cr}.

For condition (2) in~\cref{thm:cr}, we note again by the diagonalisation of $D_{\eta}F(0,\beta_*)$ that $\Im(D_{\eta}F(0,\beta_*))= \set*{\textrm{span}(e_{k^*})}^{\perp}$. Thus, we have that
\cb{
\begin{align}
D^2_{\eta \beta}F(0,\beta_*)(e_{k^*})=&-\beta_*^{-2}m \rho_\infty^{m-2} e_{k^*} \notin  \Im(D_{\eta}F(0,\beta_*)) \, .
\end{align}
}
We can now compute the derivatives of the branch. Using the identity~\cite[I.6.3]{kielhofer2006bifurcation}, it follows that $\beta'(0)=0$ if $D^2_{\eta\eta}F(0,\beta_*)(e_{k^*},e_{k^*}) \in \Ima \bra*{D_\eta F(0,\beta_*) }$. Thus, it is sufficient to check that
\cb{
\begin{align}
\skp*{D^2_{\eta\eta}F(0,\beta_*)(e_{k^*},e_{k^*}),e_{k^*}}=  \skp*{\beta_*^{-1}m(m-2) \rho_\infty^{m-3} e_{k^*}^2,e_{k^*} } =0 \, ,
\end{align}
}
where the last inequality follows by using the expression for $e_{k^*}^2$ from~\cref{trig} and orthogonality of the basis
$\set{e_k}_{k \in \N^d}$. Here $\skp{\cdot,\cdot}$ denotes the dual pairing in $\SobH^n_{0,s}$. Thus, we have that $\beta'(0)=0$. Finally we can compute $\beta''(0)$ by using~\cite[I.6.11]{kielhofer2006bifurcation} to obtain
\begin{align}
\beta''(0)=& -\frac{\skp*{D^3_{\eta \eta \eta}F(0,\beta_*)(e_{k^*},e_{k^*},e_{k^*}),e_{k^*}}}{3 \skp*{D^2_{\eta \beta}F(0,\beta_*)(e_{k^*}),e_{k^*}}} \\=&\frac{\beta_*^{-1}m(m-2)(m-3) \rho_\infty^{m-4} \intom{e_{k^*}^4}}{3\beta_*^{-2} m \rho_\infty^{m-2}} \\&
=\frac{\beta_* (m-2)(m-3)}{3 \rho_\infty^2} \intom{e_{k^*}^4} \, .
\end{align}
This completes the proof of the theorem.
\end{proof}

\begin{remark}
Since $\SobH^n_{0,s}(\Omega)$ is continuously embedded in $C^0(\Omega)$ it follows that for the branch of solutions $\rho_\infty + \eta(s)$ found in~\cref{thm:bif} are in fact strictly positive for $s$ sufficiently small and are thus stationary solutions by the result of~\cref{prop:tfae}. Any interaction potential $W(x)$ such that infinitely many $k$ satisfy the conditions of~\cref{thm:bif} will have infinitely many bifurcation points $(0,\beta_k)$ from the trivial branch. A typical example would a be a potential for which the map $k \mapsto \hat{W}(k)$ is strictly negative and injective.
\end{remark}

\begin{remark}\label{critical}
Note that $\beta''(0)>0$  for all $m \in (1,2) \cup (3,\infty)$. This means that the branch turns to the right, i.e. it is supercritical. On the other hand if $m \in (2,3)$, then $\beta''(0)<0$. This means that the branch turns to the left, i.e. it is subcritical. If $m \in \set{2,3}$ we have that $\beta''(0)=0$. The relation of this phenomenon to the minimisers of the free energy will be discussed in~\cref{special}.
\end{remark}

\section{Minimisers of the free energy and phase transitions}\label{transitions}
The nontrivial stationary solutions found as a result of the bifurcation analysis in the previous section need not correspond to minimisers of the free energy, $\cF_\beta^m(\rho)$. Indeed, we do not know yet if minimisers even exist. We start first by proving the existence of minimisers of $\cF_\beta^m$. We then show that for $\beta$ sufficiently small $\cF_\beta^m$ has a unique minimiser, namely $\rho_\infty$ (cf.~\eqref{flat}). 

The natural question to ask then is if this scenario changes for larger values of $\beta$. We provide a rigorous definition by which this change can be characterised via the notion of a transition point and define two possible kinds of transition points, continuous and discontinuous. We then provide necessary and sufficient conditions on $W$ for the existence of a transition point and sufficient conditions for the existence of continuous and discontinuous transition points. 

We start with a technical lemma that provides us with some useful a priori bounds on the minimisers of $\cF_\beta^m$.

\begin{lemma}[$\Leb^{\infty}(\Omega)$-bounds]
Assume $\beta>0,m>1$. Then there exists some $B_{\beta,m}>0$, such that if $\rho \in \cP(\Omega)$ with $\norm{\rho}_{\Leb^\infty(\Omega)}>B_{\beta,m}$, then there exists $\bar{\rho} \in \cP(\Omega)$ with $\norm{\bar{\rho}}_{\Leb^\infty(\Omega)}\leq B_{\beta,m}$ with
\[
\cF_{\beta}^m(\bar{\rho}) < \cF_{\beta}^m(\rho) \,.
\]
\label{lem:eqbound}
\end{lemma}

\begin{proof}
We start by noting that the following bounds hold
\begin{align}
S_\beta^m(\rho) \geq & \frac{\beta^{-1}}{m-1}\bra*{\frac{1}{\abs*{\Omega}}}^{m-1} - \frac{\beta^{-1}}{m-1} \label{eq:entlb}\\
\cE(\rho) \geq & -\frac{1}{2}\norm{W_{-}}_{\Leb^{\infty}(\Omega)} \, \label{eq:ielb}.
\end{align}
We divide our analysis into two cases. For $B>0$ and $\rho \in \cP(\Omega)$ let
\begin{align}
\B_{B}:= \set*{x \in \Omega: \rho  \geq B} \, ,
\end{align}
and 
\[
\varepsilon_B= \int_{\B_B} \rho \dx{x} \, .
\]

\paragraph{\bf Case 1:} $(\rho,B)$ s.t. $\varepsilon_B\geq \frac{1}{2}$

We then have the following bounds on the entropy.
\begin{align}
S_\beta^m(\rho) =& \frac{\beta^{-1}}{m-1}\bra*{\int_{\B_B} \rho^m \dx{x} + \int_{\B_B^c} \rho^m \dx{x}}  - \frac{\beta^{-1}}{m-1}\\
\geq & \frac{\beta^{-1} B^{m-1}}{2(m-1)}  - \frac{\beta^{-1}}{m-1} \, .
\end{align}
It follows then that we have the following bound on the free energy.
\begin{align}
\cF_{\beta}^m(\rho) \geq  \frac{\beta^{-1} B^{m-1}}{2(m-1)} -\frac{1}{2}\norm{W_{-}}_{\Leb^\infty(\Omega)}  - \frac{\beta^{-1}}{m-1}\, .
\end{align}
If we define a constant $B_1$ as follows
\begin{align}
B_1(m,\beta):= \bra*{\frac{2}{\abs{\Omega}^{m-1}} + \beta(m-1)  \norm{W_{-}}_{\Leb^\infty(\Omega)}}^{1/(m-1)} \, ,
\end{align}
such that for $B>B_1$, $1/\abs{\Omega}$ has a lower value of the free energy than $\rho$.
 
\paragraph{\bf Case 2:} $(\rho,B)$ s.t. $\varepsilon_B < \frac{1}{2}$

We write $\rho= \rho_B + \rho_r$, where $\rho_B:= \rho \cdot \chi_{\B_B}$
and $\rho_r:=\rho-\rho_B$. We then have the following bound on the entropy.
\begin{align}
S_\beta^m(\rho) \geq S_\beta^m(\rho_r) + \frac{\beta^{-1}B^{m-1}}{m-1}\varepsilon_B \geq S_\beta^m(\rho_r) \, .
\end{align}
We can assume without loss of generality that $\cF_\beta^m(\rho)< \cF_\beta^m(\rho_\infty)$, otherwise the proof is complete. It follows then that 
\[
\cE(\rho) < \cE(\rho_\infty)\,, \qquad S_\beta^m(\rho_r) + \frac{\beta^{-1}}{m-1}\leq S_\beta^m(\rho) + \frac{\beta^{-1}}{m-1} \leq
\frac{1}{2}\norm{W_-}_{\Leb^\infty(\Omega)} + \frac{\beta^{-1}}{(m-1)\abs{\Omega}^{m-1}} := s_*(m,\beta) \, .
\]

By expanding $\cE(\rho)$, the following estimate can be obtained
\begin{align}
\cE(\rho_r) < \cE(\rho_\infty) + \frac{1}{2}\norm{W_-}_{\Leb^\infty(\Omega)}:= e_* \,,
\end{align}
where we have used the fact that $\varepsilon_B <1/2$. Define $\bar{\rho_r} :=
(1-\varepsilon_B)^{-1} \rho_r \in \cP(\Omega)$. We have 
\begin{align}
S_\beta^m(\rho)-S_\beta^m(\bar{\rho_r}) \geq &  S_\beta^m(\rho_r) + \frac{\beta^{-1}B^{m-1}}{m-1}\varepsilon_B -\frac{\beta^{-1}}{m-1}(1- \varepsilon_B)^{-m}\intom{\rho_r^m} + \frac{\beta^{-1}}{m-1} \\
\geq&\varepsilon_B\pra*{\frac{\beta^{-1}B^{m-1}}{m-1}- \bra*{\frac{(1-\varepsilon_B)^{-m}-1}{\varepsilon_B}}s_*(m,\beta)} \,.
\end{align}
One can control the second term in the brackets as follows
\begin{align}
\bra*{\frac{(1-\varepsilon_B)^{-m}-1}{\varepsilon_B}}s_*(m,\beta)
\leq \max \bra*{m +\frac{m (m+1)(1-\delta)^{-m-2}\delta}{2}, \frac{2^m -1}{\delta} }s_*(m,\beta ) \, ,
\end{align}
for any $\delta< 1$. Setting $\delta=\frac{1}{2}$, we obtain
\begin{align}
\bra*{\frac{(1-\varepsilon_B)^{-m}-1}{\varepsilon_B}}s_*(m,\beta)
\leq m(1+(m+1)2^{m}) s_*(m,\beta) \,.
\end{align}
Similarly, for the interaction energy we can compute the difference as follows
\begin{align}
\cE(\rho)- \cE(\bar{\rho_r}) =& \cE(\rho)-\cE(\rho_r)+ \cE(\rho_r)-\cE(\bar{\rho_r}) \\
\geq & - \frac{1}{2}\norm{W_-}_{\Leb^\infty(\Omega)} \varepsilon_B + \cE(\rho_r)
\bra*{\frac{\varepsilon_B^2-2 \eps_B}{(1-\varepsilon_B)^2}} \\
\geq & \varepsilon_B \pra*{\bra*{\frac{\varepsilon_B-2}{(1-\varepsilon_B)^2}}\cE(\rho_r)-\frac{1}{2}\norm{W_-}_{\Leb^\infty(\Omega)}} \,.
\end{align}
Using the fact that $\varepsilon_B< 1/2$ we can obtain
\begin{align}
\cE(\rho)- \cE(\bar{\rho_r}) \geq &  \varepsilon_B \pra*{-8 e_*- \frac{1}{2}\norm{W_-}_{\Leb^\infty(\Omega)}} \, .
\end{align}
Now, we can define a second constant as follows
\cb{
\begin{align}
B_2(\beta,m):= \pra*{ (m-1)\beta\bra*{m\bra*{1+ 2^{m}(m+1)}s_*(m,\beta ) +8 e_*+ \frac{1}{2}\norm{W_-}_{\Leb^\infty(\Omega)}} }^{1/(m-1)}\,,
\end{align}
}
such that for $B>B_2$, $\bar{\rho_r}$ has a lower value of the free energy
than $\rho$.
We now set our constant as follows
\begin{align}
B_{\beta,m}:= \max \bra*{B_1(\beta,m),2B_2(\beta,m)} \, ,
\end{align}
and set $\bar{\rho}$ to either be $(1/\abs{\Omega})$ or $\bar{\rho_r}$. The constant $2$ in front of $B_2(\beta,m)$ follows from the fact that $\bar{\rho_r}$ has been normalised.
\end{proof}
The expression for the constant $B_{\beta,m}$ is explicit as a result of which we can even obtain some uniform control in $m$.
\begin{corollary}\label{cor:bbound}
Let $(\beta,m) \in (0,C)\times [1+ \eps,\infty)=:A \subset (0,\infty) \times
(1,\infty)$ for some $C,\eps>0$. Then $B^\star:=\sup_A B_{\beta,m} < \infty$.
\end{corollary}
We now proceed to the existence result for minimisers of $\cF_\beta^m$.
\begin{theorem}[Existence of minimisers]\label{thm:exm}
Fix $\beta>0$ and $m>1$, then $\cF_\beta^m: \cP(\Omega) \to (-\infty,+\infty]$ has a minimiser $\rho^* \in \cP(\Omega) \cap\Leb^\infty(\Omega)$. Additionally we have that
\[
\norm{\rho^*}_{\Leb^\infty(\Omega)} \leq B_{\beta,m} \, .
\]
\end{theorem}
\begin{proof}
We note first that, from~\eqref{eq:entlb} and~\eqref{eq:ielb}, $\cF_\beta^m$ is bounded below on $\cP(\Omega)$. Let $\set{\rho_n}_{n \in \N}$ be a minimising sequence. Note that by~\cref{lem:eqbound} we can pick this sequence such that $\norm{\rho_n}_{\Leb^\infty(\Omega)} \leq B_{\beta,m}$. By the Banach--Alaoglu theorem we have a subsequence $\set{\rho_{n_k}}_{k \in \N}$ and measure $\rho^* \in \Leb^\infty(\Omega)$ such that
\begin{align}
\rho_{n_k} \rightharpoonup \rho^* \textrm{ in weak-$*$ } \Leb^\infty(\Omega) \, .
\end{align}
Furthermore, we can find another subsequence (which we do not relabel), such that
\begin{align}
\rho_{n_k} \rightharpoonup \rho^* \textrm{ in weak } \Leb^2(\Omega) \, .
\end{align} 
Note that $\rho^*$ is nonnegative a.e. and also has mass one. Thus, $\rho^* \in \cP(\Omega) \cap \Leb^\infty(\Omega)$. The proof would be complete if we can show lower semicontinuity of $\cF_\beta^m$ in weak $\Leb^2(\Omega)$. Note that for $W \in C^2(\Omega)$, $\cE(\rho)$ is continuous. On the other hand, $S_\beta^m(\rho)$ is convex and lower semicontinuous in the $\Leb^2(\Omega)$ topology. It follows from fairly classical results (cf.~\cite[Theorem 3.7]{Bre11}) that $\cF_\beta^m$ is also weakly lower semicontinuous. This concludes the proof of existence of minimisers. The bound simply follows from the fact that norms are lower semicontinuous under weak-$*$ convergence.
\end{proof}
\begin{lemma}[Regularity and compactness of minimisers]\label{compactness}
Let $\rho_\beta \in \cP(\Omega)$ be a minimiser of $\cF_\beta^m(\rho)$. Then $\rho_\beta$ is H\"older continuous with exponent $a \in (0,1)$ given by~\cref{holder}, where $a$ depends continuously on $\beta$. Let $\set{\rho_\beta}_{\beta \in I}$ be a family of such minimisers, where $I \subset \R_+$ is some bounded interval. Then the family $\set{\rho_\beta}_{\beta \in I}$  is relatively compact in $C^0(\Omega)$.
\end{lemma}
\begin{proof}
The proof of the first statement follows simply by applying~\cref{prop:tfae} and~\cref{holder} with $M=B_{\beta,m}$. For the second statement,  let $\bar{I}$ be the closure of $I$. Then applying~\eqref{holdernorm} for some $x,y \in \T^d$, we have that
\begin{align}
\abs{\rho_\beta(x)-\rho_\beta(y)} \leq C_h d_{\T^d}(x,y)^a \, ,
\end{align}
where $a=a(\beta),C_h=C_h(\beta)$. Setting $a^\star= \max_{\bar{I}}a(\beta)$ and $B^\star$ to be as in~\cref{cor:bbound}, we have
that
\begin{align}
\abs{\rho_\beta(x)-\rho_\beta(y)} \leq C_h^\star d_{\T^d}(x,y)^{a^\star} \, ,
\end{align}
where $C_h^\star$ is some new constant depending on $B^\star$, $m$, $d$, and $W$. Thus, the family $\set{\rho_\beta}_{\beta \in I}$ is equicontinuous. It is clearly equibounded from~\cref{cor:bbound}. Applying the Arzel\`a--Ascoli theorem, the result follows.
\end{proof}
Now that we have shown existence and regularity of minimisers we show that for $\beta$ small or $W \in \HH_s$ minimisers of $\cF_\beta^m$ are unique and given by $\rho_\infty$. To show this we start with the following lemma which shows positivity of stationary solutions for $\beta$ sufficiently small.
\cb{\begin{lemma}\label{lem:posest}
There exists an $\delta>0$ depending on $m$ and $W$, such that for all $\beta<\delta$ it holds that if $\rho \in \cP(\Omega) \cap \Leb^m(\Omega)$ is a stationary solution of~\eqref{eq:agdif}, then $\rho(x) \geq \frac{1}{2 \abs{\Omega}}$ for all $x \in \Omega$.
\end{lemma}
\begin{proof}
Note that if $\rho \in \cP(\Omega) \cap \Leb^m(\Omega)$ is stationary, then, by~\cref{prop:tfae}, it satisfies on each connected component $A$ of its support
\begin{align}
\beta^{-1}\frac{m}{m-1}\rho^{m-1} + W \star \rho =C(A, \rho)  
\end{align}
with $C(A,\rho)$ given by
\[
C(A,\rho)= \beta^{-1}\frac{m}{\abs{A}(m-1)}\norm{\rho}_{L^{m-1}(A)}^{m-1} + \frac{1}{\abs{A}}\int_{A} W \star \rho(x) \dx{x} \, .
\]
Thus, we have that $\rho \in \Leb^\infty(\Omega)$. Using a mollification argument and \eqref{eq:sc}, one can then obtain the following bound
\begin{align}
\norm{\nabla \rho^{m-1}}_{\Leb^\infty(\Omega)} \leq\beta \frac{m-1}{m} \norm{\nabla W \star \rho}_{\Leb^\infty(\Omega)} \leq  \beta \frac{m-1}{m} \norm{\nabla W}_{\Leb^\infty(\Omega)} \, .
\end{align}
By~\cref{holder}, it follows that $\rho$ is $a$-H\"older continuous. Note further that we have that 
\begin{align}
\max_{x \in \Omega} \rho(x) \geq \abs{\Omega}^{-1} \, ,
\end{align}
Thus, we can choose $\beta$ to be small enough, dependent on $m$ and $W$, and apply the bound to argue that
\begin{align}
\min_{x \in \Omega}\rho^{m-1} \geq 2^{1-m}\abs{\Omega}^{1-m} \, .
\end{align}
Thus, the result follows.
\end{proof}
}

\cb{
We can now use the positivity estimate of~\cref{lem:posest} to prove that for $\beta$ sufficiently small stationary solutions of~\eqref{eq:agdif} (and thus minimisers of $\cF_\beta^m$) are unique. This improves the result of~\cite{CKY13}, in which uniqueness is proved only for $1 < m \leq 2$.
\begin{lemma}\label{ssunique}
For $\beta \ll 1$ and $m \in (1,\infty)$, $\rho_\infty$ is unique stationary solution of~\eqref{eq:agdif} and minimiser of the free energy,~$\cF_\beta^m$.
\end{lemma} 
\begin{proof}
Assume $\rho \in \cP(\Omega) \cap \Leb^m(\Omega)$ is a stationary solution of~\eqref{eq:agdif}. Then, we can apply the same argument as in the proof of~\cref{lem:posest} to obtain
\begin{align}
\norm{\nabla \rho^{m-1}}_{\Leb^\infty(\Omega)}\leq \beta \frac{m-1}{m} \norm{\nabla W \star \rho}_{\Leb^\infty(\Omega)} \leq\beta \frac{m-1}{m} \norm{ W}_{\Leb^1(\Omega)}\norm{ \nabla \rho}_{\Leb^\infty(\Omega)}  \, .
\end{align}
It follows that
\begin{align}
\norm{(m-1)\rho^{m-2}\nabla \rho}_{\Leb^\infty(\Omega)} \leq\beta \frac{m-1}{m} \norm{ W}_{\Leb^1(\Omega)}\norm{ \nabla \rho}_{\Leb^\infty(\Omega)} \, .
\label{eq:gradgrad}
\end{align} 
Let us now assume that $\beta<\delta$, where $\delta$ is the constant from the statement of~\cref{lem:posest}. Furthermore, if $1<m<2$ the constant $C(\Omega,\rho)$ in~\cref{prop:tfae} can be controlled as follows
\begin{align}
C(\Omega,\rho) \leq \beta^{-1}\frac{m}{\abs{\Omega}(m-1)}\intom{\rho^{m-1}}  \leq \beta^{-1}\frac{m}{\abs{\Omega}(m-1)} \, ,
\end{align}
where in the last step we have applied Jensen's inequality. Thus, we have 
\begin{align}
\abs{\rho(x)} \leq \bra*{ \beta \frac{m-1}{m} \norm{W}_{\Leb^\infty{(\Omega)}} + \frac{1}{\abs{\Omega}} }^{1/(m-1)} \,
\end{align}
for all $x \in \Omega$. Thus, for $1< m< 2$, we can apply the above bound to~\eqref{eq:gradgrad} to obtain
\begin{align}
\norm{\nabla \rho}_{\Leb^\infty(\Omega)} \leq \frac{\beta}{m} \bra*{ \beta \frac{m-1}{m} \norm{W}_{\Leb^\infty{(\Omega)}} + \frac{1}{\abs{\Omega}} }^{(2-m)/(m-1)} \norm{ W}_{\Leb^1(\Omega)}\norm{ \nabla \rho}_{\Leb^\infty(\Omega)} \, .
\end{align}
If $\beta$ is sufficiently small, we have that $\norm{\nabla \rho}_{\Leb^\infty(\Omega)}=0$. Thus, $\rho=\rho_\infty$ for $\beta$ sufficiently small. Similarly for $2 \leq m < \infty$, we can apply the bound from~\cref{lem:posest} to obtain
\begin{align}
\norm{\nabla \rho}_{\Leb^\infty(\Omega)}\leq \frac{\beta}{m} 2^{2-m}\abs{\Omega}^{2-m} \norm{ W}_{\Leb^1(\Omega)}\norm{ \nabla \rho}_{\Leb^\infty(\Omega)} \, .
\end{align}
Applying a similar argument as before, we have that, for $\beta \ll 1$, $\rho=\rho_\infty$. Thus, for $\beta \ll 1$, $\rho_\infty$ is the unique stationary solution of~\eqref{eq:agdif} and, by~\cref{prop:tfae}, the unique minimiser of $\cF_\beta^m$.
\end{proof}
}
We also have the following result on uniqueness of minimisers when $W \in \HH_s$. 
\begin{theorem}\label{minunique}
Let $W \in \HH_s$ and $m \in (1,\infty)$. Then $\cF_\beta^m(\rho)$ has a unique minimiser $\rho=\rho_\infty$. 
\end{theorem}
\begin{proof}
We first consider the case in which $W \in \HH_s$.  We write the linear interpolant as $\rho_t= \rho_0 + t \eta$ where $\eta= \rho_1-\rho_0$ where $\rho_0,\rho_1 \in \cP(\Omega)$ with $\cF_\beta^m(\rho_0),\cF_\beta^m(\rho_1)<\infty$. Differentiating with respect to $t$ twice we obtain that
\begin{align}
\frac{\dx^2}{\dx{t}^2}\cF_\beta^m(\rho_t) =\beta^{-1} \intom{m \rho_t^{m-2} \eta^2} +\iintom{W(x-y) \eta(x) \eta(y)} \, .
\end{align}
For $W \in \HH_s$ the above expression is strictly positive. Thus, $\cF_\beta^m(\rho_t)$ is a convex function, from which it follows that $\cF_\beta^m$ must have unique minimisers. We further argue that the minimiser must be $\rho_\infty$. Indeed, we have for any $\cP(\Omega) \ni \rho \neq \rho_\infty$ that
\begin{align}
\cF_\beta^m(\rho) =& S_\beta^m (\rho) + \cE(\rho) \\ >& S_\beta^m (\rho_\infty) + \cE(\rho) \\\geq & S_\beta^m(\rho_\infty)= \cF_\beta^m(\rho_\infty) \, ,   
\end{align}
where the first inequality follows from Jensen's inequality and the second one from the fact that $W \in \HH_s$ and~\cref{def:Hstab}.
\end{proof}

We know now from~\cref{ssunique}, that for $\beta \ll 1$, $\rho_\infty$ is the unique minimiser of $\cF_\beta^m$ and stationary solution of~\eqref{eq:agdif}. \cb{We now present the following result on the long-time behaviour of~\eqref{eq:agdif} in this regime:
\begin{theorem}[Long-time behaviour]\label{ltb}
Let $\rho$ be a weak solution of~\eqref{eq:agdif} with initial datum $\rho_0 \in \Leb^\infty(\Omega) \cap \cP(\Omega)$. Assume that $\beta$ and $W$ are such that $\rho_\infty$ is the unique stationary solution of~\eqref{eq:agdif} (and, therefore, the unique minimiser of $\cF_\beta^m$). Then, it holds that
\begin{align}
\lim_{t \to \infty} \norm{\rho(\cdot,t)-\rho_\infty}_{\Leb^\infty(\Omega)}=0 \, .
\end{align}  
\end{theorem}
\begin{proof}
We start by showing that if $\rho_0 \in \Leb^\infty(\Omega) \in \cP(\Omega)$, then $\norm{\rho}_{\Leb^\infty(\Omega_\infty)} \leq M < \infty$. We choose as a test function in the weak formulation, $\phi=p \rho^{p-1}$, for some $p>1$. Note that we can justify this choice by mollifying $\phi$ and then passing to the limit. We then obtain from~\eqref{eq:weakform} the following expression
\begin{align}
\left.\intom{\rho^p}\right\lvert^T_0 + \int_{0}^T \intom{\bra*{\beta^{-1}m\rho^{m-1}\nabla \rho(x,t) \cdot \nabla \phi(x,t)+ \rho(x,t) \nabla( W \star \rho)(x,t) \cdot \nabla \phi(x,t) }}\dx{t} =0 \, .
\end{align}
Plugging in the value of $\phi$ on the right hand side and integrating by parts, we obtain
\begin{align}
\norm{\rho(\cdot,T)}_{\Leb^p(\Omega)}^p =& \norm{\rho_0}_{\Leb^p(\Omega)}^p \\& + \int_0^T \bra*{-\frac{4\beta^{-1} pm (p-1)}{(m+p-1)^2}\intom{\abs*{\nabla \rho(x,t)^{\frac{m+p-1}{2}} }^2}} \dx{t} \\&+\int_0^T \bra*{(p-1)\intom{\bra{\Delta W \star\rho(x,t)}\rho(x,t)^p}  }\dx{t}
\end{align}
Applying the Lebesgue differentiation theorem, we obtain that for $t$ a.e., it holds that
\begin{align}
\frac{\dx{}}{\dx{t}}\norm{\rho(\cdot,t)}_{\Leb^p(\Omega)}^p =& -\frac{4\beta^{-1} pm (p-1)}{(m+p-1)^2}\intom{\abs*{\nabla \rho(x,t)^{\frac{m+p-1}{2}} }^2} + (p-1)\intom{\bra{\Delta W \star\rho(x,t)}\rho(x,t)^p}  \\
\leq&  -\frac{4\beta^{-1} pm (p-1)}{(m+p-1)^2}\intom{\abs*{\nabla \rho(x,t)^{\frac{m+p-1}{2}} }^2} + (p-1)\norm{\Delta W }_{\Leb^\infty{\Omega}}\norm{\rho(\cdot,t)}_{\Leb^p(\Omega)}^p \, .
\label{pregronwall}
\end{align}
Note that we can control the second term on the right hand side of the above expression as follows
\begin{align}
\norm{\rho(\cdot,t)}_{\Leb^p(\Omega)}^p \leq& \norm{\rho(\cdot,t)}_{\Leb^1(\Omega)}^{p \theta} \norm{\rho(\cdot,t)}_{\Leb^{\frac{(m+p-1)d}{d-2}}(\Omega)}^{p (1-\theta)} \\
=& \norm{\rho(\cdot,t)}_{\Leb^{\frac{(m+p-1)d}{d-2}}(\Omega)}^{p (1-\theta)} \, ,
\end{align}
where we have used the fact that $1<p <\frac{ (m+p-1) d}{d-2}$ and the constant $\theta \in (0,1)$ is given by
\begin{align}
\theta = \frac{(m-1)d+2p}{((m+p-2)d+2)} \, .
\end{align}
 We now apply the Sobolev inequality on the torus, to obtain
\begin{align}
\norm{\rho(\cdot,t)}_{\Leb^p(\Omega)}^p &\leq \norm{\rho(\cdot,t)-1 +1 }_{\Leb^{\frac{(m+p-1)d}{d-2}}(\Omega)}^{p (1-\theta)} \\
&\leq 2^{p(1-\theta)-1} \bra*{ \norm*{\rho(\cdot,t)-1 }_{\Leb^{\frac{(m+p-1)d}{d-2}}(\Omega)}^{p (1-\theta)} +1   } \\
& \leq 2^{p(1-\theta)-1} \bra*{ \bra*{C_d\norm*{\nabla \rho(\cdot,t)^{\frac{m+p-1}{2}} }_{\Leb^{2}(\Omega)}}^{\frac{2p (1-\theta)}{m+p-1}} +1 } \\
&= \frac{1}{2}\bra*{\bra*{2^{\frac{m+p-1}{2}}C_d \norm*{\nabla \rho(\cdot,t)^{\frac{m+p-1}{2}} }_{\Leb^{2}(\Omega)} }^{\frac{2p (1-\theta)}{m+p-1}} + 2^{p(1-\theta)}} \, .
\end{align}
Note that the constant $C_d$ in the above estimate depends only on dimension and is independent of $p>1$. We set $q_1:=(m+p-1)/(p(1-\theta))$ and $q_2:=q_1/(q_1-1)$. Note that from the definition of $\theta$ we have 
\begin{align}
q_1=\frac{m+p-1}{p(1-\theta)}=\frac{(m+p-2)d +2}{d(p-1)} >1 \, .
\end{align}
Thus, we have that
\begin{align}
q_2= \frac{q_1}{q_1-1} =\frac{(m+p-2)d +2}{(m-1)d+2} \, .
\end{align}
We can thus apply Young's inequality with $q_1,q_2$ to obtain
\begin{align}
\norm{\rho(\cdot,t)}_{\Leb^p(\Omega)}^p &\leq \frac{1}{2}\bra*{ C_{p,m,\beta}\norm{\nabla \rho(\cdot,t)^{\frac{m+p-1}{2}}}_{\Leb^2(\Omega)}^2 +\frac{ 2^{p(1-\theta)q_2}C_d^{2q_1^{-1}q_2}}{C_{p,m,\beta} q_2 q_1}+ 2^{p(1-\theta)}}
\label{youngs}
\end{align}
where $C_{p,m,\beta}>0$ is given by
\begin{align}
 C_{p,m,\beta} :=\frac{4 \beta^{-1}pm(p-1)}{(m+p-1)^2 \norm{\Delta W}_{\Leb^\infty(\Omega)}(p-1)} \,.
\end{align}
Multiplying through by $\norm{\Delta W}_{\Leb^\infty(\Omega)}(p-1)$, we can apply the estimate in~\eqref{youngs} to~\eqref{pregronwall} to obtain
\begin{align}
\frac{\dx{}}{\dx{t}}\norm{\rho(\cdot,t)}_{\Leb^p(\Omega)}^p \leq&  - (p-1)\norm{\Delta W }_{\Leb^\infty\bra{\Omega}}\norm{\rho(\cdot,t)}_{\Leb^p(\Omega)}^p  \\&+  \norm{\Delta W}_{\Leb^\infty(\Omega)}(p-1) \bra*{\frac{ 2^{p(1-\theta)q_2}C_d^{2q_1^{-1}q_2}}{C_{p,m,\beta} q_2 q_1}+ 2^{p(1-\theta)}}  \, .
\end{align}
Applying Gr\"onwall's inequality, we obtain that
\begin{align}
\norm{\rho(\cdot,t)}_{\Leb^p(\Omega)}^p \leq e^{-(p-1)\norm{\Delta W }_{\Leb^\infty\bra{\Omega}}t}\norm{\rho_0}_{\Leb^p(\Omega)}^p + \bra*{\frac{ 2^{p(1-\theta)q_2}C_d^{2q_1^{-1}q_2}}{C_{p,m,\beta} q_2 q_1}+ 2^{p(1-\theta)}} \, .
\end{align}
It follows that
\begin{align}
\norm{\rho(\cdot,t)}_{\Leb^p(\Omega)} \leq& \bra*{e^{-(p-1)\norm{\Delta W }_{\Leb^\infty\bra{\Omega}}t}\norm{\rho_0}_{\Leb^p(\Omega)}^p + \bra*{\frac{ 2^{p(1-\theta)q_2}C_d^{2q_1^{-1}q_2}}{C_{p,m,\beta} q_2 q_1}+ 2^{p(1-\theta)}} }^{1/p} \\
\leq& 3^{1/p} \max \set*{\norm{\rho_0}_{\Leb^\infty(\Omega)},\frac{ 2^{(1-\theta)q_2}C_d^{2p^{-1}q_1^{-1}q_2}}{C_{p,m,\beta}^{1/p} q_2^{1/p} q_1^{1/p}}, 2^{(1-\theta)} } \, .
\end{align}
Note now that
\begin{align}
\frac{ 2^{(1-\theta)q_2}C_d^{2p^{-1}q_1^{-1}q_2}}{C_{p,m,\beta}^{1/p} q_2^{1/p} q_1^{1/p}} \lesssim 1 
\end{align}
as $p \to \infty$. It follows then that we can find a constant $M$ dependent on $\norm{\rho_0}_{\Leb^\infty(\Omega)}$, $d$, $\beta$, and $m$ but independent of $t$ and $p$ such that
\begin{align}
\norm{\rho(\cdot,t)}_{\Leb^p(\Omega)} \leq M \, ,
\end{align}
for all $t \in [0,\infty)$. Passing to the limit as $p \to \infty$, it follows that
\begin{align}
\norm{\rho}_{\Leb^\infty(\Omega_\infty)} \leq M \, ,
\label{eq:ubounded}
\end{align}
for all $t \in [0,\infty)$. We can now apply~\cref{holder} to argue that the solution $\rho(x,t)$ is H\"older continuous with some exponent $a \in(0,1)$. Furthermore, we can apply~\cref{cor:holder}, to argue that 
\begin{align}
\abs{\rho(y,t_1)-\rho(x,t_2)} \leq C_{h} \bra*{d_{\T^d}(x,y) + \abs{t_1-t_2}^{1/2}}^a \, ,
\label{eq:equicont}
\end{align}
for all $x,y \in \T^d$ and $0<C<t_1<t_2<\infty$. Consider now the solution semigroup $S_t: Z_E \to Z_E, t \geq 0$ associated to the evolution in~\eqref{eq:agdif}, where $$
Z_E= \set*{\rho \in \cP(\Omega): \cF_{\beta}^m(\rho) \leq E} \, ,
$$
for some $E \in \R$. We make $Z_E$ into a complete metric space by equipping it with the $d_2(\cdot,\cdot)$ Wasserstein distance. The fact that it is complete follows from the fact that $\cF_\beta^m$ is lower semicontinuous with respect to convergence in $d_2(\cdot,\cdot)$.  Note that the family of mappings $\set{S_t}_{t \geq 0}$ forms a metric dynamical system in the sense of~\cite[Definition 9.1.1]{CH98}. This follows from the fact (cf.~\cite[Theorem 11.2.8]{AGS}) the evolution defines a gradient flow $\rho \in C([0,\infty);Z_{E_0})$ in $\cP(\Omega)$ in the sense of~\cite[Definition 11.1.1]{AGS} where $E_0= \cF_\beta^m(\rho_0)$.  We now define the $\omega$-limit set associated to the initial datum $\rho_0 \in \Leb^\infty(\Omega) \cap \cP(\Omega)$, as follows
\begin{align}
\omega(\rho_0):= \set*{\rho_* \in Z_{E_0}: \lim_{n \to \infty}d_2\bra{S_{t_n}(\rho_0), \rho_*} = 0, t_n \to \infty} \, .
\end{align}
Since the metric space $Z_{E_0}$ is compact, it follows that the set $\bigcup_{t \geq 0}S_t(\rho_0)$ is relatively compact in $Z_{E_0}$. Applying~\cite[Theorem 9.1.8]{CH98}, we have that $\omega(\rho_0) \neq \emptyset$ and
$$
\lim_{t \to \infty} d_2(\rho(\cdot,t),\omega(\rho_0))=\lim_{t \to \infty} d_2(S_t(\rho_0),\omega(\rho_0))=0 \, ,
$$
where $\rho(\cdot,t)$ is the unique solution of~\eqref{eq:agdif} with initial datum $\rho_0 \in \cP(\Omega) \cap \Leb^\infty(\Omega)$. We now need to show that $\omega(\rho_0)$ is contained in the set of stationary solutions of~\eqref{eq:agdif}. Assume $\rho_* \in \omega(\rho_0)$, then there exists a time-diverging sequence $t_n \to \infty$ such that 
$$
\lim_{n \to \infty}d_2(\rho(\cdot,t_n),\rho_*)= \lim_{n \to \infty}d_2(S_{t_n}(\rho_0),\rho_*)=0 \, .
$$
Since the solution $\rho(\cdot,t)$ is gradient flow of the free energy $\cF_\beta^m$  with respect to the $d_2(\cdot,\cdot)$ distance on $\cP(\Omega)$, it follows that the following energy-dissipation equality holds true for all $t \in [0,\infty)$ (cf.~\cite[Theorem 11.1.3]{AGS})
\begin{align}
\cF_\beta^m(\rho_0)- \cF_\beta^m(\rho(\cdot,t)) = \int_0^t \abs{\partial \cF_\beta^m}^2(\rho(\cdot,s)) \dx{s} \, ,
\label{eq:EDI}
\end{align}
where $\abs{\partial \cF_\beta^m}: \cP(\Omega) \to (-\infty,+\infty]$ is the metric slope of $\cF_{\beta}^m$ and is given by
\begin{align}
\abs{\partial \cF_\beta^m}(\rho): = \bra*{\intom{\abs*{\beta^{-1} \frac{\nabla \rho^m}{\rho} + W \star \rho}^2 \rho}}^{1/2} \, .
\end{align}
Bounding the energy from below and then passing to the limit as $t \to \infty$ in~\eqref{eq:EDI}, we obtain
\begin{align}
\int_0^\infty \abs{\partial \cF_\beta^m}^2(\rho(\cdot,s)) \dx{s} \leq -\min_{\rho \in \cP(\Omega)}\cF_\beta^m(\rho) + \cF_\beta^m(\rho_0) \leq C\, . 
\label{eq:dissbounded}
\end{align}
We now consider the time-diverging sequence $t_n \to \infty$ and the sequence of curves $\set{\rho_n}_{n \in \N} \in C([0,1];Z_{E_0})$ with $\rho_n(\cdot,t)=\rho(\cdot,t_n+t)$. For each $n \in \N$, we have that
\begin{align}
d_2(\rho_n(\cdot,t_1),\rho_n(\cdot,t_2)) \leq& \frac{L}{\sqrt{2}} \norm{\rho_n(\cdot,t_1)-\rho_n(\cdot,t_2)}_{\Leb^1(\Omega)}^{1/2}\\ \leq& \frac{L}{\sqrt{2}} \norm{\rho_n(\cdot,t_1)-\rho_n(\cdot,t_2)}_{\Leb^\infty(\Omega)}^{1/2} \leq C_h^{1/2} \frac{L}{\sqrt{2}} \abs{t_1 -t_2}^{a/4} \, ,
\end{align}
for all $t_1,t_2 \in [0,1]$, where in the last step we have used~\eqref{eq:equicont}. We can thus apply the generalised Arzel\'a--Ascoli/Aubin--Lions compactness theorem (cf.~\cite[Proposition 3.3.1]{AGS}) to argue that there exists a curve $\mu \in C([0,1];Z_{E_0})$ such that
$\rho_n(\cdot,t)$ converges to $\mu(\cdot,t)$, in the sense of weak convergence of  probability measures, for all $t \in [0,1]$. Furthermore, from the lower semicontinuity of $\abs{\partial \cF_\beta^m}$ (cf.~\cite[Theorem 5.4.4]{AGS}) and Fatou's lemma, we have that
\begin{align}
\int_0^1 \abs{\partial \cF_\beta^m}^2(\mu(\cdot,s)) \dx{s} \leq& \liminf_{n \to \infty} \int_0^1 \abs{\partial \cF_\beta^m}^2(\rho_n(\cdot,s)) \dx{s} \\
= & \liminf_{n \to \infty} \int_{t_n}^{t_n+1} \abs{\partial \cF_\beta^m}^2(\rho(\cdot,s)) \dx{s} =0 \, ,
\end{align}
where in the last step we have used~\eqref{eq:dissbounded}. It follows that $\abs{\partial \cF_\beta^m}(\mu(\cdot,t))=0$ for $t$ a.e.  Thus, since $\mu$ is continuous, we can find a sequence of times $m \in \N$, $t_m \to 0$, such that $\abs{\partial \cF_\beta^m}(\mu(\cdot,t_m))=0$ and $d_2(\mu(\cdot,t_m),\mu(\cdot,0)) \to 0$ as $m \to \infty$. Note further that $\mu(\cdot,0)= \lim_{n\to \infty} \rho(\cdot,t_n)=\rho_*$. From the lower semicontinuity of $\abs{\partial \cF_\beta^m}(\cdot)$ we have that
\begin{align}
\abs{\partial \cF_\beta^m}(\rho_*)=\abs{\partial \cF_\beta^m}(\mu(\cdot,0))=0 \, .
\end{align}
Applying~\cref{prop:tfae}, it follows that $\rho_* \in Z_{E_0} \subset \cP(\Omega) \cap \Leb^m(\Omega)$ is necessarily a stationary solution of~\eqref{eq:agdif}. Since $\rho_\infty$ is the unique stationary solution, it follows that 
\begin{align}
\lim_{t \to \infty}d_2(\rho(\cdot,t), \rho_\infty)=0 \, .
\label{eq:d2convergence}
\end{align}
However, from~\eqref{eq:ubounded} and~\eqref{eq:equicont}, we know that, for any time-diverging sequence $t_n \to \infty$, $\set{\rho(\cdot,t_n)}_{n \in \N}$ has a convergent subsequence in $\Leb^\infty(\Omega)$, whose limit must be $\rho_\infty$ by~\eqref{eq:d2convergence}. Since the limit is unique, it follows that
\begin{align}
\lim_{t \to \infty}\norm{\rho(\cdot,t)-\rho_\infty}_{\Leb^\infty(\Omega)}=0.
\end{align}
\end{proof}
\begin{remark}
We remark that the technique used in the proof of~\cref{ltb} can be adapted to study the asymptotic properties of general gradient flows in the space of probability measures. These ideas will be expanded upon in a future work (cf.~\cite{CGW20}).
\end{remark}
}

From~\cref{minunique}, it is  also immediately clear that $W \in \HH_s^c$ is a necessary condition for the existence of a nontrivial minimiser at higher values of the parameter $\beta$. \cre{Indeed,~\cref{minunique} tells us that if $W \in \HH_s$ then minimisers of $\cF_\beta^m$ are unique and are given by $\rho_\infty$.} Before we discuss this any further, we introduce a notion of transition point that allows us to capture a change in the set of minimisers. 

 \begin{defn}[Transition point]\label{def:tp}
 A parameter value $\beta_c>0$ is said to be a transition point of $\cF_\beta^m$ if the following conditions are satisfied.
\begin{enumerate}
\item For $\beta<\beta_c$, $\rho_\infty$ is the unique minimiser of $\cF_\beta^m$.
\item At $ \beta=\beta_c$, $\rho_\infty$ is a minimiser of $\cF_\beta^m$.
\item For $\beta>\beta_c$, there exists $\cP(\Omega) \ni \rho_\beta\neq \rho_\infty$, such that $\rho_\beta$ is a minimiser of $\cF_\beta^m$.
\end{enumerate}
\end{defn}
We further classify transition points into discontinuous and continuous transition points.
\begin{defn}[Continuous and discontinuous transition points]\label{def:cdtp}
A transition point $\beta_c$  of $\cF_\beta^m$ is said to be a continuous transition point if
\begin{enumerate}
\item At $ \beta=\beta_c$, $\rho_\infty$ is the unique minimiser of $\cF_\beta^m$.
\item For any family of minimisers $\set{\rho_\beta}_{\beta>\beta_c}$ it holds that
\[
\limsup_{\beta \to \beta_c^+} \norm*{\rho_\beta- \rho_\infty}_{\Leb^\infty(\Omega)} =0 \,.
\]
\end{enumerate}
A transition point $\beta_c>0$ of $\cF_\beta^m$ which is not continuous is said to be discontinuous.
\end{defn}
It turns out that $W \in \HH_s^c$ is in fact a sufficient condition for the existence of a transition point. This result is analogous to the result in case $m=1$ discussed in~\cite{GP70,CP10,CGPS19}.
\begin{proposition}\label{prop:tp}
Assume $W \in \HH_s^c$. Then there exists some parameter value $0< \beta_c \leq \beta_\sharp^m$  with $\beta_\sharp^m$ defined as 
\[
\beta_\sharp^m:= -\frac{m \rho_\infty^{m-3/2} } { \min_{k \in \N^d, k \neq 0}\dfrac{\hat{W}(k)}{\Theta(k)}} \, ,
\]
such that $\beta_c$ is a transition point of $\cF_\beta^m$. Thus, $W \in \HH_s^c$ is a necessary and sufficient condition for the existence of a transition point.
\end{proposition}
\begin{proof}
Consider the measure $\rho^\eps=\rho_\infty + \eps e_{k^\sharp} \in \cP(\Omega)$ for $0 < \eps \ll 1$ where $k^\sharp \in \N^d$ is defined as 
\begin{align}
k^\sharp:= \arg\min\limits_{k \in \N^d, k \neq 0}\frac{\hat{W}(k)}{\Theta(k)} \, . 
\end{align}
if it is defined uniquely. If not we pick any  $k^\sharp$ that realises the minimum of the above expression.  We now consider an expansion of the energy $\cF_\beta^m(\rho^\eps)$ around $\rho^\eps$ which we will use repeatedly throughout the rest of this section. We Taylor expand around $\rho_\infty$ to obtain
 \begin{align}
\cF_\beta^m(\rho^\eps)=& \cF_\beta^m(\rho_\infty) + \bra*{\beta^{-1}m \rho_\infty^{m-2} + \rho_\infty^{-1/2}\frac{\hat{W}(k^\sharp)}{\Theta(k^\sharp)} }\frac{\eps^2}{2}\norm{e_{k^\sharp}}_{\Leb^2(\Omega)}^2 \\
&+ \beta^{-1}m(m-2)\frac{\eps^{3}}{6}\intom{f^{m-3} e_{k^\sharp}^3} , 
 \end{align} 
 where the function $f(x) \in \bra{\rho_\infty, \rho^\eps(x)}$. For $\eps>0$ small enough, the highest order term can be controlled as follows
 \begin{align}
\cF_\beta^m(\rho^\eps)\leq & \cF_\beta^m(\rho_\infty) + \bra*{\beta^{-1}m \rho_\infty^{m-2} + \rho_\infty^{-1/2}\frac{\hat{W}(k^\sharp)}{\Theta(k^\sharp)} }\frac{\eps^2}{2}\norm{e_{k^\sharp}}_{\Leb^2(\Omega)}^2 \\
&+ \cb{\beta^{-1}m(m-2)N_{k^\sharp}^{3}\frac{\eps^{3}}{6}  \norm{f}_{\Leb^\infty(\Omega)}^{m-3}\abs{\Omega}}\\
=&\cF_\beta^m(\rho_\infty) + \bra*{\beta^{-1}m \rho_\infty^{m-2} + \rho_\infty^{-1/2}\frac{\hat{W}(k^\sharp)}{\Theta(k^\sharp)} }\frac{\eps^2}{2}\norm{e_{k^\sharp}}_{\Leb^2(\Omega)}^2 \\
&+ o(\eps^2) \, .
 \end{align} 
 For $\beta>\beta_\sharp^m$, the second order term in the above expression has a negative sign. Thus, for $\eps>0$ sufficiently small we have that $\cF_\beta^m(\rho^\eps)< \cF_\beta^m(\rho_\infty) $. Since, by \cref{thm:exm}, minimisers of $\cF_\beta^m$ exists for all $\beta>0$, it follows that for all $\beta>\beta_\sharp^m$ there exist nontrivial minimisers of the free energy. Thus, there exists some $\beta_c\leq\beta_\sharp^m$ which is a transition point of the free energy $\cF_\beta^m(\rho)$. 
\end{proof}
\begin{remark}
We note here that the $\beta_\sharp^m$ defined in the statement of~\cref{prop:tp} corresponds exactly to the point of critical stability of the uniform state $\rho_\infty$, i.e. if the stationary problem is linearised about $\rho_\infty$, then $\beta_\sharp^m$ corresponds to the value of the parameter at which the first eigenvalue of  the linearised operator crosses the imaginary axis. 
\end{remark}
Before attempting to \cre{provide conditions for the existence of} continuous and discontinuous transition points we define the function $\mathbf{F}^m:(0,\infty) \to \R$
\begin{align}
\mathbf{F}^m(\beta):= \min_{\rho \in \cP(\Omega)}\cF_\beta^m \, .
\end{align}
\begin{lemma}
For all $\beta>0$, the function $\mathbf{F}^m$ is continuous. Assume further that there exists $\beta'>0$ and $\cP(\Omega)\ni\rho_{\beta'} \neq \rho_\infty $ such that $\cF_{\beta'}^m(\rho_{\beta'})=\mathbf{F}^m(\beta')$. Then for all $\beta>\beta'$, $\cF_\beta^m(\rho_\infty)> \mathbf{F}^m(\beta)$. 
\label{lem:Fcont}
\end{lemma}
\begin{proof}
We note that for $0<\beta\leq\beta_c$ (where $\beta_c$ is possibly $+\infty$) we have that $\mathbf{F}^m(\beta)= \cF_\beta^m(\rho_\infty)$ which is clearly a continuous function of $\beta$. Let $\beta_2>\beta_1>\beta_c$ (if $\beta_c<\infty$, else we are done) and let $\rho_{\beta_1}$ be the minimiser of $\cF_{\beta_1}^m$. Note however due to the structure of the free energy we have that
\begin{align}
\mathbf{F}^m(\beta_2) \leq & \cF_{\beta_2}^m(\rho_{\beta_1})  \\
=&\cF_{\beta_1}^m(\rho_{\beta_1})+ \frac{1}{m-1}(\beta_2^{-1}-\beta_1^{-1})\intom{\rho_{\beta_1}^m-\rho_{\beta_1}}\\=& \mathbf{F}^m(\beta_1) + \frac{1}{m-1}(\beta_2^{-1}-\beta_1^{-1})\intom{\rho_{\beta_1}^m-\rho_{\beta_1}} \, .
\end{align}
 \cb{To obtain continuity of $\mathbf{F}^m$, note that the steps of the above equation would still hold with $\beta_1$ and $\beta_2$ exchanged. Using that $\rho_{\beta_1}$ and $\rho_{\beta_2}$ are uniformly bounded by~\cref{thm:exm}, one has the desired continuity.}

Assume now that $\cF_{\beta'}^m(\rho_{\beta'})= \cF_{\beta'}^m(\rho_\infty)$ and let $\beta>\beta'$. We then have that
\begin{align}
\mathbf{F}^m(\beta) \leq & \cF_{\beta}^m(\rho_{\beta'})  \\
=&\cF_{\beta'}^m(\rho_{\beta'})+\frac{1}{m-1} (\beta^{-1}-\beta'^{-1})\intom{\rho_{\beta'}^m-\rho_{\beta'}}\\\leq & \cF_{\beta'}^m(\rho_\infty) + \frac{1}{m-1}(\beta^{-1}-\beta'^{-1})\intom{\rho_{\beta'}^m- \rho_{\beta'}}\\
< & \cF_{\beta'}^m(\rho_\infty) +\frac{1}{m-1} (\beta^{-1}-\beta'^{-1})\intom{\rho_\infty^m- \rho_\infty} =  \cF_{\beta}^m(\rho_\infty)  \, .
\end{align}
\end{proof}

We will now try and refine our \cre{descriptions} of discontinuous and continuous transition points in analogy with the results in~\cite{CP10,CGPS19}.
\begin{lemma}\label{lem:ctp@cs}
If a transition point $\beta_c>0$ is continuous, then $\beta_c=\beta_\sharp^m$.
\end{lemma}
\begin{proof}
We know already from~\cref{prop:tp} that $\beta_c \leq \beta_\sharp^m$. Let us assume that $\beta_c<\beta_\sharp^m$. We know from~\cref{def:cdtp} that $\rho_\infty$ is the unique minimiser of $\cF_{\beta_c}^m$. Additionally for any sequence of minimisers $\set{\rho_\beta}_{\beta>\beta_c}$ we know that 
\[
\limsup_{\beta \to \beta_c^+} \norm{\rho_\beta-\rho_\infty}_{\Leb^\infty(\Omega)}=0 \, .
\]
Consider such a sequence and set $\eta_\beta=\rho_\beta-\rho_\infty$. For $\beta>\beta_c$, we expand the free energy about $\rho_\infty$ as follows
\begin{align}
\cF_\beta^m(\rho_\beta)&
=\cF_\beta^m(\rho_\infty) + \beta^{-1}m \rho_\infty^{m-2} \frac{\norm{\eta_\beta}_{\Leb^2(\Omega)}^2}{2}+ \frac{\rho_\infty^{-1/2}}{2}\sum\limits_{k \in \N^d} \frac{\hat{W}(k)}{\Theta(k)}\sum\limits_{\sigma \in \Sym_k(\Lambda)}|\hat{\eta_\beta}(\sigma(k))|^2  \\
&- \frac{\beta^{-1}m(m-2)}{6}\intom{f^{m-3} \eta_\beta^3} \, .
\end{align}
where $f(x) \in \bra{\rho_\infty,\rho_\beta(x)}$ and can be bounded by $\norm{\rho_\beta}_{\Leb^\infty(\Omega)}\leq B_{\beta,m}\leq B$ from the result of~\cref{thm:exm} and~\cref{cor:bbound}. Additionally we can control $\dfrac{\hat{W}(k)}{\Theta(k)}$ to obtain the following bound 
\begin{align}
\cF_\beta^m(\rho_\beta)\geq &  \cF_\beta^m(\rho_\infty) + \bra*{\beta^{-1}m \rho_\infty^{m-2} + \rho_\infty^{-1/2}\min_{k \in \N^d, k \neq 0}\frac{\hat{W}(k)}{\Theta(k)}}\frac{\norm{\eta_\beta}_{\Leb^2(\Omega)}^2}{2} \\&- \frac{\beta^{-1}m(m-2)}{6}B^{m-3}\norm{\eta_\beta}_{\Leb^3(\Omega)}^3 \, .
\end{align}
Note that due to the fact that $\norm{\eta_\beta}_{\Leb^\infty(\Omega)} \to 0$ as $\beta\to \beta_c^+$, we have that $\norm{\eta_\beta}_{\Leb^3(\Omega)}^3$ is $o(\norm{\eta_\beta}_{\Leb^2(\Omega)}^2)$, i.e. $\norm{\eta_\beta}_{\Leb^3(\Omega)}^3 \leq \norm{\eta_\beta}_{\Leb^\infty(\Omega)}\norm{\eta_\beta}_{\Leb^2(\Omega)}^2$. This leaves us with
\begin{align}
\cF_\beta^m(\rho_\beta)\geq &  \cF_\beta^m(\rho_\infty) + \bra*{\beta^{-1}m \rho_\infty^{m-2} + \rho_\infty^{-1/2}\min_{k \in \N^d, k \neq 0}\frac{\hat{W}(k)}{\Theta(k)}}\frac{\norm{\eta_\beta}_{\Leb^2(\Omega)}^2}{2} \\&- o\bra*{\norm{\eta_\beta}_{\Leb^2(\Omega)}^2} \, .
\end{align}

Since $\beta_c<\beta_\sharp^m$, the term in the brackets is positive close to $\beta_c$ we obtain a contradiction as $\rho_\beta$ is a nontrivial minimiser of $\cF_\beta^m$. Thus, we must have that $\beta_c=\beta_\sharp^m$.
\end{proof} 
From~\cref{def:cdtp},  we see that some $\beta_c>0$ is a discontinuous transition point if it violates either (or both) of the conditions (1) and (2). In the following lemma, we will show that if (2) is violated then (1) is as well.
\begin{lemma}\label{lem:exnu}
Assume $\beta_c>0$ is a discontinuous transition point of the energy $\cF_{\beta}^m$ and that for some family of minimisers $\set{\rho_\beta}_{\beta>\beta_c}$ it holds that
\[
\limsup_{\beta \to \beta_c^+} \norm{\rho_{\beta}-\rho_\infty}_{\Leb^\infty(\Omega)} \neq 0 \, .
\]
Then there exists $\cP(\Omega)\ni \rho_{\beta_c} \neq \rho_\infty$ such that:
\begin{enumerate}
	\item $\mathbf{F}^m_{\beta_c}=\cF_{\beta_c}^m(\rho_{\beta_c})= \cF_{\beta_c}^m(\rho_\infty)$.
	\item $S_{\beta_c}^m(\rho_{\beta_c})> S_{\beta_c}^m(\rho_\infty)$ and $\cE(\rho_{\beta_c})< \cE(\rho_\infty)=0$.
\end{enumerate}
\end{lemma}
\begin{proof}
Consider a sequence of points $\set{\beta_n}_{n \in \N} >\beta_c$ and $\beta_n \to \beta_c $ as $n \to \infty$. We know that the set of minimisers $\set{\rho_{\beta_n}}_{n \in \N} $ is compact in $C^0(\Omega) \cap \cP(\Omega)$ from~\cref{compactness}. Thus,  there exists a subsequence $\rho_{\beta_{n}} \in \set{\rho_\beta}_{\beta>\beta_c} $ (which we do not relabel) and a limit $\rho_{\beta_c} \in \cP(\Omega) \cap C^0(\Omega)$ such that 
\[
\lim_{n \to \infty} \norm{\rho_{\beta_{n}} - \rho_{\beta_c}}_{C^0(\Omega)} =0 \, .
\]
From the statement of the lemma we know that $\rho_{\beta_c} \neq \rho_\infty$. All that remains is to show that $\rho_{\beta_c}$ is a minimiser of $\cF_{\beta_c}^m$. We first note that $\lim_{n \to \infty}\cF_{\beta_n}(\rho_{\beta_n})= \cF_{\beta_c}(\rho_{\beta_c})$. This follows from the fact that the interaction energy $\cE$ is continuous on $C^0(\Omega) \cap \cP(\Omega)$ for $W \in C^2(\Omega)$ and the entropy $S_{\beta}^m$ is essentially an $\Leb^m$-norm and is thus also controlled by the $C^0(\Omega)$ topology. Finally we use the result of~\cref{lem:Fcont} to note that
\begin{align}
\cF_{\beta_c}(\rho_{\beta_c}) =& \lim_{n \to \infty}\cF_{\beta_n}(\rho_{\beta_n}) \\=&
\lim_{n \to \infty}\mathbf{F}^m(\beta_n)= \mathbf{F}^m(\beta_c) \, ,
\end{align} 
which completes the proof of (1). The proof of (2) follows immediately from the fact that $\rho_\infty$ is the unique minimiser of $S_\beta^m(\rho)$ on $\cP(\Omega)$ (which is a consequence of Jensen's inequality).
\end{proof}
\begin{remark}\label{topology}
The above lemma tells us that we have not lost much by defining discontinuous transition points with respect to the $\Leb^\infty(\Omega)$ norm since the transition points obtained are discontinuous with respect to the $\Leb^p(\Omega)$ norm as well for all $p \in [1,\infty]$.
Indeed if we consider the sequence constructed in the proof of~\cref{lem:exnu} $\set{\rho_{\beta_n}}_{n \in \N}$ it follows that
\begin{align}
\lim_{n \to \infty} \norm{\rho_{\beta_n}-\rho_{\beta_c}}_{\Leb^p(\Omega)} \leq \abs{\Omega}^{1/p}\lim_{n \to \infty} \norm{\rho_{\beta_n}-\rho_{\beta_c}}_{C_0(\Omega)} =0 \, ,
\end{align}
where $\rho_{\beta_c}$ is the limiting object btained in the proof of~\cref{lem:exnu}. Thus, $\limsup\limits_{\beta \to \beta_c^+} \norm{\rho_{\beta}-\rho_\infty}_{\Leb^p(\Omega)} \neq 0$ for all $p \in [1,\infty]$.
\end{remark}

In the following proposition we outline the strategy we will use to provide sufficient conditions for the existence of continuous and discontinuous transition points.
\begin{proposition}\label{prop:strat}
Assume that $W \in \HH_s^c$ so that there exists a transition point $\beta_c>0$ of $\cF_\beta^m$. Then:
\begin{tenumerate}
\item If $\rho_\infty$ is the unique minimiser of $\cF_{\beta_\sharp^m}^m$, then $\beta_c=\beta_\sharp^m$
is a continuous transition point. \label{prop:strat!a}
\item If $\rho_\infty$ is not a minimiser of $\cF_{\beta_\sharp^m}^m$, then $\beta_c<\beta_\sharp^m$
is a discontinuous transition point. \label{prop:strat!b}
\end{tenumerate}
\end{proposition}
\begin{proof}
For the proof of~\Dref{prop:strat!a} we note that $\beta_c$ already satisfies condition (1) of~\cref{def:cdtp}. All we need to show is that it satisfies condition (2). Assume $\beta_c< \beta_\sharp^m$, then by the very definition of a transition point we would have a contradiction since $\rho_\infty$ is the unique minimiser of $\cF_\beta^m$ at $\beta=\beta_\sharp^m$. It follows then that $\beta_c=\beta_\sharp^m$. Assume now that condition (2) of~\cref{def:cdtp} is violated, i.e. there exists a family of minimisers
$\set{\rho_\beta}_{\beta>\beta_\sharp^m}$ such that 
\[
\lim_{\beta\to \beta_\sharp^{m +}} \norm{\rho_\beta-\rho_\infty}_{\Leb^\infty(\Omega)} \neq 0 \, .
\]
By~\cref{lem:exnu} it follows that there exists $\cP(\Omega)\ni\rho_{\beta_\sharp^m} \neq \rho_\infty$  which minimises $\cF_{\beta_\sharp^m}^m$. This is a contradiction.

For~\Dref{prop:strat!b}, we note that since $\rho_\infty$ is not a minimiser at $\beta=\beta_\sharp^m$ by~\cref{def:tp} and~\cref{prop:tp} it follows that $\beta_c<\beta_\sharp$. Thus, by~\cref{lem:ctp@cs}, $\beta_c$ is a discontinuous transition point.
\end{proof}
The next theorem provides conditions on the Fourier modes of $W(x)$ for the existence of discontinuous transition points. It can be thought of as the analogue for the case of nonlinear diffusion.
\begin{theorem}\label{thm:dctp}
Assume $W \in \HH_s^c$ and \cb{$m \neq 2$}. Define, for some $\delta>0$, the  set $K^\delta$ as follows
\[
K^{\delta}:=\left\{k' \in \N^d\setminus\set{\mathbf{0}}: \frac{\hat{W}(k')}{\Theta(k')}\leq \min_{k \in \N^d\setminus\set{\mathbf{0}}} \frac{\hat{W}(k)}{\Theta(k)} +\delta  \right\}
\]
We define $\delta_*$ to be the smallest value, if it exists, of $\delta$ for which the following condition is satisfied:

\begin{equation}\label{eq:c1}
  \text{there exist } k^a,k^b, k^c \in K^{\delta_*} , \text{ such that } k^a=k^b + k^c \,  \tag{A1}. 
\end{equation}
\cb{We remark that two of the modes in the above expression can be repeated. For example, we could have $k^a=2, k_b=1,k_c=1$.}  Then if $\delta_*$ is sufficiently small, $\cF_\beta^m$ exhibits a discontinuous transition point at some $\beta_c<\beta_\sharp$. 
\end{theorem}
\begin{proof}
We know already from~\cref{prop:tp} that the system possesses a transition point $\beta_c$. We are going to use \Dref{prop:strat!b} and construct a competitor $\rho \in \cP(\Omega)$ which has a lower value of the free energy than $\rho_\infty$ at $\beta=\beta_\sharp^m$. Define the function
\begin{align}
\gamma(m):=
\begin{cases}
1 & m <2 \\
-1 & m \geq 2
\end{cases}
\end{align}
and let
\begin{align}
\rho^\eps= \rho_\infty\bra[\bigg]{1 +\gamma(m) \eps \sum_{k \in K^{\delta_*} }e_{k}} \in \cP(\Omega) \ ,
\end{align}
for some $\eps >0$, sufficiently small. We denote by $\abs*{K^{\delta_*}}$ the cardinality of $K^{\delta_*}$, which is necessarily finite as $W \in \Leb^2(\Omega)$. Expanding about the free energy about $\rho_\infty$ we obtain
\begin{align}
\cF_{\beta_\sharp^m}^m(\rho^\eps)\leq & \cF_{\beta_\sharp^m}^m(\rho_\infty) + \abs*{K^{\delta_*}} \bra*{(\beta_\sharp^m)^{-1}m \rho_\infty^{m-2} + \rho_\infty^{-1/2}\min_{k \in \N^d \setminus\set{0}}\frac{\hat{W}(k)}{\Theta(k)} + \rho_\infty^{-1/2}\delta_* }\frac{\eps^2}{2}\norm{e_k}_{\Leb^2(\Omega)}^2 \\
&+ (\beta_\sharp^m)^{-1}\gamma(m)^3m(m-2) \rho_\infty^{m-3} \frac{\eps^{3}}{6}\intom{\bra*{\sum_{k \in K^{\delta_*}} e_k}^3} \\&+ (\beta_\sharp^m)^{-1} m(m-2)(m-3)  \frac{\eps^{4}}{24}\intom{f^{m-4} \bra*{\sum_{k \in K^{\delta_*}} e_k}^4} ,
\end{align} 
where the function $f(x) \in \bra{\rho_\infty, \rho^\eps(x)}$. We use the definition of $\beta_\sharp^m$ and control the highest order term in the same manner as~\cref{prop:tp} to simplify the expansion as follows:
\begin{align}
\cF_{\beta_\sharp^m}^m(\rho^\eps)\leq & \cF_{\beta_\sharp^m}^m(\rho_\infty) + \abs*{K^{\delta_*}} \bra*{\rho_\infty^{-1/2}\delta_* }\frac{\eps^2}{2} \\
&+ (\beta_\sharp^m)^{-1}\gamma(m)^3m(m-2) \rho_\infty^{m-3} \frac{\eps^{3}}{6}\intom{\bra*{\sum_{k \in K^{\delta_*}} e_k}^3} + o(\eps^3),
\end{align} 

Setting $\eps=\delta_*^{\frac{1}{2}}$ (if $\delta_* >0$, otherwise we stop here), we obtain
\begin{align}
\cF_{\beta_\sharp^m}^m(\rho^\eps)\leq &\cF_{\beta_\sharp^m}^m(\rho_\infty) + (\beta_\sharp^m)^{-1}\gamma(m)^3m(m-2) \rho_\infty^{m-3} \frac{\delta_*^{3/2}}{3}\intom{\bra*{\sum_{k \in K^{\delta_*}} e_k}^3} + \abs*{K^{\delta_*}}\rho_\infty^{-1/2}\frac{\delta_*^{2}}{2}+ o(\delta_*^{\frac{3}{2}})  \, .
\label{eq:defup}
\end{align} 
One can now check that under condition~\eqref{eq:c1}, it holds that
\[
  \intom{
\bra[\bigg]{\sum\limits_{k \in K^{\delta_*}}e_{k}}^3 } > a>0 \, ,
\]
 where the constant $a$ is independent of $\delta_*$. Indeed, the cube of the sum of $n$ numbers $a_i$, $i=1, \dots, n$ consists of only three types of terms, namely: $a_i^3$, $a_i^2 a_j$ and $a_i a_j a_k$. Setting the $a_i=w_{s(i)}$, with $s(i) \in K^{\delta_*}$, one can check that the first type of term will always integrate to zero.  \cb{The sum of the other two will take nonzero and in fact positive values if and only if condition~\eqref{eq:c1} is satisfied.}  This follows from the fact that 
 \[
 \int_{-\pi}^{\pi} \,\cos({\ell x})\cos(m x) \cos(nx) \! \dx{x}= \frac{\pi}{2}\bra*{\delta_{\ell+m,n}+\delta_{m+n,\ell} + \delta_{n+\ell,m}}\, .
\]
Also the term $\gamma(m)^3m(m-2)$ is always negative.  Thus, for $\delta_*$ sufficiently small, considering the fact that $|K^{\delta_*}| \geq 2$ and is nonincreasing as $\delta_*$ decreases, $\rho^\eps$ has smaller free energy and $\rho_\infty$ is not a minimiser at $\beta=\beta_\sharp^m$.
\end{proof}
\cb{
\begin{remark}\label{rem:special}
The case $m=2$ is special, as transition points for any $W \in \HH_s^c$ are necessarily discontinuous. This case will be treated in detail in~\cref{special}.
\end{remark}
}
The following lemma shows that discontinuous transitions are stable in $m$.
\begin{lemma}
Assume $W \in \HH_s^c$  such that $\cF^{m'}_\beta$ has a discontinuous transition point and $\beta_c^{m'} < \beta_\sharp^{m'}$. Then for $m \in (m'- \eps,m'+\eps )$ (or $m \in [1,1+\eps)$ for $m'=1$) for some $\eps>0$ small enough, $\cF_\beta^m$ has a discontinuous transition point at some $\beta_c^m<\beta_\sharp^m$.
\end{lemma}
\begin{proof}
We start with the case $m'>1$. Denote by $\rho^* \in C^0(\Omega) \cap \cP(\Omega)$ the nontrivial minimiser of $\cF_{\beta_\sharp^{m'}}^{m'}(\rho)$. We know that
\begin{align}
\cF_{\beta_\sharp^{m'}}^{m'}(\rho_\infty) - \cF_{\beta_\sharp^{m'}}^{m'}(\rho^*)= \delta >0 \, .
\end{align}
It would be sufficient for the purposes of this proof to show that such a nontrivial minimiser exists for 
$\cF_{\beta_\sharp^{m}}^{m}$  for $m$ close enough to $m'$. Choosing $\rho^*$ to be the competitor state, we have
\begin{align}
\cF_{\beta_\sharp^m}^m(\rho_\infty) - \cF_{\beta_\sharp^m}^m(\rho^*)= & \cF_{\beta_\sharp^{m'}}^{m'}(\rho_\infty) - \cF_{\beta_\sharp^{m'}}^{m'}(\rho^*)  \\&+ \frac{(\beta_\sharp^m)^{-1}}{m-1}\frac{1}{\abs{\Omega}^{m-1}} - \frac{(\beta_\sharp^m)^{-1}}{m-1} 
- \frac{(\beta_\sharp^{m'})^{-1}}{m'-1}\frac{1}{\abs{\Omega}^{m'-1}} + \frac{(\beta_\sharp^{m'})^{-1}}{m'-1} 
\\&  + \bra*{ \frac{(\beta_\sharp^{m'})^{-1}}{m'-1}\int_\Omega (\rho^*)^{m'} \dx{x} -\frac{(\beta_\sharp^{m'})^{-1}}{m'-1} -\frac{(\beta_\sharp^m)^{-1}}{m-1}\int_\Omega \rho^m \dx{x} +\frac{(\beta_\sharp^m)^{-1}}{m-1} } \\
=&\delta + \frac{(\beta_\sharp^m)^{-1}}{m-1}\frac{1}{\abs{\Omega}^{m-1}} - \frac{(\beta_\sharp^m)^{-1}}{m-1} 
- \frac{(\beta_\sharp^{m'})^{-1}}{m'-1}\frac{1}{\abs{\Omega}^{m'-1}} + \frac{(\beta_\sharp^{m'})^{-1}}{m'-1} 
\\&  + \bra*{ \frac{(\beta_\sharp^{m'})^{-1}}{m'-1}\int_\Omega (\rho^*)^{m'} \dx{x} -\frac{(\beta_\sharp^{m'})^{-1}}{m'-1} -\frac{(\beta_\sharp^m)^{-1}}{m-1}\int_\Omega \rho^m \dx{x} +\frac{(\beta_\sharp^m)^{-1}}{m-1} } 
\end{align}
Since $\beta_\sharp^m \to \beta_\sharp^{m'}$ and $(m-1)^{-1}(a^m -1) \to(m'-1)^{-1}(a^{m'} -1) , a \geq 0$ as $m \to m'$, it follows, using the fact that $\rho* \in C^0(\Omega)$, that we can choose $m$ close enough to $m'$ so that the above term is strictly positive. We then have that for $m \in (m'-\eps,m'+\eps)$ for some $\eps>0$ small enough, $\rho_\infty$ is not a minimiser of the free energy $\cF_{\beta_\sharp^m}^m(\rho)$. By~\Dref{prop:strat!b}, it follows that $\cF^m_\beta$ possesses a discontinuous transition point at some $\beta_c^m < \beta_\sharp^{m}$. The case $m'=1$ can be treated similarly.
\end{proof}

In the following proposition, we single out some special values of $m$ at which one always finds a discontinuous transition point for $W \in \HH_s^c$.
\begin{proposition}\label{special}
Assume $W \in \HH_s^c$ such that $\beta_c$ is a transition point of $\cF_\beta^m$. Then if $m \in [2,3]$, $\beta_c$ is a discontinuous transition point. Specifically for the case $m=2$
we have that
\begin{enumerate}
\item $\beta_c^2=\beta_\sharp^2$
\item There exists a one parameter family of minimiser $\set{\rho_\alpha}_{\alpha \in [0,\abs{\Omega}^{-1/2}\Theta(k^\sharp)^{-1}]}$ of $\cF_{\beta_\sharp^2}^2$ with $\rho_0=\rho_\infty$.
\end{enumerate}
\end{proposition}
\begin{proof}
We will try again to show that we have a competitor at $\beta_\sharp^m$. We start with the case $2<m<3$. Consider the competitor
\[
\rho^\eps=\rho + \eps e_{k^\sharp}
\]
for $\eps>0$ and small and $k^\sharp :=\arg \min_{k \in \N^d \setminus \set{0}}\hat{W}(k)/\Theta(k) $ if it is uniquely defined or any one such $k$ if it is not. Expanding the energy upto fifth order and noting that second order terms vanish we obtain
\begin{align}
\cF_{\beta_\sharp^m}^m(\rho^\eps)=& \cF_{\beta_\sharp^m}^m(\rho_\infty)  
+ (\beta_\sharp^m)^{-1} m(m-2) \rho_\infty^{m-3} \frac{\eps^{3}}{3!}\intom{e_{k^\sharp}^3} \\&+ (\beta_\sharp^m)^{-1} m(m-2)(m-3)\rho_\infty^{m-4}  \frac{\eps^{4}}{4!}\intom{ e_{k^\sharp}^4} \\
&+ (\beta_\sharp^m)^{-1} m(m-2)(m-3)(m-4)  \frac{\eps^{5}}{5!}\intom{ f^{m-5} e_{k^\sharp}^5} ,
\end{align}
where the function $f(x) \in \bra{\rho_\infty, \rho^\eps(x)}$. We again bound the highest order term as in~\cref{prop:tp} and use the fact that $\intom{e_k^3}=0$ for any $k \in \N^d \setminus \set{0}$ to obtain
\begin{align}
\cF_{\beta_\sharp^m}^m(\rho^\eps)=& \cF_{\beta_\sharp^m}^m(\rho_\infty)  +(\beta_\sharp^m)^{-1} m(m-2)(m-3)\rho_\infty^{m-4}  \frac{\eps^{4}}{4!}\intom{ e_{k^\sharp}^4} + o(\eps^4) \, .
\end{align}
Since $m(m-2)(m-3)$ is negative for $m \in(2,3)$, for $\eps>0$ sufficiently small, we have shown that $\rho_\infty$ is no longer the minimiser of $\cF_{\beta_\sharp^m}^m$. The result follows by~\Dref{prop:strat!b}: we have a discontinuous transition point at some $\beta_c<\beta_\sharp^m$.

We now consider the case $m=2,3$. Using the same expansion we have that
\begin{align}
\cF_{\beta_\sharp^2}^2(\rho^\eps) = \cF_{\beta_\sharp^2}^2(\rho_\infty) \qquad \cF_{\beta_\sharp^3}^3(\rho^\eps) = \cF_{\beta_\sharp^3}^3(\rho_\infty) \, .
\end{align}
Thus, $\rho_\infty$ is not the unique minimiser of $\cF_{\beta_\sharp^m}^m$ for $m=2,3$. It then follows from~\cref{def:tp} that there must be a discontinuous transition point at $\beta_c^m \leq \beta_\sharp^m$.

Consider now the convex interpolant $\rho_t:= (1-t) \rho_0 + t\rho_1, t \in(0,1) $  for $\rho_0,\rho_1 \in \cP(\Omega)$ such that $\cF_\beta^2(\rho_0), \cF_\beta^2(\rho_1)< \infty$. We then have that
\begin{align}
\frac{\dx^2}{\dx{t}^2}\cF_\beta^2(\rho_t)= &2 \beta^{-1} \intom{  \eta^2} +\iintom{W(x-y) \eta(x) \eta(y)} \\
\geq & \bra*{2 \beta^{-1} + \min_{k \in \N^d \setminus \set{0}} \frac{\hat{W}(k)}{\Theta(k)}} \norm{\eta}_{\Leb^2(\Omega)}^2 \, .
\end{align}
Note that the above expression is strictly positiove if $\beta<\beta_\sharp^2$. Thus, $\cF_\beta^2$ is strictly convex for $\beta <\beta_\sharp^2$ and has only one minimiser, namely, $\rho_\infty$. Since the function $\mathbf F$ is continuous (cf.~\cref{lem:Fcont}), it follows that $\beta_c^2=\beta_\sharp^2$ for all $W \in \HH_s^c$. Furthermore, $\rho_\alpha= \rho_\infty + \alpha e_{k^\sharp}$ form a one-parameter family of minimisers of $\cF_{\beta_\sharp^2}^2$ for $\alpha \in [0,\abs{\Omega}^{-1/2}\Theta(k^\sharp)^{-1}]$.
\end{proof}
We conclude the section by discussing the existence of continuous transition points. We show that for $m=4$
one can construct a large class of potentials for which the transition point $\beta_c$ is continuous. We start with the following proposition.
\begin{proposition}\label{trig}
Let $k \in \N^d$ be such that $k \not \equiv 0$ and let $k_i \in \N, i=1, \dots, d$ be such that
\begin{align}
k=\begin{pmatrix}
k_1 & \dots &k_d 
\end{pmatrix}
\, .
\end{align}
Then we have:
\begin{align}
e_k^2= \sum_{j \in P_2(k)} c_j e_j + c_{0}e_0 \, ,
\end{align} 
where 
\begin{align}
P_2(k):=\set{ j \in \N^d,j\not \equiv 0, j_i \in \set{2k_i,0}}\, , \, c_j= \frac{\rho_\infty}{N_j } \textrm{ and } c_0 =\frac{\rho_\infty}{N_0 } \, .
\end{align}
Similarly
\begin{align}
e_k^3=  \sum_{\ell \in P_3(k)} c_{\ell}e_{\ell} +c_{k}e_k
\end{align}
with
\begin{align}
P_3(k):=\set{\ell \in \N^d,\ell \neq k, \ell_i \in \set{3k_i,k_i}}\, , \, c_\ell= \frac{\rho_\infty^2}{N_{k^\sharp}N_\ell }(3)^{\abs*{\set{\ell_i:\ell_i=k_i}}} \textrm{ and } c_k  =  \frac{\rho_\infty^2}{N_{k^\sharp}N_\ell }(3)^{d} \, .
\end{align}
Note that $P_2(k) \cap P_3(k)=\emptyset$. Similarly, we have that
\begin{align}
\bra*{\sum_{\sigma \in \Sym_k(\Lambda)}a_{\sigma(k)} e_{\sigma(k)}}^2=& \sum_{j \in P_2(k)} \sum_{\sigma_1, \sigma_2 \in \Sym_k(\Lambda)} a_{\sigma_1(k)}a_{\sigma_2(k)}c^{\sigma_1,\sigma_2}_{j} e_{\sigma_1\cdot \sigma_2 (j)} +C_0 e_0  \\
\bra*{\sum_{\sigma \in \Sym_k(\Lambda)}a_{\sigma(k)} e_{\sigma(k)}}^3=& \sum_{\ell \in P_3(\ell)} \sum_{\sigma_1, \sigma_2,\sigma_3 \in \Sym_k(\Lambda)} a_{\sigma_1(k)}a_{\sigma_2(k)} a_{\sigma_3(k)}c^{\sigma_1,\sigma_2,\sigma_3}_{\ell} e_{\sigma_1\cdot \sigma_2\cdot\sigma_3(\ell)} \\&+\sum_{\sigma \in \Sym_k(\Lambda)}C_{k}^\sigma e_k
\end{align}
where the constants $c^{\sigma_1,\sigma_2}_{j}, c^{\sigma_1,\sigma_2,\sigma_3}_{\ell}, C_0,C_{k}^\sigma \in \R$ depend only on $d$, $k$, and $\rho_\infty$ but are independent of the coefficients $a_{\sigma(k)} \in \R$. Note that, as before, there is no repetition of the terms in the sum. 
\end{proposition}
\begin{proof}
The proof is simply a careful application of the trigonometric identities $\cos^2(a)= 2^{-1}(1+ \cos(2a))$, $\cos^3(a)=4^{-1}(\cos(3a)+3\cos(a))$, and $\sin^3(a) = 4^{-1}\bra*{3\sin(a) - \sin (3a)}$.
\end{proof}
We now proceed to the result concerning continuous transition points for $m=4$.
\begin{theorem}\label{m=4}
Let $W \in \HH_s^c$, such that $\beta_c<\infty$ is a transition point of $\cF_\beta^4$. Assume that 
\begin{align}
k^\sharp:= \arg\min\limits_{k \in \N^d, k \neq 0}\frac{\hat{W}(k)}{\Theta(k)} \, ,
\end{align}
is uniquely defined. Furthermore, we assume that $\hat{W}(k) \geq 0$ for all $k \neq k^\sharp$ and that
\begin{align}
\hat{W}(j) &> \max_{\sigma_1,\sigma_2 \in \Sym_{k^\sharp}(\Lambda)}\frac{6   \Theta(j)^5 \bra*{c_j^{\sigma_1,\sigma_2}}^2 \abs*{P_2(k^\sharp) \cup P_3(k^\sharp)}}{ \rho_\infty \Theta(k^\sharp) }\abs*{\hat{W}(k^\sharp)} \qquad \forall j \in P_2(k^\sharp) \tag{A2} \label{ass1} \\
\hat{W}(\ell) &> \max_{\sigma_1,\sigma_2, \sigma_3 \in \Sym_{k^\sharp}(\Lambda)} \frac{2 \Theta(\ell)^9 \Theta\bra{k^\sharp} \bra*{c_\ell^{\sigma_1,\sigma_2,\sigma_3}}^2 \abs*{P_2(k^\sharp) \cup P_3(k^\sharp)}}{ 3 \rho_\infty^2 }\abs*{\hat{W}(k^\sharp)} \qquad \forall \ell \in P_3(k^\sharp) \tag{A3} \label{ass2} \, ,
\end{align}
where the sets $P_2,P_3$ and the constants $c^{\sigma_1,\sigma_2}_{j}, c^{\sigma_1,\sigma_2,\sigma_3}_{\ell}$ are as defined in~\cref{trig}. Then $\beta_c=\beta_\sharp^4$ is a continuous transition point. \cre{Note that the constant $\Theta(k)$ for $k \in \N^d$ is as defined in~\eqref{NkTk}.}
\end{theorem}
\begin{proof}
We will rely on~\Dref{prop:strat!a} for the proof of this result. We need to show that, at $\beta=\beta_\sharp^4$, $\rho_\infty$ is the unique minimiser of $\cF_\beta^4$. Let $\rho \in \cP(\Omega) \in \Leb^\infty(\Omega)$ be any measure different from $\rho_\infty$. Then it is sufficient to show that $\cF_{\beta_\sharp^4}^4(\rho) > \cF_{\beta_\sharp^4}^4(\rho_\infty)$ (it is sufficient to check bounded densities from the result of~\cref{lem:eqbound}). We now define $\eta:= \rho -\rho_\infty$ and note that $\eta$ has the following properties
\begin{align}
\eta \in \Leb^\infty(\Omega), \qquad \eta \geq -\rho_\infty,  \qquad \intom{\eta}=0 \, .\label{etaprops}
\end{align}
We can compute the free energy of $\rho$ as follows
\begin{align}
\cF_{\beta_\sharp^4}^4(\rho)=& \frac{(\beta_\sharp^4)^{-1}}{3} \intom{\rho^4} -\frac{(\beta_\sharp^4)^{-1}}{3} + \frac{1}{2}\iintom{W(x-y) \rho(x) \rho(y)} \\
=& \frac{(\beta_\sharp^4)^{-1}}{3} \bra*{\intom{\rho_\infty^4} -1 + 4\intom{\rho_\infty^3 \eta} + 6\intom{\rho_\infty^2 \eta^2}  + 4\intom{\rho_\infty \eta^3} +\intom{ \eta^4}} \\& +\sum\limits_{k \in \N^d} \hat{W}(k)\frac{1}{2 N_k}\sum\limits_{\sigma \in \Sym_k(\Lambda)}|\hat{\eta}(\sigma(k))|^2 \, ,
\end{align}
where we have used~\eqref{Fourier:Interaction}. Simplifying further, by using the definition of $\beta_\sharp^4$ and the fact that $\eta$ has mean zero, we obtain
\begin{align}
\cF_{\beta_\sharp^4}^4(\rho)=&  \cF_{\beta_\sharp^4}^4(\rho_\infty) +\sum\limits_{k \in \N^d, k \neq k^\sharp} \bra*{6\frac{(\beta_\sharp^4)^{-1}}{3}\rho_\infty^2 + \hat{W}(k)\frac{1}{2 N_k}}\sum\limits_{\sigma \in \Sym_k(\Lambda)}|\hat{\eta}(\sigma(k))|^2  \\&+ \underbrace{4\frac{(\beta_\sharp^4)^{-1}}{3} \intom{\rho_\infty \eta^3}}_{I_1} + \underbrace{\frac{(\beta_\sharp^4)^{-1}}{3}\intom{ \eta^4} }_{I_2}  \, . \label{main}
\end{align}
We define $\eta_2:=\eta- f_{\eta,k^\sharp}$ where $f_{\eta,k^\sharp}=\sum_{\sigma \in \Sym_{k^\sharp}(\Lambda)}\hat{\eta}(\sigma(k^\sharp))e_{\sigma(k^\sharp)}$ and deal with the two terms $I_1$ and $I_2$ separately. We then have
\begin{align}
I_1=& 4\frac{(\beta_\sharp^4)^{-1}}{3} \intom{\rho_\infty \eta^3} \\
=& 4\frac{(\beta_\sharp^4)^{-1}}{3}\rho_\infty \left(\intom{\pra*{f_{\eta,k^\sharp}^3 
+ 3 f_{\eta,k^\sharp}^2 \eta_2 }} \right. \left.+ \intom{ \pra*{3f_{\eta,k^\sharp} \eta_2^2 + \eta_2^3}} \right)\\
=&4\frac{(\beta_\sharp^4)^{-1}}{3}\rho_\infty \intom{\pra*{3f_{\eta,k^\sharp}^2 \eta_2
+ 3 f_{\eta,k^\sharp} \eta_2^2  +  \eta_2^3 }} \, ,
\end{align}
where we have used the fact that 
\[
\intom{f_{\eta,k^\sharp}^3}=0 \, .
\]
 We now use the fact that $\eta$ has mean zero from~\eqref{etaprops} and~\cref{trig} to obtain
\begin{align}
I_1= & \frac{(\beta_\sharp^4)^{-1}}{3} \bra*{4 \rho_\infty \intom{\eta_2^3} + 12\rho_\infty \intom{f_{\eta,k^\sharp} \eta_2^2} }\\&
+4(\beta_\sharp^4)^{-1} \rho_\infty\sum_{j \in P_2(k^\sharp)} \sum_{\sigma_1, \sigma_2 \in \Sym_{k^\sharp}(\Lambda)}\hat{\eta}(\sigma_1(k^\sharp)) \hat{\eta}(\sigma_2(k^\sharp))c^{\sigma_1,\sigma_2}_{j} \hat{\eta_2}(\sigma_1\cdot\sigma_2(j)) \, . \label{I1}
\end{align}
For the second term we obtain
\begin{align}
I_2 =&\frac{(\beta_\sharp^4)^{-1}}{3}\intom{ \eta^4}  \\
=&\frac{(\beta_\sharp^4)^{-1}}{3}\intom{\pra*{f_{\eta,k^\sharp}^4 + 4 f_{\eta,k^\sharp}^3 \eta_2 + 6f_{\eta,k^\sharp}^2 \eta_2^2 + 4f_{\eta,k^\sharp} \eta_2^3 + \eta_2^4 }}  \, .
\end{align}
 Applying~\cref{trig} again, we obtain
\begin{align}
I_2 = &\frac{(\beta_\sharp^4)^{-1}}{3}\intom{ \eta^4}  \\
=&\frac{(\beta_\sharp^4)^{-1}}{3}\intom{\pra*{f_{\eta,k^\sharp}^4 + 6f_{\eta,k^\sharp}^2 \eta_2^2  +  4f_{\eta,k^\sharp} \eta_2^3 + \eta_2^4 }}   \\&+ 4 \frac{(\beta_\sharp^4)^{-1}}{3} \sum_{\ell \in P_3(\ell)} \sum_{\sigma_1, \sigma_2,\sigma_3 \in \Sym_{k^\sharp}(\Lambda)} \hat{\eta}\bra{\sigma_1(k^\sharp)}\hat{\eta}\bra{\sigma_2(k^\sharp)} \hat{\eta}\bra{\sigma_3(k^\sharp)}c^{\sigma_1,\sigma_2,\sigma_3}_{\ell} \hat{\eta_2}\bra{\sigma_1\cdot \sigma_2\cdot\sigma_3(\ell)}  \, , \label{I2}
\end{align}
where we have used the fact that $\hat{\eta_2}(\sigma(k^\sharp))=0$ for all $\sigma \in \Sym_{k^\sharp}(\Lambda)$. We now note that
\begin{align}
\sum\limits_{k \in \N^d, k \neq k^\sharp} \sum\limits_{\sigma \in \Sym_k(\Lambda)}|\hat{\eta}(\sigma(k))|^2= \norm{\eta_2}^2_{\Leb^2(\Omega)} \,. \label{symmetry}
\end{align}
Putting~\eqref{main},~\eqref{I1},~\eqref{I2}, and~\eqref{symmetry}, together we obtain
\begin{align}
\cF_{\beta_\sharp^4}^4(\rho)= &  \cF_{\beta_\sharp^4}^4(\rho_\infty) +\sum\limits_{k \in \N^d, k \neq k^\sharp} \bra*{\hat{W}(k)\frac{1}{2 N_k}}\sum\limits_{\sigma \in \Sym_k(\Lambda)}|\hat{\eta_2}(\sigma(k))|^2  \\&+4(\beta_\sharp^4)^{-1} \rho_\infty\sum_{j \in P_2(k^\sharp)} \sum_{\sigma_1, \sigma_2 \in \Sym_{k^\sharp}(\Lambda)}\hat{\eta}(\sigma_1(k^\sharp)) \hat{\eta}(\sigma_2(k^\sharp))c^{\sigma_1,\sigma_2}_{j} \hat{\eta_2}(\sigma_1\cdot\sigma_2(j))\\ &+  4 \frac{(\beta_\sharp^4)^{-1}}{3} \sum_{\ell \in P_3(\ell)} \sum_{\sigma_1, \sigma_2,\sigma_3 \in \Sym_{k^\sharp}(\Lambda)} \hat{\eta}\bra{\sigma_1(k^\sharp)}\hat{\eta}\bra{\sigma_2(k^\sharp)} \hat{\eta}\bra{\sigma_3(k^\sharp)}c^{\sigma_1,\sigma_2,\sigma_3}_{\ell} \hat{\eta_2}\bra{\sigma_1\cdot \sigma_2\cdot\sigma_3(\ell)}\\&+ \frac{(\beta_\sharp^4)^{-1}}{3} \int_{\Omega}\! \left[ 6\rho_\infty^2 + 12\rho_\infty f_{\eta,k^\sharp} + 4 \eta_2 \rho_\infty + 6f_{\eta,k^\sharp}^2 + 4f_{\eta,k^\sharp} \eta_2 + \eta_2^2 \right] \eta_2^2 \, \dx{x}  \\& + \frac{(\beta_\sharp^4)^{-1}}{3}\intom{f_{\eta,k^\sharp}^4} \,.
\end{align}
Note now that
\begin{align}
 &\left[ 6\rho_\infty^2 + 12\rho_\infty f_{\eta,k^\sharp}+ 4 \eta_2 \rho_\infty + 6f_{\eta,k^\sharp}^2 + 4f_{\eta,k^\sharp} \eta_2 + \eta_2^2 \right] \\
=& \left[2 \bra*{\rho_\infty  + f_{\eta,k^\sharp} }^2 + \bra*{\eta_2 + 2\bra*{\rho_\infty  + f_{\eta,k^\sharp} }}^2 \right] \geq 0 \, .
\end{align}
Thus, it follows that
\begingroup
\allowdisplaybreaks
\begin{align}
&\cF_{\beta_\sharp^4}^4(\rho) \\
\geq & \cF_{\beta_\sharp^4}^4(\rho_\infty) +\sum\limits_{k \in \N^d, k \neq k^\sharp} \bra*{\hat{W}(k)\frac{1}{2 N_k}}\sum\limits_{\sigma \in \Sym_k(\Lambda)}|\hat{\eta_2}(\sigma(k))|^2  \\&+4(\beta_\sharp^4)^{-1} \rho_\infty\sum_{j \in P_2(k^\sharp)} \sum_{\sigma_1, \sigma_2 \in \Sym_{k^\sharp}(\Lambda)}\hat{\eta}(\sigma_1(k^\sharp)) \hat{\eta}(\sigma_2(k^\sharp))c^{\sigma_1,\sigma_2}_{j} \hat{\eta_2}(\sigma_1\cdot\sigma_2(j))\\ &+  4 \frac{(\beta_\sharp^4)^{-1}}{3} \sum_{\ell \in P_3(\ell)} \sum_{\sigma_1, \sigma_2,\sigma_3 \in \Sym_{k^\sharp}(\Lambda)} \hat{\eta}\bra{\sigma_1(k^\sharp)}\hat{\eta}\bra{\sigma_2(k^\sharp)} \hat{\eta}\bra{\sigma_3(k^\sharp)}c^{\sigma_1,\sigma_2,\sigma_3}_{\ell} \hat{\eta_2}\bra{\sigma_1\cdot \sigma_2\cdot\sigma_3(\ell)}\\& + \frac{(\beta_\sharp^4)^{-1}}{3}\intom{f_{\eta,k^\sharp}^4}\\
\geq & \cF_{\beta_\sharp^4}^4(\rho_\infty) +\sum\limits_{k \in P_2(k^\sharp) \cup P_3(k^\sharp)} \bra*{\hat{W}(k)\frac{1}{2 N_k}}\sum\limits_{\sigma \in \Sym_k(\Lambda)}|\hat{\eta_2}(\sigma(k))|^2  \\&\\&+4(\beta_\sharp^4)^{-1} \rho_\infty\sum_{j \in P_2(k^\sharp)} \sum_{\sigma_1, \sigma_2 \in \Sym_{k^\sharp}(\Lambda)}\hat{\eta}(\sigma_1(k^\sharp)) \hat{\eta}(\sigma_2(k^\sharp))c^{\sigma_1,\sigma_2}_{j} \hat{\eta_2}(\sigma_1\cdot\sigma_2(j))\\ &+  4 \frac{(\beta_\sharp^4)^{-1}}{3} \sum_{\ell \in P_3(\ell)} \sum_{\sigma_1, \sigma_2,\sigma_3 \in \Sym_{k^\sharp}(\Lambda)} \hat{\eta}\bra{\sigma_1(k^\sharp)}\hat{\eta}\bra{\sigma_2(k^\sharp)} \hat{\eta}\bra{\sigma_3(k^\sharp)}c^{\sigma_1,\sigma_2,\sigma_3}_{\ell} \hat{\eta_2}\bra{\sigma_1\cdot \sigma_2\cdot\sigma_3(\ell)}\\& + \frac{(\beta_\sharp^4)^{-1}}{3}\intom{f_{\eta,k^\sharp}^4} \, ,
\label{eq:penineq}
\end{align}
\endgroup
where in the last step we have simply used the fact that $\hat{W}(k) \geq 0$ for all $k \neq k^\sharp$.  We now note that 
\begin{align}
\sum_{\sigma \in \Sym_k(\Lambda)} |\hat{\eta_2}(\sigma(k))|^2 =& \Theta\bra{k}^{-2} \sum_{\sigma_1, \sigma_2 \in \Sym_k(\Lambda)} |\hat{\eta_2}(\sigma_1 \cdot \sigma_2(k))|^2 \\ =& \Theta\bra{k}^{-4} \sum_{\sigma_1, \sigma_2, \sigma_3 \in \Sym_k(\Lambda)} |\hat{\eta_2}(\sigma_1 \cdot \sigma_2 \cdot \sigma_3(k))|^2 \,  ,
\end{align}
where we have used the fact that $\abs{\Sym_k(\Lambda)} =\Theta(k)^2$. Additionally, we have that 
\begin{align}
\intom{f_{\eta,k^\sharp}^4}=& \bra*{\sum_{\sigma \in \Sym_{k^\sharp}(\Lambda)}\abs*{\hat{\eta}(\sigma(k^\sharp))}^2}^2\intom{\bra*{\frac{1}{\bra*{\sum_{\sigma \in \Sym_{k^\sharp}(\Lambda)}\abs*{\hat{\eta}(\sigma(k^\sharp))}^2}^{1/2}}\sum_{\sigma \in \Sym_{k^\sharp}(\Lambda)}\hat{\eta}(\sigma(k^\sharp))e_{\sigma(k^\sharp)}}^4 }\\
\geq & \rho_\infty \bra*{\sum_{\sigma \in \Sym_{k^\sharp}(\Lambda)}\abs*{\hat{\eta}(\sigma(k^\sharp))}^2}^2  \, ,
\end{align}
where in the last step we applied Jensen's inequality and used the fact that the integrand has unit $\Leb^2(\Omega)$ norm. For any $k \in \N^d$, we define the following quantity
\begin{align}
\abs*{\hat{\eta}}_{k}^2=\sum_{\sigma \in \Sym_k(\Lambda)}\abs*{\hat{\eta}(\sigma(k))}^2 \,,
\end{align}
and note that
\begin{align}
\abs*{\hat{\eta}}_{k}^4 \geq \max_{\sigma_1,\sigma_2 \in \Sym_k(\Lambda)}\prod_{i=1}^2 \abs*{\hat{\eta}(\sigma_1(k))}^2 \abs*{\hat{\eta}(\sigma_2(k))}^2 \,.
\label{eq:seminormbound}
\end{align}
 Finally, we can rewrite the inequality in~\eqref{eq:penineq} as
\begin{align}
\cF_{\beta_\sharp^4}^4(\rho) \geq&  \cF_{\beta_\sharp^4}^4(\rho_\infty) + \sum_{j \in P_2(k^\sharp)} \sum_{\sigma_1,\sigma_2 \in \Sym_{k^\sharp}(\Lambda)}\bra*{A_j\abs*{ \hat{\eta_2}(\sigma_1 \cdot \sigma_2 (j))}^2 + B_j^{\sigma_1,\sigma_2} \hat{\eta_2}(\sigma_1 \cdot \sigma_2 (j)) + C_j} \\&+\sum_{\ell \in P_3(k^\sharp)}\sum_{\sigma_1,\sigma_2,\sigma_3 \in \Sym_{k^\sharp}(\Lambda)}\bra*{ A_\ell \abs*{\hat{\eta_2}(\sigma_1 \cdot \sigma_2\cdot \sigma_3 (\ell))}^2 + B_\ell^{\sigma_1,\sigma_2,\sigma_3} \hat{\eta_2}(\sigma_1 \cdot \sigma_2 \cdot\sigma_3 (\ell)) + C_\ell } \, , \label{final}
\end{align}
where
\begin{align}
&A_j=\frac{\hat{W}(j)}{2N_j}\Theta(j)^{-2} \,  &&A_\ell=\frac{\hat{W}(\ell)}{2N_\ell}\Theta(\ell)^{-4}  \\
&B_j^{\sigma_1,\sigma_2}=4(\beta_\sharp^4)^{-1} \rho_\infty c^{\sigma_1,\sigma_2}_{j} \prod_{i=1}^2 \hat{\eta}(\sigma_i(k^\sharp))  \, && B_\ell^{\sigma_1,\sigma_2,\sigma_3}= 4\frac{(\beta_\sharp^4)^{-1}}{3}  c^{\sigma_1,\sigma_2,\sigma_3}_{\ell}\prod_{i=1}^3\hat{\eta}\bra{\sigma_i(k^\sharp)} \\
&C_j= \frac{(\beta_\sharp^4)^{-1}}{ 3\Theta(j)^2\abs*{P_2(k^\sharp) \cup P_3(k^\sharp)}} \rho_\infty \abs*{\hat{\eta}}_{k^\sharp}^4   \, && C_\ell= \frac{(\beta_\sharp^4)^{-1}}{ 3\Theta(\ell)^4\abs*{P_2(k^\sharp) \cup P_3(k^\sharp)}} \rho_\infty \abs*{\hat{\eta}}_{k^\sharp}^4 
\end{align}
Assume that $\abs*{\hat{\eta}}_{k^\sharp} \neq 0$. Then~\eqref{ass1} and~\eqref{ass2} along with the expression for $\beta_\sharp^4$,~\eqref{eq:seminormbound}, and the fact that $\abs*{\hat{\eta}(k)} \leq N_{k} $,  imply that the discriminants of the quadratic expressions in~\eqref{final} are all negative, i.e. $\bra*{B_j^{\sigma_1,\sigma_2}}^2-4 A_jC_j <0,\bra*{B_\ell^{\sigma_1,\sigma_2,\sigma_3}}^2-4 A_\ell C_\ell <0$.  Indeed, we have that
\begin{align}
\frac{\bra*{B_j^{\sigma_1,\sigma_2}}^2}{4 A_jC_j} =&\frac{24 (\beta_\sharp^4)^{-1} \rho_\infty \Theta(j)^4 \bra*{c_j^{\sigma_1,\sigma_2}}^2 N_j\abs*{P_2(k^\sharp) \cup P_3(k^\sharp)}}{  \abs*{\hat{\eta}}_{k^\sharp}^4  \hat{W}(j) } \prod_{i=1}^2\abs*{\hat{\eta}(\sigma_i(k^\sharp))}^2 \\
\leq &\frac{6  \abs*{\hat{W}(k^\sharp)}  \Theta(j)^5 \bra*{c_j^{\sigma_1,\sigma_2}}^2 \abs*{P_2(k^\sharp) \cup P_3(k^\sharp)}}{ \rho_\infty \Theta(k^\sharp) \hat{W}(j) }  \stackrel{\eqref{ass1}}{<}1 \, . 
\end{align}
Similarly,
\begin{align}
\frac{\bra*{B_\ell^{\sigma_1,\sigma_2,\sigma_3}}^2}{4 A_\ell C_\ell } =&\frac{8 (\beta_\sharp^4)^{-1} \Theta(\ell)^8 \bra*{c_\ell^{\sigma_1,\sigma_2,\sigma_3}}^2 N_{\ell}\abs*{P_2(k^\sharp) \cup P_3(k^\sharp)}}{ 3 \rho_\infty \abs*{\hat{\eta}}_{k^\sharp}^4  \hat{W}(\ell) } \prod_{i=1}^3\abs*{\hat{\eta}(\sigma_i(k^\sharp))}^2  \\
\leq & \frac{2\abs{\hat{W}(k^\sharp)} \Theta(\ell)^9 \Theta\bra{k^\sharp} \bra*{c_\ell^{\sigma_1,\sigma_2,\sigma_3}}^2 \abs*{P_2(k^\sharp) \cup P_3(k^\sharp)}}{ 3  \rho_\infty^2  \hat{W}(\ell) }  \stackrel{\eqref{ass2}}{<}1 \, .
\end{align}
Thus, it follows that $\cF_{\beta_\sharp^4}^4(\rho)>\cF_{\beta_\sharp^4}^4(\rho_\infty)$. On the other hand if $\abs*{\hat{\eta}}_{k^\sharp} \neq 0$, the proof follows by noting that any contribution from the interaction energy is positive and that $\rho_\infty$ is the unique minimiser of $S_{\beta,4}(\rho)$. The fact that $\beta_c=\beta_\sharp^4$ is a consequence of~\cref{lem:ctp@cs}.
\end{proof}

\begin{remark}
Note that although the assumptions in~\cref{m=4} seem complicated, all they really require is that all Fourier coefficients of $W$, except the dominant negative mode $\hat{W}(k^\sharp)$ are nonnegative and that a finitely many of them ``positive enough'' compared to $\hat{W}(k^\sharp)$. Consider $d=1$, with $W(x)= w_1 e_1(x) + w_2 e_2(x) + w_3 e_3(x)$ with $w_1<0$ and $w_2,w_3>0$. If, for some explicitly computable positive constants $c_2,c_3>0$, $w_2 > c_2 \abs{w_1} $ and $w_3 > c_3 \abs{w_1}$, the conditions of~\cref{m=4} are satisfied and the transition point $\beta_c=\beta_\sharp^4$ is continuous. In this setting, $P_2(1)=\set{e_2}$ and $P_3(1)=\set{e_3}$. 
\end{remark}

\section{The mesa limit $m \to \infty$} \label{mesa}
A natural question to ask is ho w the sequence of free energies $\cF_\beta^m: \cP(\Omega) \to (-\infty,+\infty]$ behave in the limit as $m \to \infty$. We conjecture the following limit free energy, $\cF^\infty:\cP(\Omega) \to (-\infty,+\infty]$, 
\begin{align}
\cF^\infty (\rho)= 
\begin{cases}
\dfrac{1}{2}{\displaystyle \iintom{W(x-y) \rho(x) \rho(y)} }& \norm{\rho}_{\Leb^\infty(\Omega)} \leq 1 \\
+ \infty & \textrm{otherwise} 
\end{cases}
\, \label{eq:mesalim}.
\end{align}
This is analogous to the so-called mesa limit of the porous medium equation considered by Caffarelli and Friedman \cite{CF87}. It is also studied in~\cite{CKY18,CT20} for Newtonian interactions and~\cite{KPW19} for general drift-diffusion equations. We rederive the result in our setting.
\begin{theorem} \label{thm:mesa}
 Consider the sequence of functionals $\set{\cF_\beta^m}_{m \geq 1}$ defined on $\cP(\Omega) \cap \Leb^\infty(\Omega)$ equipped with the weak-$*$ topology.  Then
\[
\cF^\infty= \Gamma\mhyphen \lim\limits_{m \to \infty} \cF_\beta^m  \, ,
\]
 for any fixed $\beta>0$.
\label{thm:Gcon}
\end{theorem}
\begin{proof}

\begin{enumerate}[wide, labelwidth=!, labelindent=0pt]
\item Recovery sequence: For each $\rho \in \cP(\Omega) \cap \Leb^\infty(\Omega)$ we choose $\rho_m=\rho$
as the recovery sequence. The interaction energy term remains unchanged as it is independent of $m$, while $(m-1)^{-1}$ converges to $0$ as $m \to \infty$. Assume first that $\norm{\rho}_{\Leb^\infty(\Omega)}>1$. It follows that there exists some $\eps>0$ and a set $A$ of positive measure susch that $\rho|_{A} > 1+ \epsilon$. Thus, we have
\[
\frac{\beta^{-1}}{m-1}\intom{\rho^m} \geq \frac{\beta^{-1}}{m-1}\abs{A}(1+\epsilon)^m \stackrel{m \to \infty}{\to} \infty \, ,
\] 
and thus $\cF_\beta^m(\rho) \to \infty$ for all $\norm{\rho}_{\Leb^\infty(\Omega)}>1$. Now, let us assume that $\norm{\rho}_{\Leb^\infty(\Omega)}\leq 1$. This gives us
\begin{align}
\frac{\beta^{-1}}{m-1}\intom{\rho^m}  \leq \frac{\beta^{-1}}{m-1} \norm{\rho}_\infty^{m-1} \stackrel{m \to \infty}{\to}0 \, ,
\end{align}
and thus completes the construction of the recovery sequence.

\item $\Gamma \mhyphen \liminf$: Assume that there exists $\set{\rho_m}_{m\geq1}$ such that $\rho_m \weak \rho$ in 
$\Leb^\infty$-weak-$*$. For $W \in C^2(\Omega)$, the interaction energy is continuous and so we can disregard its behaviour. We start with the case in which $\norm{\rho}_{\Leb^\infty(\Omega)} \leq 1$. In this case the entropic term, $S_\beta^m(\rho_m)$, can be controlled from below by 0 and thus the $\Gamma \mhyphen \liminf$ holds trivially. The other case left to treat is when 
$\norm{\rho}_{\Leb^\infty(\Omega)}>1$. This implies again that there exists some $\eps>0$ and a set of positive measure $A$ such that $\rho|_A >1+ \eps$.  It follows from the weak-$*$ convergence that
\begin{align}
\lim_{m \to \infty}\int_{A} \rho_m \dx{x} = (1+ \epsilon)\abs{A} + \delta \, ,
\end{align}
for some fixed positive constant $\delta>0$ independent of $m$. We define the sets $A^+_m:= \set{x \in A: \rho_m >(1+ \eps)}$ and $A^-_m: = A\setminus A^+_m$. There also exists $N \in \N$ such that for $m \geq N$, 
$\int_{A} \rho_m \dx{x} \geq (1+ \epsilon)\abs{A} + \delta/2$. Thus, for $m \geq N$ we have that
\begin{align}
\int_{A^+_m} \rho_m \dx{x} + \int_{A^-_m} \rho_m \dx{x} \geq  (1+ \epsilon)\abs{A^+_m} + (1+ \epsilon)\abs{A^-_m} + \delta/2
\end{align}
from which it follows that 
\begin{align}
\int_{A^+_m} \rho_m \dx{x} \geq \delta/2 \, .
\end{align}
 This gives us the estimate we need on the entropic term since
\begin{align}
\frac{\beta^{-1}}{m-1} \intom{\rho_m^m} \geq & \frac{\beta^{-1}}{m-1} \int_{A^+_m}\rho_m^m \dx{x}\\ 
\geq &\frac{\beta^{-1}}{m-1} (1+ \epsilon)^{m-1} \int_{A^+_m} \rho_m \dx{x} \\ \geq & \frac{\beta^{-1}}{m-1} (1+ \epsilon)^{m-1} \delta/2\, .
\end{align}
Passing to the limit as $m \to \infty$, the result follows.
\end{enumerate}
\end{proof}
We would now like to understand how the presence of phase transitions for finite $m$ affects the minimisers of $\cF^\infty$. This is discussed in the next result.
\begin{theorem}[Minimisers of the mesa problem] \label{thm:mesapt}
Let $\cF^\infty :\cP(\Omega) \to (-\infty,+\infty]$ be as defined in~\eqref{eq:mesalim}. Then
\begin{tenumerate}
\item If $\abs{\Omega}<1$, $\cF^\infty \equiv +\infty$.\label{thm:mept!a}
\item If $\abs{\Omega}=1$, $\cF^\infty(\rho)< +\infty$ if and only if $\rho=\rho_\infty$. Thus, $\rho_\infty$ is the unique minimiser of $\cF^\infty$. \label{thm:mept!b}
\item If $\abs{\Omega}>1$ and $W \in \HH_s$ and $W \not \equiv0 $, $\rho_\infty$ is the unique minimiser of $\cF^\infty$. On the other hand if $W \in \HH_s^c$ there exists  $\cP(\Omega) \ni\rho \neq \rho_\infty$ such that $\rho$ is the minimiser of $\cF^\infty$ with $\cF^\infty(\rho) < \cF^\infty(\rho_\infty)$. Furthermore, there exists a sequence, $\set{\rho_m}_{m \geq 1}$ of nontrivial minimisers of $\cF_\beta^m$
such that $\rho_m \weak \rho$ in $\Leb^\infty$-weak-$*$ as $m \to \infty$. \label{thm:mept!c}
\end{tenumerate}
\label{thm:mept}
\end{theorem}
\begin{proof}
The proof of~\Dref{thm:mept!a} follows from the fact that if $\abs{\Omega}<1$, then for any $\rho \in \cP(\Omega) \cap \Leb^1(\Omega)$ there exists a set $A$ of positive measure such that $\rho(x) >1$ for all $x \in A$. Indeed, if this were not the case we would have that
\begin{align}
\intom{\rho} \leq \abs{\Omega} <1 \, ,
\end{align} 
which would be a contradiction. Thus, we have that $\norm{\rho}_{\Leb^\infty(\Omega)}>1$ for all $\rho \in \cP(\Omega) \cap \Leb^1(\Omega)$ and so $\cF^\infty \equiv \infty$.

The proof of~\Dref{thm:mept!b} is similar. If $\rho \neq \rho_\infty$, we can again find a set of positive measure $A$ such that $\rho(x)>1$ for all $x \in A$. We then repeat the same argument as in the previous case.

Assume now that $\abs{\Omega}>1$ and $W \in \HH_s, W \not \equiv 0$ (if $W$ is identically zero then clearly $\cF^\infty \equiv 0$). Since $W$ is mean-zero we have that 
\begin{align}
\cF^\infty(\rho_\infty)=0 \, .
\end{align}
On the other hand if $\cP(\Omega) \cap \Leb^\infty(\Omega) \ni \rho \neq \rho_\infty$, we know from~\cref{def:Hstab}, that
\begin{align}
\cF^\infty(\rho)=\frac{1}{2}\iintom{W(x-y) \rho(x) \rho(y)} >0 \, . 
\end{align}
Finally consider the case $W \in \HH_s^c$. Let $\beta>0$ be fixed and note that, since $\abs{\Omega}>1$, $\beta_\sharp^m \to 0$ as $m \to \infty$.  Clearly for $m$ large enough a nontrivial minimiser $\rho_m \in \cP(\Omega)$ exists for $\beta>0$ from the result of~\cref{prop:tp}. Consider the measure $\rho^\eps= \rho_\infty + \eps e_{k^\sharp}$ where $k^\sharp$ is as defined previously. We then have the following bound
\begin{align}
\cF_\beta^m(\rho_m) \leq \cF_\beta^m(\rho^\eps)= &  \cF_\beta^m(\rho_\infty) + \bra*{\beta^{-1}m \rho_\infty^{m-2} + \rho_\infty^{-1/2}\frac{\hat{W}(k^\sharp)}{\Theta(k^\sharp)} }\frac{\eps^2}{2}\norm{e_{k^\sharp}}_{\Leb^2(\Omega)}^2 \\
&+ \beta^{-1}m(m-2)\frac{\eps^{3}}{6}\intom{f^{m-3} e_{k^\sharp}^3} , 
 \end{align} 
 where the function $f(x) \in \bra{\rho_\infty, \rho^\eps(x)}$. Note that $\abs{f} \leq (\rho_\infty + \eps N_{k^\sharp})$. Thus, we have the bound
 \begin{align}
\cF_\beta^m(\rho_m) \leq \cF_\beta^m(\rho^\eps)\leq &  \cF_\beta^m(\rho_\infty) + \bra*{\beta^{-1}m \rho_\infty^{m-2} + \rho_\infty^{-1/2}\frac{\hat{W}(k^\sharp)}{\Theta(k^\sharp)} }\frac{\eps^2}{2}\norm{e_{k^\sharp}}_{\Leb^2(\Omega)}^2  \\
&+\beta^{-1}m(m-2)\frac{\eps^{3}}{6}(\rho_\infty + \eps N_{k^\sharp})^{m-3} N_{k^\sharp}^3 \abs{\Omega} \, ,
 \end{align}
 Additionally note  that if $\eps$ is small enough and $\rho_\infty<1$, the last term tends to $0$ as $m \to \infty$. Also since $W \in \HH_s^c$, the second term in the above expression is negative for $m$ large enough as $m \rho_\infty^m \to 0$ as $m \to \infty$. It follows from this that, for $m$ large enough, the following estimate holds
\begin{align}
\cF_\beta^m(\rho_m) \leq \cF_\beta^m(\rho^\eps)\leq &  \cF_\beta^m(\rho_\infty) -C_1 \eps^2 +C_2\eps^3  \, , \label{eq:Fgap}
\end{align}
where $C_1,C_2>0$ are independent of $m$. it hus follows from~\cref{thm:Gcon}, 
\eqref{eq:Fgap}, and the definition of $\Gamma$-convergence that
\begin{align}
\cF^\infty(\rho)< \cF^\infty(\rho_\infty) \, ,
\end{align}
where $\rho \in \cP(\Omega)$ is the minimiser of $\cF^\infty$. Thus, $\rho \neq \rho_\infty$ and the result follows.
 \end{proof}
\section{Numerical experiments}\label{numexp}
The numerical experiments in this section are meant to shed light on the qualitative features of the global bifurcation diagram of the system, while also serving as a source of possible conjectures that can be studied in future work. They were performed using a modified version on the numerical scheme in~\cite{CCH15}. 

\subsection{Discontinuous bifurcations for $m >2$ and $W=-\cos(2 \pi x/L)$}
\begin{figure}[]
\centering
\includegraphics[scale=0.6]{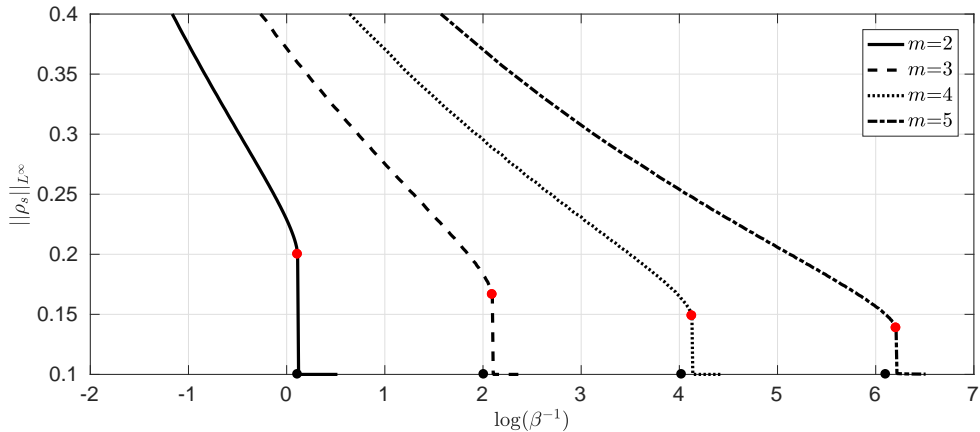}
\caption{Stationary solutions in the long-time limit for $m\geq2$ and $W=-\cos(2 \pi x/L)$. The black dot denotes the point of linear stability $\beta_\sharp^m$ while the red dot denotes the value of $\beta$ at which the support of the stationary solution is a strict subset of $\T$. Note that $\norm{\rho_{s}}_{\Leb^\infty}=0.1$ corresponds to the flat state $\rho_\infty$.}
\label{dctpmlarge}
\end{figure}
\cref{dctpmlarge} shows the branches of stationary solutions obtained in the long-time limit for $m \geq 2$
and $W=-\cos(2 \pi x/L)$. The black dot denotes the point of linear stability $\beta_\sharp^m$ while the 
red dot denotes the value of $\beta$ at which the support of the stationary solution is a strict subset of $\T$. Note that the diagram does not necessarily reflect the actual bifurcation diagram of the system as it is obtained from the long-time dynamics and thus will only see stable solutions. We already know that this choice of $W$ satisfies the conditions of~\cref{thm:bif} and so there will a bifurcation at $\beta_\sharp^m$ (the black points in~\cref{dctpmlarge}). One would expect this branch to turn to the right for $m \in (2,3)$ (cf.~\cref{critical}) and then turn back. We conjecture that the red points are all saddle-node bifurcations and correspond to discontinuous phase transitions for $m\geq2$ due to~\cref{lem:ctp@cs} and the fact that they lie ahead of the corresponding $\beta_\sharp^m$. 

\subsection{The mesa minimisers}
In~\cref{mesafig}, we plot the stationary solutions observed in the long-time limit for $m$ large and $\beta>\beta_c$. Since the stationary solutions are potentially minimisers of $\cF_\beta^m$ and the minimisers converge to the minimisers of $\cF^\infty$ as $m \to \infty$ (cf.~\cref{thm:mesa}), the plots in~\cref{mesafig} provide us with some information about the structure of the minimisers of the mesa problem. It seems to be that they converge to the indicator function of some fixed set. A natural next question one can ask is what happens to the continuity of phase transitions in the limit as $m \to \infty$.

\begin{figure}[]
\centering
\includegraphics[scale=0.6]{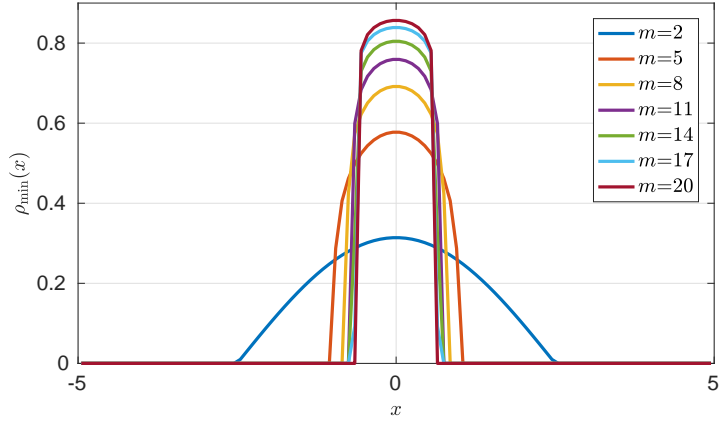}
\caption{Stationary solutions/minimisers for $m$ large and $\abs{\Omega}>1$. The limiting object seems to be the indicator function of some interval.}
\label{mesafig}
\end{figure}

\section{Proof of H\"older regularity}\label{sec:proofofholder}
We divide the proof into two parts. In~\cref{apriori}, we derive some a priori estimates that will be useful in the proof of regularity. In~\cref{proof}, we perform the so-called reduction of oscillation scheme and complete the proof of~\cref{holder}.  As mentioned earlier, readers interested only in bifurcations and phase transitions can skip directly to~\cref{bif}.

Before turning to the proof of~\cref{holder}, we introduce some notation. Since the equation~\eqref{eq:agdif} is invariant under translations of the co-ordinate axis, we define the parabolic cylinder
\begin{align}
Q(\tau,R)=\pra{-R,R}^d \times [-\tau,0] \, ,
\end{align}
centred at $(0,0)$ and note that we can move it to any point by adding $(x_0,t_0)$. We also used $K_R$
as a shorthand for $[-R,R]^d$. We denote the parabolic boundary by
\begin{align}
\partial_p Q(\tau,R)= \partial K_R \times (-\tau,0) \cup K_R \times \set{-\tau,0} \, . 
\end{align}
We use the following shorthand notation:
\[
w_+= \max(w,0), \qquad w_{-}=-\min(w,0), \qquad \rho^\ell_+= \min(\rho,\ell), \qquad \rho^\ell_-= -\min(-\rho,-\ell) \, .
\]
Additionally, we consider the cut-off functions $\zeta$ such that
\[
0 \leq \zeta \leq 1, \qquad \abs{\nabla \zeta}< +\infty,  \qquad \zeta(x,t)=0, x \notin  K_R \, .
\]
Through the rest of this section we will also use $f(x,t)$ to denote $W \star \rho(x,t)$. Note that
\begin{align}
\norm{\nabla f}_{\Leb^\infty(\Omega)} \leq \norm{\nabla W}_{\Leb^\infty(\Omega)}, \qquad \norm{D^2 f}_{\Leb^\infty(\Omega)} \leq \norm{D^2 W}_{\Leb^\infty(\Omega)} 
\end{align}
The reader should note that proof of regularity holds for any $f \in C^2(\Omega)$ that for which one can prove bounds of the kind shown above. We note before starting the proof that all estimates in the  proof have constants that depend continuously on $\beta>0$. Thus, the H\"older exponent $a$ and semi-norm $\abs{\rho}_a$ also depend continuously on $\beta>0$. 

\subsection{A priori estimates}\label{apriori}
 There are two a priori estimates that play a key role in the proof of H\"older  regularity: a Cacciopoli-type energy estimate and a logarithmic estimate. The proof of the energy estimate is essentially the same as~\cite[Proposition 2.4]{UrB08} and we state it without proof.
\begin{lemma}[Energy estimates]\label{lem:eest}
Pick $k, \ell \in \R_+$ and some cut-off function $\zeta$, such that $\zeta= 0$ on $\partial_p Q(\tau,R)$.  Then it holds for any weak solution of~\eqref{eq:agdif} that
\begin{align}~\label{eq:leep}
&\frac{1}{2} \pra*{\esup_{ t \in [-\tau,0]} \int_{K_R \times \set{t} } \bra{\rho^\ell_+ -k }_+^2 \zeta^2 \dx{x}  + \int_{Q(\tau,R)} (\rho^\ell_+)^{m-1 }\abs{\nabla (\rho^\ell_+ -k)_+ \zeta}^2 \dx{x} \dx{t} }  \\
\leq & \int_{Q(\tau,R)} (\rho^\ell_+ -k)_+^2 \zeta \zeta_t \dx{x} \dx{t} + 2(\ell -k)_+ \int_{Q(\tau,R)}(\rho- \ell)_+ \zeta \zeta_t \dx{x} \dx{t} \\&+ 2m\beta^{-1} \int_{Q(\tau,R)}(\rho^\ell_+ - k)_+^2(\rho^\ell_+)^{m-1}\abs{\nabla \zeta}^2 \dx{x} \dx{t}\\& +2m \beta^{-1} (\ell -k)_+ \int_{Q(\tau,R)} \bra*{\int_{\ell}^\rho s^{m-1} \dx{s}} (\abs{\nabla \zeta}^2 + \zeta \Delta \zeta) \chi_{\rho \geq \ell} \dx{x} \dx{t} \\&+  \int_{Q(\tau,R)} \abs{\nabla f}\abs{\zeta} \abs{\nabla \zeta} (\rho^\ell_+ - k)_+^2 \dx{x} \dx{t} +  \int_{Q(\tau,R)} \abs{\Delta f } (\rho^\ell_+ - k)_+^2  \zeta^2 \dx{x} \dx{t} \, .
\end{align}
Similarly we have,
\begin{align}~\label{eq:leem}
&\frac{1}{2} \pra*{\esup_{t\in [-\tau,0]} \int_{K_R \times \set{t} } \bra{\rho^\ell_- -k }_-^2 \zeta^2 \dx{x}  + \int_{Q(\tau,R)} (\rho^\ell_-)^{m-1 }\abs{\nabla (\rho^\ell_- -k)_- \zeta}^2 \dx{x} \dx{t} }   \\
\leq & \int_{Q(\tau,R)} (\rho^\ell_- -k)_-^2 \zeta \zeta_t \dx{x} \dx{t} + 2(\ell -k)_- \int_{Q(\tau,R)}(\rho- \ell)_- \zeta \zeta_t \dx{x} \dx{t} \\&+ 2m\beta^{-1} \int_{Q(\tau,R)}(\rho^\ell_- - k)_-^2(\rho^\ell_-)^{m-1}\abs{\nabla \zeta}^2 \dx{x} \dx{t} \\&-2m\beta^{-1} (\ell -k)_- \int_{Q(\tau,R)} \bra*{\int_{\ell}^\rho s^{m-1} \dx{s}} (\abs{\nabla \zeta}^2 + \zeta \Delta \zeta) \chi_{\rho \leq \ell} \dx{x} \dx{t} \\&+  \int_{Q(\tau,R)} \abs{\nabla f}\abs{\zeta} \abs{\nabla \zeta} (\rho^\ell_- - k)_-^2 \dx{x} \dx{t} +  \int_{Q(\tau,R)} \abs{\Delta f } (\rho^\ell_- - k)_-^2  \zeta^2 \dx{x} \dx{t} \, .
\end{align}
\end{lemma} 
\cre{We note that Urbano~\cite[Proposition 2.4]{UrB08} proves the above energy estimate for the $p$-Laplace equation, $\partial_t \rho- \Delta_p \rho =0$. The proof in our setting follows the same technique. We test the weak formulation in~\cref{thm:weak2} (see page~\pageref{thm:weak2}) against $\phi= ((\rho_{\pm}^\ell)_h-k)_{\pm}\zeta^2$, for some cut-off function $\zeta$ supported in $Q(\tau,R)$ and integrate by parts. Applying similar bounds as in~\cite[Proposition 2.4]{UrB08} and then passing to the limit as $h \to 0$, we obtain the desired energy estimate. We also refer the reader to~\cite[Proposition 2.7]{AR} where the proof of the energy estimate is carried out for the porous medium equation, $\partial_t \rho - \Delta \rho^m=0$, which is closer in structure to~\eqref{eq:agdif}.} We now move on to the logarithmic estimate. The proof of this needs to be adapted from the classical estimate in the presence of the drift term $\nabla \cdot (\nabla f \rho)$. Before stating and proving it, we introduce the following function
\begin{align}
\psi^{\pm}(s)= \psi^{\pm}_{k,c}(s) := \bra*{\ln \bra*{\frac{H^\pm_{s,k}}{(H^\pm_{s,k}+c) - (s-k)_\pm}   }}_+, \quad 0<c< H^\pm_{s,k} ,
\end{align}
where $s$ is a bounded, measurable function on $Q(\tau,R)$ and 
\begin{align}
H^\pm_{s,k}=  \esup_{Q(\tau,R)} \abs{(s-k)_\pm} \,.
\end{align} 
The function has certain useful properties, namely,
\begin{align}
0 \leq \psi^\pm(s) \\
(\psi^+)'(s) \geq 0, \quad (\psi^-)'(s) \leq 0 \\
(\psi^{\pm})''= ((\psi^\pm)')^2 \, .
\end{align}
We also need to define the Steklov  average for any $\rho \in \Leb^1(\Omega \times [0,T])$ for any $0<h<T$ as follows
\begin{align}
\rho_h:=
\begin{cases}
h^{-1}\int_t^{t+h} \rho(\cdot, \tau) \dx{\tau} &  0 \leq t \leq T-h \\
0 & \textrm{otherwise}
\end{cases}
\, .
\end{align}
The Steklov average has certain nice properties which we state without proving.
\begin{lemma}\cite[Lemma 2.2]{UrB08}
Let $\rho \in \Leb^q([0,T]; \Leb^r(\Omega))$ then $\rho_h$ converges to $\rho$ in $\rho \in \Leb^q([0,T]; \Leb^r(\Omega))$ as $h \to 0$ for $q,r \in (1,\infty)$. Additionally, if $\rho \in C([0,T]; \Leb^2(\Omega))$, then $\rho_h(\cdot,t)$ converges to $\rho(\cdot,t)$  in $\Leb^q(\Omega)$ for $t \in [0,T]$.
\end{lemma}
Using this we have the following alternative notion of a weak solution of
\begin{defn}\label{thm:weak2}
A weak solution of~\eqref{eq:agdif} is a bounded measurable function 
\[
\rho \in C([0,T]; \Leb^2(\Omega)) 
\]
with
\[
\rho^m \in \Leb^2([0,T]; \SobH^1(\Omega)) \, ,
\]
such that
\begin{align}
\int_{\Omega \times \set{t}} \partial_t (\rho_h) \phi + m\beta^{-1} (\rho^{m-1}\nabla \rho)_h \cdot \nabla \phi +
(\rho \nabla W \star \rho)_h \cdot \nabla \phi \, \dx{x}=0 \, ,
\label{eq:steklov}
\end{align}
for all $\phi \in \SobH^1_0(\Omega)$, $h \in (0,T)$, $t \in (0,T]$ and $\rho(x,0)=\rho_0$.
\end{defn}
\begin{proposition}\cite{UrB08}
The notion of weak solution introduced in~\cref{thm:weak1} and~\cref{thm:weak2} are equivalent.
\end{proposition}

\begin{lemma}[Logarithmic estimates] \label{lem:logest}
Let $\rho$ be a nonnegative weak solution of~\eqref{eq:agdif} and $\zeta$ be a time-independent cut-off function, then it holds that
\begin{align}
\int_{K_R \times \set{t}} ((\psi^\pm)^2)(\rho) \zeta^2\dx{x} \leq & \int_{K_R \times \set{-\tau}} ((\psi^\pm)^2)(\rho) \zeta^2\dx{x}  - 2m\beta^{-1}\int_{-\tau}^t \int_{K_R \times \set{s}}  (\rho^{m-1}\abs{\nabla \rho}^2  ((\psi^\pm)'(\rho))^2 \zeta^2) \dx{x} \dx{s} \\&+2m\beta^{-1} \int_{-\tau}^t \int_{K_R \times\set{s}} \rho^{m-1} \psi^\pm(\rho)\abs{\nabla \zeta}^2 \dx{x} \dx{s}\\& +2\int_{-\tau}^t \int_{K_R \times \set{s}}  \rho\abs{\nabla f} \abs{ \nabla \rho} \abs{((\psi^\pm)'(\rho))^2(1+(\psi^\pm(\rho))} \zeta^2 \dx{x} \dx{s} \\&+ 4\int_{-\tau}^t \int_{K_R \times \set{s}}  \rho\abs{\nabla f} \abs{\nabla \zeta}\abs{ ((\psi^\pm)'(\rho)) \psi^\pm(\rho)}  \abs{\zeta} \dx{x} \dx{s} \, .
\end{align}
for any $-\tau \leq t \leq 0$.
\end{lemma}
\begin{proof}
We start by testing~\eqref{eq:steklov} against $((\psi^\pm)^2)'(\rho_h) \zeta^2$ and integrating by parts to obtain
\begin{align}
\int_{\Omega \times \set{t}} \partial_t (\rho_h) ((\psi^\pm)^2)'(\rho_h) \zeta^2  + m\beta^{-1} (\rho^{m-1}\nabla \rho)_h \cdot \nabla(((\psi^\pm)^2)'(\rho_h) \zeta^2) +
(\rho \nabla f)_h \cdot \nabla (((\psi^\pm)^2)'(\rho_h) \zeta^2) \, \dx{x}=0 \, ,
 \label{eq:logtest}
\end{align}
Consider the first term on the LHS and integrating from $-\tau$ to $t$
\begin{align}
 &\int_{-\tau}^t\int_{\Omega \times \set{s}} \partial_s (\rho_h) ((\psi^\pm)^2)'(\rho_h) \zeta^2\dx{x}\dx{s}\\ =& \int_{-\tau}^t\int_{\Omega \times \set{s}} \partial_s  ((\psi^\pm)^2)(\rho_h) \zeta^2\dx{x}\dx{s} \\
 =& \int_{\Omega \times \set{t}} ((\psi^\pm)^2)(\rho_h) \zeta^2\dx{x} - \int_{\Omega \times \set{-\tau}} ((\psi^\pm)^2)(\rho_h) \zeta^2\dx{x} \, .
\end{align} 
Passing to the limit as $h \to 0$ we obtain that
\begin{align}
\int_{-\tau}^t\int_{\Omega \times \set{s}} \partial_s (\rho_h) ((\psi^\pm)^2)'(\rho_h) \zeta^2\dx{x}\dx{s} 
\to  \int_{\Omega \times \set{t}} ((\psi^\pm)^2)(\rho) \zeta^2\dx{x} - \int_{\Omega \times \set{-\tau}} ((\psi^\pm)^2)(\rho) \zeta^2\dx{x} \, .
\end{align}
Now consider the second term on the LHS of~\eqref{eq:logtest} (after passing to the limit as $h \to 0$)
\begin{align}
&\beta^{-1}\int_{-\tau}^t \int_{\Omega \times \set{s}} m (\rho^{m-1}\nabla \rho) \cdot \nabla(((\psi^\pm)^2)'(\rho) \zeta^2) \dx{x} \dx{s}\\ =&  2m\beta^{-1}\int_{-\tau}^t \int_{\Omega \times \set{s}}  (\rho^{m-1}\abs{\nabla \rho}^2  ((\psi^\pm)'(\rho))^2(1+(\psi^\pm(\rho)) \zeta^2) \dx{x} \dx{s}\\& +4 m\beta^{-1} \int_{-\tau}^t \int_{\Omega \times \set{s}}  (\rho^{m-1}\nabla \rho (\psi^\pm)'(\rho)\psi^\pm(\rho) \zeta \cdot \nabla \zeta) \dx{x} \dx{s}  \\
\geq &   2m\beta^{-1}\int_{-\tau}^t \int_{\Omega \times \set{s}}  (\rho^{m-1}\abs{\nabla \rho}^2  ((\psi^\pm)'(\rho))^2 \zeta^2) \dx{x} \dx{s} \\&- 2m\beta^{-1} \int_{-\tau}^t \int_{\Omega \times\set{s}} \rho^{m-1} \psi^\pm(\rho)\abs{\nabla \zeta}^2 \dx{x} \dx{s} \, ,
\end{align} 
where the last expression follows from Youngs inequality. Finally we consider the last term on the LHS of~\eqref{eq:logtest} (after passing to the limit as $h \to 0$)
\begin{align}
&\int_{-\tau}^t \int_{\Omega \times \set{s}}  (\rho\nabla f) \cdot \nabla(((\psi^\pm)^2)'(\rho) \zeta^2) \dx{x} \dx{s}\\ = & 2\int_{-\tau}^t \int_{\Omega \times \set{s}}  \rho\nabla f \cdot \nabla \rho ((\psi^\pm)'(\rho))^2(1+(\psi^\pm(\rho)) \zeta^2 \dx{x} \dx{s} \\&+ 4\int_{-\tau}^t \int_{\Omega \times \set{s}}  \rho\nabla f \cdot \nabla \zeta ((\psi^\pm)'(\rho)) \psi^\pm(\rho) \zeta  \dx{x} \dx{s} \\
\geq & -2\int_{-\tau}^t \int_{\Omega \times \set{s}}  \rho\abs{\nabla f} \abs{ \nabla \rho} \abs{((\psi^\pm)'(\rho))^2(1+(\psi^\pm(\rho))} \zeta^2 \dx{x} \dx{s} \\&- 4\int_{-\tau}^t \int_{\Omega \times \set{s}}  \rho\abs{\nabla f} \abs{\nabla \zeta}\abs{ ((\psi^\pm)'(\rho)) \psi^\pm(\rho)}  \abs{\zeta} \dx{x} \dx{s} \, .
\end{align}
Putting it all together we obtain
\begin{align}
\int_{\Omega \times \set{t}} ((\psi^\pm)^2)(\rho) \zeta^2\dx{x} \leq & \int_{\Omega \times \set{-\tau}} ((\psi^\pm)^2)(\rho) \zeta^2\dx{x}  - 2m\beta^{-1}\int_{-\tau}^t \int_{\Omega \times \set{s}}  (\rho^{m-1}\abs{\nabla \rho}^2  ((\psi^\pm)'(\rho))^2 \zeta^2) \dx{x} \dx{s} \\&+2m\beta^{-1} \int_{-\tau}^t \int_{\Omega \times\set{s}} \rho^{m-1} \psi^\pm(\rho)\abs{\nabla \zeta}^2 \dx{x} \dx{s}\\& +2\int_{-\tau}^t \int_{\Omega \times \set{s}}  \rho\abs{\nabla f} \abs{ \nabla \rho} \abs{((\psi^\pm)'(\rho))^2(1+(\psi^\pm(\rho))} \zeta^2 \dx{x} \dx{s} \\&+ 4\int_{-\tau}^t \int_{\Omega \times \set{s}}  \rho\abs{\nabla f} \abs{\nabla \zeta}\abs{ ((\psi^\pm)'(\rho)) \psi^\pm(\rho)}  \abs{\zeta} \dx{x} \dx{s} \, .
\end{align}
Taking into account the support of $\zeta$, one obtains the result of the lemma.
\end{proof}
\subsection{Proof of~\cref{holder}}\label{proof}
We now get to the meat of the regularity argument, i.e. the reduction of oscillation. We assume again that $\rho$ is a nonnegative weak solution of~\eqref{eq:agdif}. We pick a cylinder $Q(4R^{2- \eps},2R)$ that lies inside $\Omega_T$ (shifted to $(0,0)$) for $0<R<1$. Then we can define
\begin{align}
\mu^+ = \esup\limits_{Q(4R^{2- \eps},2R)} \rho ,  \qquad \mu^- = \einf\limits_{Q(4R^{2- \eps},2R)} \rho ,
\end{align}
along with
\[
\omega= \mu^+ -\mu^- = \eosc\limits_{Q(4R^{2- \eps},2R)} \rho \, .
\]
We then define the rescaled cylinder
\begin{align}
Q(\omega^{1-m}R^{2},R) \subset Q(4R^{2- \eps},2R) \, ,
\end{align} 
which holds true if
\begin{align}
\alpha \omega^{m-1} > R^\eps \, .
\label{eq:driftalpha}
\end{align}
For a fixed $\eps>0, \alpha \in (0,1)$ if the above inequality does not hold true for any $R$ that can be made arbitrarily small, it follows that $\omega$ is comparable to the radius of the cylinder and thus we have H\"older continuity already. The proof of this statement is by contradiction. Let $\omega_R:= \eosc\limits_{Q(4R^{2- \eps},2R)} \rho$. Then for any point $(x,t) \in \Omega_T$  we set $R:= d_{\T^d}(x,0) + \abs{t}^{1/2}$, the parabolic distance to the origin. Thus, we have
\begin{align}
\abs{\rho(x,t)-\rho(0,0)} \leq \omega_R \leq\alpha^{-\frac{1}{m-1}} R^{\frac{\eps}{m-1}} = \alpha^{-\frac{1}{m-1}}  \bra*{d_{\T^d}(x,0) + \abs{t}^{1/2}}^{\frac{\eps}{m-1}} \, .
\end{align}
 We will specify the value of $\alpha$ later.  We thus have by this inclusion that
\[
\eosc_{Q(w^{1-m}R^2,R)} \rho \leq \omega \, .
\]
We will also assume throughout the remainder of this proof that $\mu^- <\omega/4 $, as otherwise the equation is uniformly parabolic in $Q(4R^{2- \eps}, 2R)$. Before we proceed we pick some $\nu_0 \in (0,1)$ and divide our analysis into two cases.
\paragraph{\bf Case 1} 
\begin{align}\label{eq:case1}
\frac{\abs{(x,t) \in Q(\omega^{1-m}R^{2},R) : \rho < \mu^- + \omega/2}}{\abs{Q(\omega^{1-m}R^{2},R) }} \leq \nu_0 \, ,
\end{align}
or 
\paragraph{\bf Case 2} 
\[
\frac{\abs{(x,t) \in Q(\omega^{1-m}R^{2},R) : \rho \geq \mu^- + \omega/2}}{\abs{Q(\omega^{1-m}R^{2},R) }} <1- \nu_0 \, ,
\]
or equivalently
\begin{align}
\frac{\abs{(x,t) \in Q(\omega^{1-m}R^{2},R) : \rho \geq \mu^+ - \omega/2}}{\abs{Q(\omega^{1-m}R^{2},R) }} <1- \nu_0 \, .
\label{eq:case2}
\end{align}
We now treat the two cases independently.

\subsubsection{Reduction of oscillation in case 1}
In the first case, we start by proving the following result.
\begin{lemma}\label{lem:redcase1}
Assume that $\mu^-< \omega/4$ and that~\eqref{eq:case1} holds for some $\nu_0$(to be chosen), then
\[
\rho(x,t) > \mu^- + \frac{\omega}{4} \textrm{ a.e. in } Q\bra*{\omega^{1-m}\bra*{\frac{R}{2}}^2, \frac{R}{2}} \, .
\]
\end{lemma}
\begin{proof}
We start by considering the sequence
\[
R_n= \frac{R}{2} + \frac{R}{2^{n+1}} \qquad n=0,1, \cdots
\]
such that $R_0=R$ and $R_n \to R/2$ as $n \to \infty$. We then construct a sequence of nested shrinking cylinders $Q(\omega^{1-m}R_n^2, R_n)$ along with cut-off functions $\zeta_n$ satisfying
\cb{
\begin{align}
0 \leq \zeta_n \leq 1, &\qquad \zeta_n=1 \textrm{ in } Q(\omega^{1-m}R^2_{n+1},R_{n+1}),
\qquad 
\zeta_n=0 \textrm{ on } \cb{\partial_p} Q(\omega^{1-m}R^2_{n},R_{n}),  \\
\abs{\nabla \zeta_n} \leq \frac{2^{n+2}}{R}, &\qquad 0 \leq (\zeta_n)_t \leq \frac{2^{2n+2}}{R^2} \omega^{m-1 },  \qquad
\Delta \zeta_n  \leq \frac{2^{2n+5}}{R^2}.
\end{align}
}
We now apply the energy estimate of~\cref{lem:eest} in $Q(\omega^{1-m}R_n^2, R_n)$ with $\ell= \mu^- + \omega/4$, and $k_n=\mu^- + \omega/4 + \omega/(2^{n+1})$ for the function $(\rho^\ell_- -k_n )_-$. We will bound the terms on the LHS and RHS separately. Considering first the terms on the LHS we have
\begin{align}
&\frac{1}{2} \left[\esup\limits_{- R_n^2 \omega^{1-m} < t <0 } \int_{K_{R_n} \times \set{t} } \bra{\rho^\ell_- -k_n }_-^2 \zeta_n^2 \dx{x}  \right.\\&\left.+ \int_{Q(\omega^{1-m}R_n^2, R_n)}(\rho^\ell_-)^{m-1 } \abs{\nabla (\rho^\ell_- -k_n)_- \zeta_n}^2 \dx{x} \dx{t}  \right]  \\
\geq & 2^{1-2m} \left[\esup_{- R_n^2 \omega^{1-m} < t <0 } \int_{K_{R_n} \times \set{t} } \bra{\rho^\ell_- -k_n }_-^2 \zeta_n^2 \dx{x} \right. \\& \left.+ \omega^{m-1}\int_{Q(\omega^{1-m}R_n^2, R_n)}\abs{\nabla (\rho^\ell_- -k_n)_- \zeta_n}^2 \dx{x} \dx{t} \right] \, , 
\end{align}
where we have used the fact that $\rho^\ell_-= \max\bra*{\rho, \mu^- + \omega/4} \geq \mu^- + \omega/4 \geq \omega/4$. For the RHS we first note the following facts:
\begin{enumerate}
\item $0 \leq \mu^- \leq \omega/4 $ which implies that $\rho \leq 5 \omega/4$, $\ell \leq \omega/2$, and $\rho^\ell_- \leq 5 \omega/4$  .
\item $\ell = \mu^- + \omega/4< k_n$ which implies that $\chi_{[\rho \leq \ell]} \leq \chi_{[\rho \leq k_n]}=\chi_{[(\rho- k_n)_->0]} $.
\item If $\rho^\ell_-= \rho$, then $\chi_{[(\rho^\ell_-- k_n)_->0]}=\chi_{[(\rho- k_n)_->0]}  $. On the other hand if  $\rho^\ell_-= \ell$, we have that $\rho \leq \ell < k_n$ we have that $\chi_{[(\rho- k_n)_->0]}  =0 =\chi_{[(\ell- k_n)_->0]} = \chi_{[(\rho^\ell_- - k_n)_->0]} $.
\item $(l-k_n)_{-}=\omega/(2^{n+1}) \leq \omega/2$, $(\rho^\ell_{-}-k_n)_{-}\leq \omega/2^{n+1} \leq \omega/2$,  $(\rho- \ell)_{-} \leq \omega/4$ .
\end{enumerate}
We now proceed to bound individual terms on the RHS of~\eqref{eq:leem}. For the first term we have:
\begin{align}
 &\int_{Q(\omega^{1-m}R_n^2, R_n)} (\rho^\ell_- -k_n)_-^2 \zeta_n (\zeta_n)_t \dx{x} \dx{t} +
+ 2(\ell -k_n)_- \int_{Q(\omega^{1-m}R_n^2, R_n)}(\rho- \ell)_- \zeta_n (\zeta_n)_t \dx{x} \dx{t} \\
\leq & \frac{\omega^2}{2}\omega^{m-1} \frac{2^{2n+2}}{R^2} \int_{Q(\omega^{1-m}R_n^2, R_n)} \chi_{[(\rho^\ell_-- k_n)_->0]}\dx{x} \dx{t} \, .
\end{align}
For the second term:
\begin{align}
 &2m\beta^{-1} \int_{Q(\omega^{1-m}R_n^2, R_n)}(\rho^\ell_- - k_n)_-^2(\rho^\ell_-)^{m-1}\abs{\nabla \zeta_n}^2 \dx{x} \dx{t} \\\leq & m \beta^{-1}\bra*{\frac{5}{4}}^{m-1}\omega^2 \omega^{m-1}\frac{2^{2n+3}}{R^2}
 \int_{Q(\omega^{1-m}R_n^2, R_n)} \chi_{[(\rho^\ell_-- k_n)_->0]}\dx{x} \dx{t} \, .
\end{align}
For the third term:
\begin{align}
& -2m\beta^{-1} (\ell -k_n)_- \int_{Q(\omega^{1-m}R_n^2, R_n)} \bra*{\int_{\ell}^\rho s^{m-1} \dx{s}} (\abs{\nabla \zeta_n}^2 + \zeta_n \Delta \zeta_n) \chi_{\rho \leq \ell} \dx{x} \dx{t}
\\\leq & m\beta^{-1}\frac{\omega^2}{4}\omega^{m-1}\frac{2^{2n+5}}{R^2}\int_{Q(\omega^{1-m}R_n^2, R_n)} \chi_{[(\rho^\ell_-- k_n)_->0]}\dx{x} \dx{t} \, .
\end{align}
For the final two terms we have:
\begin{align}
&\int_{Q(\omega^{1-m}R_n^2, R_n)} \abs{\nabla f}\abs{\zeta_n} \abs{\nabla \zeta_n} (\rho^\ell_- - k_n)_-^2 \dx{x} \dx{t} +  \int_{Q(\omega^{1-m}R_n^2, R_n)} \abs{\Delta f } (\rho^\ell_- - k_n)_-^2  \zeta_n^2 \dx{x} \dx{t}
\\\leq & \bra*{\frac{2^{n+2}}{R}\norm{\nabla f}_{\Leb^\infty(\Omega)} + \norm{\Delta f}_{\Leb^\infty(\Omega)} }\frac{\omega^2}{4} \int_{Q(\omega^{1-m}R_n^2, R_n)} \chi_{[(\rho^\ell_-- k_n)_->0]}\dx{x} \dx{t} \\
=& \frac{2^{2n}}{R^2} \omega^{m-1}\bra*{\frac{\omega^{1-m}R}{2^{n-2}}\norm{\nabla f}_{\Leb^\infty(\Omega)} + \norm{\Delta f}_{\Leb^\infty(\Omega)}\omega^{1-m}R^2 2^{-2n} }\frac{\omega^2}{4} \int_{Q(\omega^{1-m}R_n^2, R_n)} \chi_{[(\rho^\ell_-- k_n)_->0]}\dx{x} \dx{t} 
\\\leq & \frac{2^{2n}}{R^2} \omega^{m-1}\bra*{4L^{1-\eps}\norm{\nabla f}_{\Leb^\infty(\Omega)} + \norm{\Delta f}_{\Leb^\infty(\Omega)} L^{2-\eps }} \frac{\omega^2}{4} \int_{Q(\omega^{1-m}R_n^2, R_n)} \chi_{[(\rho^\ell_-- k_n)_->0]}\dx{x} \dx{t}  \,,
\end{align}
where in the last step we have used the fact that $R^\eps\omega^{1-m}<\alpha<1$ and that $R <L$. Putting the bounds for the LHS and RHS of~\eqref{eq:leem} together we obtain
\begin{align}
&\pra*{\esup_{- R_n^2 \omega^{1-m} < t <0 } \int_{K_{R_n} \times \set{t} } \bra{\rho^\ell_- -k }_-^2 \zeta_n^2 \dx{x}  + \omega^{m-1}\int_{Q(\omega^{1-m}R_n^2, R_n)}\abs{\nabla (\rho^\ell_- -k)_- \zeta_n}^2 \dx{x} \dx{t} } \\\leq & C\bra*{m, L,\beta,\norm{\nabla f}_{\Leb^\infty(\Omega)}, \norm{\Delta f}_{\Leb^\infty(\Omega)}} \frac{2^{2n}}{R^2}\omega^{m-1} \frac{\omega^2}{4} \int_{Q(\omega^{1-m}R_n^2, R_n)} \chi_{[(\rho^\ell_-- k_n)_->0]}\dx{x} \dx{t} \, .
\end{align}
Let $\bar{t}=\omega^{m-1}t$ and define the following rescaled functions
\begin{align}
\bar{\rho}^\ell_-( \cdot,\bar{t}) =\rho^\ell_-( \cdot,t) , \qquad \bar{\zeta_n}( \cdot,\bar{t}) =\zeta_n( \cdot,t)  \, .
\end{align}
In these new variables the inequality simplifies to
\begin{align}
&\pra*{\esup_{- R_n^2  < \bar{t} <0 } \int_{K_{R_n} \times \set{\bar{t}} } \bra{\bar{\rho}^\ell_- -k_n }_-^2 \bar{\zeta_n}^2 \dx{x}  + \int_{Q(R_n^2, R_n)}\abs{\nabla (\bar{\rho}^\ell_- -k_n)_- \bar{\zeta_n}}^2 \dx{x} \dx{t} } \\\leq & C \frac{2^{2n}}{R^2} \frac{\omega^2}{4} A_n \, ,
\label{eq:rescaledineq}
\end{align}
where
\begin{align}
A_n:= \int_{Q(R_n^2, R_n)} \chi_{[(\bar{\rho}^\ell_-- k_n)_->0]}\dx{x} \dx{t} \, .
\end{align}
Furthermore we have
\begin{align}
\frac{1}{2^{2n+2}}\frac{\omega^2}{4}A_{n+1}=&\abs{k_n-k_{n+1}}^2 A_{n+1} \\
=&\int_{Q(R_{n+1}^2, R_{n+1})} \abs{k_n-k_{n+1}}^2\chi_{[(\bar{\rho}^\ell_-- k_{n+1})_->0]}\dx{x} \dx{t} \\
\leq & \int_{Q(R_{n+1}^2, R_{n+1})} \abs{k_n-\bar{\rho}^\ell_-}^2\chi_{[(\bar{\rho}^\ell_-- k_{n+1})_->0]}\dx{x} \dx{t} \\
\leq & \norm*{\bra*{k_n-\bar{\rho}^\ell_-}_-}^2_{\Leb^2(Q(R_{n+1}^2, R_{n+1}))} \\
\leq &  C_d A_n^{2/(2+d)}\norm{\bra*{k_n-\bar{\rho}^\ell_-}_-}^2_{V^2(Q(R_{n+1}^2, R_{n+1}))} \, ,
\end{align}
where in the last step we have used the embedding into the parabolic space $V^2$(cf.~\cref{embedding}). Thus, we have
\begin{align}
\frac{1}{2^{2n+2}}\frac{\omega^2}{4}A_{n+1} \leq & C_d \pra*{\esup_{- R_{n+1}^2  < \bar{t} <0 } \int_{K_{R_{n+1}} \times \set{\bar{t}} } \bra{\bar{\rho}^\ell_- -k }_-^2  \dx{x}  + \int_{R_{n+1}^2, R_{n+1})}\abs{\nabla (\bar{\rho}^\ell_- -k)_- }^2 \dx{x} \dx{t} }  \\
\leq & C_d A_n^{2/(2+d)} \pra*{\esup_{- R_n^2  < \bar{t} <0 } \int_{K_{R_n} \times \set{\bar{t}} } \bra{\bar{\rho}^\ell_- -k }_-^2 \bar{\zeta_n}^2 \dx{x}  + \int_{Q(R_n^2, R_n)}\abs{\nabla (\bar{\rho}^\ell_- -k)_- \bar{\zeta_n}}^2 \dx{x} \dx{t} } \\\leq & C \frac{2^{2n}}{R^2} \frac{\omega^2}{4} A_n^{1+ 2/(d+2)} \, ,
\end{align}
where we have used the fact that $\bar{\zeta_n}=1$ on $Q(R_{n+1}^2,R_{n+1})$ and have used~\eqref{eq:rescaledineq}.  Thus, we have
\begin{align}
\frac{A_{n+1}}{\abs*{Q(R_{n+1}^2, R_{n+1})}} \leq & C\abs*{Q(R_{n+1}^2, R_{n+1})}^{2/(2+d)}\frac{4^{2n+1}}{R^2}\bra*{ \frac{A_{n}}{\abs*{Q(R_{n+1}^2, R_{n+1})}}}^{1+2/(d+2)} \\
\leq & C 4^{2n} \bra*{\frac{\abs*{Q(R_{n}^2, R_{n})}}{\abs*{Q(R_{n+1}^2, R_{n+1})}} \frac{A_{n}}{\abs*{Q(R_{n}^2, R_{n})}}}^{1+2/(d+2)} \leq C 4^{2n} \bra*{ \frac{A_{n}}{\abs*{Q(R_{n}^2, R_{n})}}}^{1+2/(d+2)}  \, ,
\end{align}
where we use the fact that $\abs*{Q(R_{n}^2, R_{n})}= R_{n+1}^{d+2} \leq R^{d+2}$ and $R_n/R_{n+1} \leq 2$.
Setting 
\begin{align}
X_{n}:=\bra*{ \frac{A_{n}}{\abs*{Q(R_{n}^2, R_{n})}}} \, ,
\end{align}
we have the recursive inequality
\begin{align}
X_{n+1} \leq C 4^{2n} X_n^{1+ 2/(2+d)} \, ,
\end{align}
with the constant $C$ independent of $\omega, R,n$ and dependent only $d,m,\beta,f$. Setting $\nu_0=C^{-(d+2)/2} 4^{-(d+2)^2/2}$, we see that $X_0 \leq \nu_0$ is equivalent~\eqref{eq:case1} to being satisfied with constant $\nu_0$, since $k_0= \omega/2$. Thus, for this choice, $X_n \to 0$ by the geometric convergence lemma (cf.~\cref{convergence}). It follows then, after changing variables, that $\rho^\ell_- > \mu^- + \omega/4$ a.e. in $Q\bra*{\omega^{1-m}\bra*{\frac{R}{2}}^2, \frac{R}{2}}$. The result follows  by noting that $\rho^\ell_- > \mu^- + \omega/4=\ell$ implies that $\rho^\ell_- =\rho$.
\end{proof}
\begin{corollary}[Reduction of oscillation in case 1]
Assume that~\eqref{eq:case1} holds with constant $\nu_0$ as specified in the proof of~\cref{lem:redcase1}. Then there exists a $\sigma_1 \in(0,1)$, independent of $\omega$, $R$, such that
\begin{align}
\eosc_{Q\bra*{\omega^{1-m}\bra*{\frac{R}{2}}^2, \frac{R}{2}}}\rho \leq \sigma_1 \omega \, .
\end{align}
\end{corollary}
\begin{proof}
We have by the result of the previous lemma that
\begin{align}
\einf_{Q\bra*{\omega^{1-m}\bra*{\frac{R}{2}}^2, \frac{R}{2}}}\rho \geq \mu^- + \omega/4 \, .
\end{align}
Thus, we have that
\begin{align}
\eosc_{Q\bra*{\omega^{1-m}\bra*{\frac{R}{2}}^2, \frac{R}{2}}}\rho =&\esup_{Q\bra*{\omega^{1-m}\bra*{\frac{R}{2}}^2, \frac{R}{2}}}\rho -\einf_{Q\bra*{\omega^{1-m}\bra*{\frac{R}{2}}^2, \frac{R}{2}}}\rho \\
\leq & \mu^+ - \mu^- - \omega/4  \\
\leq & \frac{3}{4} \omega \, .
\end{align}
Thus, the result holds with $\sigma_1=\frac{3}{4}$.
\end{proof}
\subsubsection{Reduction of oscillation in case 2}
We now assume that~\eqref{eq:case2} holds but with the constant $\nu_0$ fixed from the previous argument. We argue now that if~\eqref{eq:case2} is satisfied then there exists some $t_0$,
\begin{align}
t_0 \in \pra*{-\omega^{1-m}R^2, -\frac{\nu_0}{2}\omega^{1-m}R^2} \, ,
\end{align}
such that
\begin{align}
\abs*{\set*{x \in K_{R}: \rho(x,t_0) > \mu^+ -\frac{\omega}{2}}} \leq \frac{1- \nu_0}{1- \nu_0/2}\abs{K_R} \, .
\end{align}
We prove this by contradiction. Assume this is not the case then
\begin{align}
&\abs*{\set*{x \in Q \bra*{\omega^{1-m}R^2,R}: \rho(x,t)> \mu^+ - \frac{\omega}{2}}}\\
\geq & \int_{- \omega^{1-m}R^2}^{-\frac{\nu_0}{2}\omega^{1-m}R^2} \abs{x \in K_R: \rho(x,s)> \mu^+- \omega/2} \dx{s} \\
 > & \bra*{-\frac{\nu_0}{2}\omega^{1-m}R^2 + \omega^{1-m}R^2} \bra*{\frac{1- \nu_0}{1- \nu_0/2}}\abs{K_R} \\
=& (1- \nu_0) \abs*{Q(\omega^{1-m}R^2,R)} \, ,
\end{align}
which contradicts~\eqref{eq:case2}. We now proceed to prove the following lemma.
\begin{lemma} \label{lem:driftest}
Assume that~\eqref{eq:case2} holds. Then there exists a $q \in \N$, depending only on the data, such that
\begin{align}
\abs*{\set*{x \in K_R: \rho(x,t) > \mu^+ - \frac{\omega}{2^q}}} \leq \bra*{1- \bra*{\frac{\nu_0}{2}}^2}\abs{K_R} \, ,
\end{align}
for all $t \in \pra*{-\frac{\nu_0}{2}\omega^{1-m}R^2,0}$ and $\alpha$ in~\eqref{eq:driftalpha} chosen to be small, depending only on $\nu_0$, $m$, $d$, $\beta$, $W$, $M$ but independent of $R$ and $\omega$.
\end{lemma}
\begin{proof}
The proof of this lemma relies on the~\cref{lem:logest} with the function $\psi^+(u)$ on the cylinder $Q(-t_0,R)$. We choose
\begin{align}
k = \mu^+ - \frac{\omega}{2}, \qquad c= \frac{\omega}{2^{n+1}} \, ,
\end{align}
where the constant $n>1$ will be chosen later. It is fine to apply it to this function as we can assume that
\begin{align}
H^+_{\rho,k}= \esup_{Q(-t_0,R)}\abs*{\bra*{\rho-\mu^+ + \frac{\omega}{2}}_+} > \frac{\omega}{4} \geq \frac{\omega}{2^{n+1}} \, ,
\end{align}
otherwise the proof of the lemma would be complete with $q=2$. Indeed, we would have for all $t \in [t_0,0]$:
\begin{align}
\abs*{\set*{x \in K_R: \rho(x,t) > \mu^+ - \frac{\omega}{4}}} = \abs*{\set*{x \in K_R: \rho(x,t) -\mu^+ + \frac{\omega}{2} >  \frac{\omega}{4}} } =0 \, .
\end{align}
Before we write down the inequality, we need to further understand the properties of the function $\psi^+(\rho)$ defined on the cylinder $Q(-t_0,R)$. Note first that
\begin{align}
\psi^+(\rho)=
\begin{cases}
\ln \bra*{\frac{H^+_{\rho,k}}{H^+_{\rho,k} - \rho + k + \frac{\omega}{2^{n+1}} }} & \rho >  k + \frac{\omega}{2^{n+1}} \\
0 &  \rho \leq  k + \frac{\omega}{2^{n+1}}
\end{cases}
\, .
\end{align}
Furthermore in $Q(-t_0,R)$, we have that
\begin{align}
\rho -k \leq H^+_{\rho,k} \leq \frac{\omega}{2} \, . 
\end{align}
Therefore
\begin{align}
\psi^+(\rho) \leq  \ln \bra*{\frac{H^+_{\rho,k}}{H^+_{\rho,k} - \rho + k + \frac{\omega}{2^{n+1}} }} \leq \ln(2^n) \leq n \ln(2) \, .
\end{align}
Furthermore, we need to study the properties of $(\psi^+)'(\rho)$:
\begin{align}
(\psi^+)'(\rho)=
\begin{cases}
\frac{1}{H^+_{\rho,k}- \rho +k +\frac{\omega}{2^{n+1}}} & \rho >  k + \frac{\omega}{2^{n+1}} \\
0 &  \rho \leq  k + \frac{\omega}{2^{n+1}} 
\end{cases}
\, .
\end{align}
Thus, we have
\begin{align}
0\leq (\psi^+)'(\rho) \leq \frac{2^{n+1}}{\omega} \, .
\end{align}
We now proceed to writing down the estimate
\begin{align}
\int_{K_R \times \set{t}} ((\psi^+)^2)(\rho) \zeta^2\dx{x} \leq & \int_{K_R \times \set{t_0}} ((\psi^+)^2)(\rho) \zeta^2\dx{x}  - 2m\beta^{-1}\int_{t_0}^t \int_{K_R \times \set{s}}  (\rho^{m-1}\abs{\nabla \rho}^2  ((\psi^+)'(\rho))^2 \zeta^2) \dx{x} \dx{s} \\&+2m\beta^{-1} \int_{t_0}^t \int_{K_R \times\set{s}} \rho^{m-1} \psi^+(\rho)\abs{\nabla \zeta}^2 \dx{x} \dx{s}\\& +2\int_{t_0}^t \int_{K_R \times \set{s}}  \rho\abs{\nabla f} \abs{ \nabla \rho} \abs{((\psi^+)'(\rho))^2(1+(\psi^+(\rho))} \zeta^2 \dx{x} \dx{s} \\&+ 4\int_{t_0}^t \int_{K_R \times \set{s}}  \rho\abs{\nabla f} \abs{\nabla \zeta}\abs{ ((\psi^+)'(\rho)) \psi^+(\rho)}  \abs{\zeta} \dx{x} \dx{s} \, .
\label{eq:logest1}
\end{align}
for any $t_0 \leq t \leq 0$. We choose a time-independent cut-off function $0 \leq \zeta \leq 1$ such that
\begin{align}
\zeta \equiv 1,\quad  x \in K_{(1- \delta)R} , \qquad \abs{\nabla \zeta} \leq (\delta R)^{-1} \, .
\end{align}
Consider now the  first term involving $f$ on the RHS of~\eqref{eq:logest1}
\begin{align}
&2\int_{t_0}^t \int_{K_R \times \set{s}}  \rho\abs{\nabla f} \abs{ \nabla \rho} \abs{((\psi^+)'(\rho))^2(1+(\psi^+(\rho))} \zeta^2 \dx{x} \dx{s} \\
\leq & \lambda 2m\beta^{-1}\int_{t_0}^t \int_{K_R \times \set{s}}  (\rho^{m-1}\abs{\nabla \rho}^2  ((\psi^+)'(\rho))^2 \zeta^2) \dx{x} \dx{s}\\& +\frac{1}{2 \lambda m \beta^{-1}} \int_{t_0}^t \int_{K_R} \rho^{3-m} \abs{\nabla f}^2 \abs{((\psi^+)'(\rho))^2(1+(\psi^+(\rho))^2} \zeta^2 \dx{x} \dx{s} \, ,
\end{align}
where we have simply applied Young's inequality and the constant $\lambda \in (0,1/2) $. We derive a similar bound for the second term involving $f$ as follows
\begin{align}
&4\int_{t_0}^t \int_{K_R \times \set{s}}  \rho\abs{\nabla f} \abs{\nabla \zeta}\abs{ ((\psi^+)'(\rho)) \psi^+(\rho)}  \abs{\zeta} \dx{x} \dx{s}
\\\leq & \lambda 2m\beta^{-1}\int_{t_0}^t \int_{K_R \times \set{s}}  (\rho^{m-1}\abs{\nabla \rho}^2  ((\psi^+)'(\rho))^2 \zeta^2) \dx{x} \dx{s}\\& +\frac{2}{ \lambda m \beta^{-1}} \int_{t_0}^t \int_{K_R} \rho^{3-m} \abs{\nabla f}^2 (\psi^+(\rho))^2 \abs{\nabla\zeta}^2 \dx{x} \dx{s} \, .
\end{align}
Putting it all together we can get rid of the negative term in~\eqref{eq:logest1} and take the $\esup$ to obtain:
\begin{align}
\esup_{t \in [t_0,0]}\int_{K_R \times \set{t}} ((\psi^+)^2)(\rho) \zeta^2\dx{x} \leq & \int_{K_R \times \set{t_0}} ((\psi^+)^2)(\rho) \zeta^2\dx{x}  \\&+2m\beta^{-1} \int_{t_0}^0 \int_{K_R \times\set{s}} \rho^{m-1} \psi^+(\rho)\abs{\nabla \zeta}^2 \dx{x} \dx{s}\\&+ \frac{1}{2 \lambda m \beta^{-1}} \int_{t_0}^0 \int_{K_R} \rho^{3-m} \abs{\nabla f}^2 \abs{((\psi^+)'(\rho))^2(1+(\psi^+(\rho))^2} \zeta^2 \dx{x} \dx{s} \\&+\frac{2}{ \lambda m \beta^{-1}} \int_{t_0}^0 \int_{K_R} \rho^{3-m} \abs{\nabla f}^2 (\psi^+(\rho))^2 \abs{\nabla\zeta}^2 \dx{x} \dx{s} \, .
\label{eq:logest2}
\end{align}
We proceed to bound each of the terms individually. For the first term on the RHS of~\eqref{eq:logest2}
we obtain:
\begin{align}
\int_{K_R \times \set{t_0}} ((\psi^+)^2)(\rho) \zeta^2\dx{x} \leq n^2 \ln(2)^2 \frac{1- \nu_0}{1- \nu_0/2}\abs{K_R}  \, .
\end{align}
For the second term we use the fact that $\rho \leq 5 \omega/4$  to obtain:
\begin{align}
2m\beta^{-1} \int_{t_0}^0 \int_{K_R \times\set{s}} \rho^{m-1} \psi^+(\rho)\abs{\nabla \zeta}^2 \dx{x} \dx{s} \leq & 2 m \beta^{-1} \bra*{\frac{5}{4}}^{m-1} \omega^{m-1 }(\delta R)^{-2}\abs{t_0}n \ln(2) \abs{K_R} \\
\leq &  2 m \beta^{-1} \bra*{\frac{5}{4}}^{m-1}\ln(2) \delta^{-2}n  \abs{K_R} \, .
\end{align}
For the third term we use the fact that $5/4 \omega\geq \rho \geq \omega/2$ on the supports of $\psi^+(\rho)$ and $(\psi^+)'(\rho)$ to obtain:
\begin{align}
&\frac{1}{2 \lambda m \beta^{-1}} \int_{t_0}^0 \int_{K_R} \rho^{3-m} \abs{\nabla f}^2 \abs{((\psi^+)'(\rho))^2(1+(\psi^+(\rho))^2} \zeta^2 \dx{x} \dx{s} \\
\leq &  C \frac{1}{2 \lambda m \beta^{-1}} \omega^{3-m} \omega^{1-m} R^2 \norm{\nabla f}_{\Leb^\infty(\Omega)}^2 2^{n+1} \omega^{-2} (1 + n \ln(2))^2 \abs{K_R} \\
=& \frac{C}{2 \lambda m \beta^{-1}} \omega^{1-m} \omega^{1-m} R^2 \norm{\nabla f}_{\Leb^\infty(\Omega)}^2 2^{n+1}  (1 + n \ln(2))^2 \abs{K_R} \, .
\end{align}
Similarly for the final term we obtain
\begin{align}
&\frac{2}{ \lambda m \beta^{-1}} \int_{t_0}^0 \int_{K_R} \rho^{3-m} \abs{\nabla f}^2 (\psi^+(\rho))^2 \abs{\nabla\zeta}^2 \dx{x} \dx{s}  \\
\leq &  \frac{2C}{ \lambda m \beta^{-1}} \omega^2 \omega^{1- m} \omega^{1-m} R^2 \norm{\nabla f}_{\Leb^\infty(\Omega)}^2   n^2 \ln(2)^2 \abs{K_R} \, .
\end{align}
For the LHS of~\eqref{eq:case2}, consider the set
\begin{align}
S_t= \set*{x \in K_{(1- \delta)R}: \rho(x,t) > \mu^+ - \omega/2^{n+1}} \subset K_R, \qquad t \in(t_0, 0) \, .
\end{align}
It is clear that $\zeta=1$ on $S_t$ and, since $-\rho+k + \omega/2^{n+1}<0$, the function
\begin{align}
\frac{H^+_{\rho,k}}{H^+_{\rho,k} - \rho + k + \frac{\omega}{2^{n+1}} } \, ,
\end{align}
is decreasing in $H^+_{\rho,k}$. Thus, in $S_t$ we have
\begin{align}
&\frac{H^+_{\rho,k}}{H^+_{\rho,k} - \rho + k + \frac{\omega}{2^{n+1}} } \\
\geq & \frac{\omega/2}{\omega/2 - \rho + k + \frac{\omega}{2^{n+1}} } \\
\geq & \frac{\omega/2}{\omega/2+\omega/2^{n+1} - \omega/2 + \omega/2^{n+1}   }=2^{n-1}.
\end{align}
Thus, we have 
\begin{align}
\esup_{t \in [t_0,0]}\int_{K_R \times \set{t}} ((\psi^+)^2)(\rho) \zeta^2\dx{x} \geq (n-1)^2 \ln(2)\abs{S_t} \, .
\end{align}
Putting all the terms back together we obtain and bounding $\omega^2$ by $M^2$,
\begin{align}
\abs{S_t} \leq & \bra*{\bra*{\frac{n}{n-1}}^2 \frac{1- \nu_0}{1- \nu_0/2} + C(m,\beta) \delta^{-2}\frac{n}{(n-1)^2}  } \abs{K_R} \\ 
&+ \bra*{C_1\bra*{m,\beta,\lambda, \norm{\nabla f}_{\Leb^\infty(\Omega)}}\omega^{1-m} \omega^{1-m} R^2  2^{n+1}  \bra*{\frac{1 + n \ln(2)}{n-1}}^2  } \abs{K_R} \\
& +\bra*{C_2\bra*{m,\beta,\lambda, \norm{\nabla f}_{\Leb^\infty(\Omega)},M}\omega^{1-m} \omega^{1-m} R^2  \bra*{\frac{n}{n-1}}^2 } \abs{K_R} \, . 
\end{align}
Finally, we obtain the estimate we need
\begin{align}
&\abs*{\set*{x \in K_R: \rho(x,t) > \mu^+ - \frac{\omega}{2^q}}} \\
\leq & \abs{S_t} + \abs{K_R \setminus K_{(1-\delta)R}} \\
\leq & \bra*{\bra*{\frac{n}{n-1}}^2 \frac{1- \nu_0}{1- \nu_0/2} + C(m,\beta) \delta^{-2}\frac{n}{(n-1)^2}  + d \delta} \abs{K_R} \\ 
&+ \bra*{C_1\bra*{m,\beta,\lambda, \norm{\nabla f}_{\Leb^\infty(\Omega)}}R^\eps\omega^{1-m} R^\eps\omega^{1-m} L^{2 -2 \eps}  2^{n+1}  \bra*{\frac{1 + n \ln(2)}{n-1}}^2  } \abs{K_R} \\
& +\bra*{C_2\bra*{m,\beta,\lambda, \norm{\nabla f}_{\Leb^\infty(\Omega)},M} R^\eps\omega^{1-m} R^\eps\omega^{1-m} L^{2- 2\eps}  \bra*{\frac{n}{n-1}}^2 } \abs{K_R} \, ,
\end{align}
where one should note that $R \leq L$ and the term $R^\eps\omega^{1-m}$ can be controlled by $\alpha$ through~\eqref{eq:driftalpha}.
Note that for the term in the first set of brackets we can choose $d \delta \leq 3 \nu_0^2/16$ and $n$ large enough such that
\begin{align}
\bra*{\frac{n}{n-1}}^2 \leq (1- \nu_0/2)(1+ \nu_0)  , \qquad C(m,\beta)\frac{n}{(n-1)^2}\delta^{-2} \leq 3 \nu_0^2/16 \, ,
\end{align}
because $(1- \nu_0/2)(1+ \nu_0) >1$. Now that $n$ and $\delta$ have been fixed we note that the constant $\alpha$ in~\eqref{eq:driftalpha} can be made small enough (independent of $\omega$ and $R$) so that terms in the other two brackets are lesser that $3\nu_0^2/16$.
This gives us
\begin{align}
\abs*{\set*{x \in K_R: \rho(x,t) > \mu^+ - \frac{\omega}{2^q}}} \leq \bra*{1- \nu_0^2 + 3 \nu_0^2/4}\abs{K_R}= \bra*{1- \frac{\nu_0^2}{4} }\abs{K_R} \, .
\end{align}
The proof follows by setting $q=n+1$ and noting that $\pra{t_0,0} \supset\pra*{-\frac{\nu_0}{2}\omega^{1-m}R^2,0}$.
\end{proof}
We now proceed to prove that $\rho$ is strictly lesser than its supremum in a smaller parabolic cylinder.
\begin{lemma}\label{lem:redcase2}
Assume that~\eqref{eq:case2} holds. Then there exists some $s_0 \in \N$ large enough, independent of $\omega$, such that
\[
\rho(x,t) < \mu^+ - \frac{\omega}{2^{s_0}} \textrm{ a.e. }(x,t) \in Q \bra*{\frac{\nu_0}{2}\omega^{1-m}\bra*{\frac{R}{2}}^2, \frac{R}{2}} \, .
\]
\end{lemma}
\begin{proof}
The proof is similar to that of~\cref{lem:redcase1} and relies on the energy estimates in~\cref{lem:eest}.
We start by considering the sequence
\[
R_n= \frac{R}{2} + \frac{R}{2^{n+1}} \qquad n=0,1, \cdots
\]
such that $R_0=R$ and $R_n \to R/2$ as $n \to \infty$. We then construct a sequence of nested shrinking cylinders $Q(\nu_0 2^{-1}\omega^{1-m}R_n^2, R_n)$ along with cut-off functions $\zeta_n$ satisfying
\begin{align}
0 \leq \zeta_n \leq 1, &\qquad \zeta_n=1 \textrm{ in } Q(\nu_0 2^{-1}\omega^{1-m}R^2_{n+1},R_{n+1}),
\qquad 
\zeta_n=0 \textrm{ on } \cb{\partial_p} Q(\nu_0 2^{-1}\omega^{1-m}R^2_{n},R_{n}),  \\
\abs{\nabla \zeta_n} \leq \frac{2^{n-1}}{R}, &\qquad 0 \leq (\zeta_n)_t \leq \frac{2^{2n-2}}{R^2} \omega^{m-1 },  \qquad
\Delta \zeta_n  \leq \frac{2^{2n-2}}{R^2}.
\end{align}
We now apply the energy estimate of~\cref{lem:eest} in $Q(\nu_0 2^{-1}\omega^{1-m}R_n^2, R_n)$ with $\ell= \mu^+ - \omega/2^{s_0}$, and $k_n=\mu^+ - \omega/(2^{s_0}) - \omega/(2^{n+s_0})$ for the function $(\rho^\ell_+ -k_n )_+$. We will bound the terms on the LHS and RHS separately. Considering first the terms on the LHS we have
\begin{align}
&\frac{1}{2} \left[\esup\limits_{- R_n^2 \omega^{1-m}\nu_0 2^{-1} < t <0 } \int_{K_{R_n} \times \set{t} } \bra{\rho^\ell_+ -k_n }_+^2 \zeta_n^2 \dx{x}  \right.\\&\left.+ \int_{Q(\nu_0 2^{-1}\omega^{1-m}R_n^2, R_n)}(\rho^\ell_+)^{m-1 } \abs{\nabla (\rho^\ell_+ -k_n)_+ \zeta_n}^2 \dx{x} \dx{t}\right]  \\
\geq & 2^{-m} \left[\esup_{-\nu_0 2^{-1} R_n^2 \omega^{1-m} < t <0 } \int_{K_{R_n} \times \set{t} } \bra{\rho^\ell_+ -k_n }_+^2 \zeta_n^2 \dx{x} \right.\\& \left.+ \omega^{m-1}\int_{Q(\nu_0 2^{-1}\omega^{1-m}R_n^2, R_n)}\abs{\nabla (\rho^\ell_+ -k_n)_+ \zeta_n}^2 \dx{x} \dx{t} \right] \, , 
\end{align}
where we have used the fact that when $\abs{\nabla (\rho^\ell_+ -k)_+ \zeta_n}$ is nonzero, $\rho^\ell_+ \geq k_n \geq \omega/2$. For the RHS we first note the following facts:
\begin{enumerate}
\item $0 \leq \mu^- \leq \omega/4 $ which implies that $\rho \leq 5 \omega/4$,  and $\rho^\ell_+ \leq 5 \omega/4$ .
\item $\ell = \mu^- - \omega/2^{s_0}> k_n$ which implies that $\chi_{[\rho \geq \ell]} \leq \chi_{[\rho \geq k_n]}=\chi_{[(\rho- k_n)_+>0]} $.
\item If $\rho^\ell_+= \rho$, then $\chi_{[(\rho^\ell_+- k_n)_+>0]}=\chi_{[(\rho- k_n)_+>0]}  $. On the other hand if  $\rho^\ell_+= \ell$, we have that $\rho \geq \ell \geq k_n$. Thus, we have that $\chi_{[(\rho- k_n)_+>0]} =\chi_{[(\rho^\ell_+ - k_n)_+>0]} $. 
\item $(l-k_n)_{+}=\omega/(2^{n+s_0}) \leq \omega/2^{s_0-1}$, $(\rho^\ell_{+}-k_n)_{+}\leq \omega/2^{n+s_0} \leq \omega/2^{s_0-1}$,  $(\rho- \ell)_{+} \leq \omega/2^{s_0-1}$. 
\end{enumerate}
Applying, essentially the same bounds as~\cref{lem:redcase1}, we obtain
\begin{align}
&\pra*{\esup_{-\nu_0 2^{-1} R_n^2 \omega^{1-m} < t <0 } \int_{K_{R_n} \times \set{t} } \bra{\rho^\ell_+ -k }_+^2 \zeta_n^2 \dx{x}  + \omega^{m-1}\int_{Q(\nu_0 2^{-1}\omega^{1-m}R_n^2, R_n)}\abs{\nabla (\rho^\ell_+ -k)_+ \zeta_n}^2 \dx{x} \dx{t} } \\\leq & C\bra*{m, L,\beta,\norm{\nabla f}_{\Leb^\infty(\Omega)}, \norm{\Delta f}_{\Leb^\infty(\Omega)}} \frac{2^{2n}}{R^2}\omega^{m-1} \frac{\omega^2}{2^{2s_0-2}} \int_{Q(\nu_0 2^{-1}\omega^{1-m}R_n^2, R_n)} \chi_{[(\rho^\ell_+- k_n)_+>0]}\dx{x} \dx{t} \, .
\end{align}
Let $\bar{t}=\nu_0^{-1} 2\omega^{m-1}t$ and define the following rescaled functions
\begin{align}
\bar{\rho}^\ell_+( \cdot,\bar{t}) =\rho^\ell_+( \cdot,t) , \qquad \bar{\zeta_n}( \cdot,\bar{t}) =\zeta_n( \cdot,t)  \, .
\end{align}
In these new variables the inequality simplifies to
\begin{align}
&\pra*{\esup_{- R_n^2  < \bar{t} <0 } \int_{K_{R_n} \times \set{\bar{t}} } \bra{\bar{\rho}^\ell_+ -k_n }_+^2 \bar{\zeta_n}^2 \dx{x}  + \frac{\nu_0}{2}\int_{Q(R_n^2, R_n)}\abs{\nabla (\bar{\rho}^\ell_+ -k_n)_+ \bar{\zeta_n}}^2 \dx{x} \dx{t} } \\\leq & C \frac{2^{2n}}{R^2} \frac{\nu_0}{2} \frac{\omega^2}{2^{2s_0-2}} A_n \, ,
\end{align}
where
\begin{align}
A_n:= \int_{Q(R_n^2, R_n)} \chi_{[(\bar{\rho}^\ell_+- k_n)_+>0]}\dx{x} \dx{t} \, .
\end{align}
Since $\nu_0 \in(0,1)$ it simplifies to,
\begin{align}
&\pra*{\esup_{- R_n^2  < \bar{t} <0 } \int_{K_{R_n} \times \set{\bar{t}} } \bra{\bar{\rho}^\ell_+ -k_n }_+^2 \bar{\zeta_n}^2 \dx{x}  + \int_{Q(R_n^2, R_n)}\abs{\nabla (\bar{\rho}^\ell_+ -k_n)_+ \bar{\zeta_n}}^2 \dx{x} \dx{t} } \\\leq & C \frac{2^{2n}}{R^2}  \frac{\omega^2}{2^{2s_0-2}} A_n \, .
\end{align}
Furthermore we have
\begin{align}
\frac{1}{2^{2n+2}}\frac{\omega^2}{2^{2s_0-2}}A_{n+1}=&\abs{k_n-k_{n+1}}^2 A_{n+1} \\
=&\int_{Q(R_{n+1}^2, R_{n+1})} \abs{k_n-k_{n+1}}^2\chi_{[(\bar{\rho}^\ell_+- k_{n+1})_+>0]}\dx{x} \dx{t} \\
\leq & \int_{Q(R_{n+1}^2, R_{n+1})} \abs{k_n-\bar{\rho}^\ell_+}^2\chi_{[(\bar{\rho}^\ell_+- k_{n+1})_+>0]}\dx{x} \dx{t} \\
\leq & \norm*{\bra*{k_n-\bar{\rho}^\ell_+}_+}^2_{\Leb^2(Q(R_{n+1}^2, R_{n+1}))} \\
\leq &  C_d A_n^{2/(2+d)}\norm{\bra*{k_n-\bar{\rho}^\ell_+}_+}^2_{V^2(Q(R_{n+1}^2, R_{n+1}))} \, ,
\end{align}
where in the last step we have used the emebedding into the parabolic space $V^2$ (cf.~\cref{embedding}). Thus, as in~\cref{lem:redcase1} we have
\begin{align}
\frac{1}{2^{2n+2}}\frac{\omega^2}{2^{2s_0-2}}A_{n+1} \leq & C \frac{2^{2n}}{R^2} \frac{\omega^2}{2^{2s_0-2}} A_n^{1+ 2/(d+2)} \, .
\end{align}
This can be simplified to
\begin{align}
X_{n+1} \leq C 4^{2n} X_n^{1+2/(d+2)} \, ,
\end{align}
where
\begin{align}
X_n= \frac{A_n}{\abs{Q(R_n^2,R_n)}} \, ,
\end{align}
and the constant $C$ independent of $\omega, R,n$ and dependent only $d,m,\beta,f$. Thus, if
\begin{align}
X_0 \leq C^{-(d+2)/2}4^{(d+2)^2/2}:=\nu_0^* \, ,
\label{eq:X0bound}
\end{align}
by the geometric convergence lemma (cf.~\cref{convergence}), $X_n \to 0$ and the result follows as in the proof of~\cref{lem:redcase1}. Thus, all that remains to be shown is~\eqref{eq:X0bound} holds. Before we do this we introduce the following notation
\begin{align}
B_\sigma(t)= \set*{x \in K_R: \rho(x,t) > \mu^+ - \frac{\omega}{2^\sigma}} \, ,
\end{align}
and
\begin{align}
B_\sigma= \set*{(x,t) \in Q \bra*{\frac{\nu_0}{2}\omega^{1-m}R^2,R}: \rho(x,t) > \mu^+ - \frac{\omega}{2^\sigma}} \, .
\end{align}
In this notation~\eqref{eq:X0bound} reads as
\begin{align}
\abs{B_{s_0-1}} \leq \nu_0^* \abs*{Q \bra*{\frac{\nu_0}{2}\omega^{1-m}R^2,R}} \, .
\end{align}
The above inequality means that the subset  of $Q \bra*{\frac{\nu_0}{2}\omega^{1-m}R^2,R}$ where $\rho$
is close to its supremum can be made arbitrarily small. To show this, we apply the energy estimate of~\cref{lem:eest} to the function $(\rho^{\mu^+}_+- k)_+$ with
\begin{align}
k= \mu^+ - \frac{\omega}{2^{s}}, \qquad q < s<s_0 \, ,
\end{align}
with a cut-off function $\zeta$ defined in $Q \bra*{\frac{\nu_0}{2}\omega^{1-m}R^2,2R}$ such that
\begin{align}
\zeta &\equiv 1, \textrm{ in }Q \bra*{\frac{\nu_0}{2}\omega^{1-m}R^2,R}, \qquad &\zeta=0 \textrm{ on }\partial_p Q \bra*{\frac{\nu_0}{2}\omega^{1-m}R^2,2R} \, , \\
\abs{\nabla \zeta} \leq & \frac{1}{R} , \qquad &0 \leq \zeta_t \leq \frac{\omega^{m-1}}{R^2} \, . 
\end{align}
We delete the first term on the LHS and bound the rest as follows:
\begin{align}
&\frac{1}{2} \pra*{\esup\limits_{- R^2 \omega^{1-m}\nu_0 2^{-1} < t <0 } \int_{K_{2R} \times \set{t} } \bra{\rho -k }_+^2 \zeta^2 \dx{x}  + \int_{Q(\nu_0 2^{-1}\omega^{1-m}R^2, 2R)}(\rho)^{m-1 } \abs{\nabla (\rho -k)_+ \zeta}^2 \dx{x} \dx{t} }   \\
\geq & 2^{-m} \omega^{m-1}\int_{Q(\nu_0 2^{-1}\omega^{1-m}R^2, R)}\abs{\nabla (\rho -k)_+ \zeta}^2 \dx{x} \dx{t}  =\geq 2^{-m} \omega^{m-1}\int_{Q(\nu_0 2^{-1}\omega^{1-m}R^2, R)}\abs{\nabla (\rho -k)_+ }^2 \dx{x} \dx{t}\, ,
\end{align}
where we have used the fact that when $\abs{\nabla(\rho- k)_+ \zeta}$ is nonzero then $\rho>k > \omega/2$.
For the terms on the RHS we bound them as in~\cref{lem:redcase1} (note that two of the terms are zero because $\rho \leq \ell=\mu^+$ a.e. $(x,t)$). Thus, we have the bound
\begin{align}
 &2^{-m} \omega^{m-1}\int_{Q(\nu_0 2^{-1}\omega^{1-m}R^2, R)}\abs{\nabla (\rho -k)_+ }^2 \dx{x} \dx{t}
 \\\leq &  C\bra*{m, L,\beta,\norm{\nabla f}_{\Leb^\infty(\Omega)}, \norm{\Delta f}_{\Leb^\infty(\Omega)}} \frac{\omega^{m-1}}{R^2} \frac{\omega^2}{2^{2s-2}} \int_{Q(\nu_0 2^{-1}\omega^{1-m}R^2, 2R)} \chi_{[(\rho- k)_+>0]}\dx{x} \dx{t}\\
 \leq & C \frac{\omega^{m-1}}{R^2} \frac{\omega^2}{2^{2s-2}} \abs*{Q(\nu_0 2^{-1}\omega^{1-m}R^2, 2R)}
\end{align}
Since $\abs*{Q(\nu_0 2^{-1}\omega^{1-m}R^2, 2R)}= 2^{d+1} \abs*{Q(\nu_0 2^{-1}\omega^{1-m}R^2, R)}$, multiplying my $\omega^{1-m}$ this reduces to
\begin{align}
\int_{Q(\nu_0 2^{-1}\omega^{1-m}R^2, R)}\abs{\nabla (\rho -k)_+ }^2 \dx{x} \dx{t} \leq\frac{C}{R^2} \frac{\omega^2}{2^{2s-2}} \abs*{Q(\nu_0 2^{-1}\omega^{1-m}R^2, R)} \, .
\end{align}
Note now that $B_s \subset Q(\nu_0 2^{-1}\omega^{1-m}R^2, R)$ and, in $B_s$, $\abs{\nabla(\rho-k)_+}=\abs{\nabla(\rho-k)}=\abs{\nabla \rho}$. Thus, the above inequality gives us
\begin{align}
\int_{B_s}\abs{\nabla \rho}^2 \dx{x} \dx{t} \leq \frac{C}{R^2} \frac{\omega^2}{2^{2s-2}} \abs*{Q(\nu_0 2^{-1}\omega^{1-m}R^2, R)} \, .
\label{eq:1step}
\end{align}
We now apply the lemma of De Giorgi (cf.~\cref{degiorgi}) with $k_1=\mu^+ - \omega/2^s$ and $k_2= \mu^+ - \omega/2^{s+1}$, to obtain that for all $t \in \pra*{-\nu_0 2^{-1} \omega^{1-m}R^2,0}$
\begin{align}
\frac{\omega}{2^{s+1}}\abs{B_{s+1}(t)} \leq C \frac{R^{d+1}}{\abs{K_R \setminus B_s(t)}}\int_{B_{s}(t) \setminus B_{s+1}(t)}\abs{\nabla \rho} \dx{x} \, .
\label{eq:digiorgi1}
\end{align}
Since $q \leq s-1$, by~\cref{lem:driftest}, it follows that $\abs{B_{s-1}(t)} \leq\abs{B_{q}(t)}  \leq (1- \nu_0^2/4)\abs{K_R}$ for all $t \in \pra*{-\nu_0 2^{-1} \omega^{1-m}R^2,0}$. Thus, for all such $t$ it follows that
\begin{align}
&\abs{K_R \setminus B_s(t)}=\abs*{\set*{x \in K_R : \rho(x,t) < \mu^+ - \frac{\omega}{2^s}}} \\\geq & \abs*{\set*{x \in K_R : \rho(x,t) < \mu^+ - \frac{\omega}{2^{s-1}}}}\\= & \abs{K_R}- \abs{B_{s-1}(t)} \geq \frac{\nu_0^2}{4} \abs{K_R} \, . 
\end{align}
Thus,~\eqref{eq:digiorgi1} can be rewritten as
\begin{align}
\frac{\omega}{2^{s+1}}\abs{B_{s+1}(t)} \leq C \frac{R^{d+1}}{\abs{K_R}\nu_0^2}\int_{B_{s}(t) \setminus B_{s+1}(t)}\abs{\nabla \rho} \dx{x} \, .
\end{align}
for $t \in \pra*{-\nu_0 2^{-1} \omega^{1-m}R^2,0}$. We integrate the above inequality over $ \pra*{-\nu_0 2^{-1} \omega^{1-m}R^2,0}$ to obtain
\begin{align}
\frac{\omega}{2^{s+1}}\abs{B_{s+1}} \leq & C \frac{R}{\abs{K_R}\nu_0^2}\int_{B_{s} \setminus B_{s+1}}\abs{\nabla \rho} \dx{x} \dx{t} \\
\leq & C \frac{R}{\nu_0^2}\bra*{\int_{B_{s} \setminus B_{s+1}}\abs{\nabla \rho}^2 \dx{x} \dx{t}}^{1/2} \abs*{B_{s} \setminus B_{s+1}}^{1/2}\\
\leq &   \frac{C}{\nu_0^2}\frac{\omega}{2^s}  \abs*{Q(\nu_0 2^{-1}\omega^{1-m}R^2, R)}^{1/2} \abs*{B_{s} \setminus B_{s+1}}^{1/2} \, ,
\end{align}
where in the last step we have applied~\eqref{eq:1step}. Squaring both sides we obtain
\begin{align}
\abs{B_{s+1}}^2 \leq \frac{C}{\nu_0^4}\abs*{Q(\nu_0 2^{-1}\omega^{1-m}R^2, R)} \abs*{B_{s} \setminus B_{s+1}} \, .
\end{align}
Since $q<s<s_0$, we sum the above inequality for $s=q+1, \dots, s_0-2$ to obtain
\begin{align}
\sum_{s=q+1}^{s_0-2}\abs{B_{s+1}}^2 \leq \frac{C}{\nu_0^4}\abs*{Q(\nu_0 2^{-1}\omega^{1-m}R^2, R)} \sum_{s=q+1}^{s_0-2}\abs*{B_{s} \setminus B_{s+1}} \, .
\end{align}
Note that $\sum_{s=q+1}^{s_0-2}\abs*{B_{s} \setminus B_{s+1}} \leq \abs*{Q(\nu_0 2^{-1}\omega^{1-m}R^2, R)}$. Additionally, $B_{s_0-1} \subset B_s$ for all $s=q+1, \dots, s_0-2 $. Thus, we have
\begin{align}
\abs{B_{s_0-1}}^2 \leq \frac{C}{\nu_0^4((s_0-q-3) )}\abs*{Q(\nu_0 2^{-1}\omega^{1-m}R^2, R)}^2 \, .
\end{align}
For $s_0 \in \N$ sufficiently large independent of $\omega$, $R$, \eqref{eq:X0bound} is satisfied and the result follows.
\end{proof}
Finally we can state the reduction of oscillation result in case 2.
\begin{corollary}[Reduction of oscillation in case 2]
Assume that~\eqref{eq:case1} holds with constant $\nu_0$ as specified in the proof of~\cref{lem:redcase1}. Then there exists a $\sigma_2 \in(0,1)$, independent of $\omega$, $R$, such that
\begin{align}
\eosc_{Q\bra*{\nu_0 2^{-1}\omega^{1-m}\bra*{\frac{R}{2}}^2, \frac{R}{2}}}\rho \leq \sigma_2 \omega \, .
\end{align}
\end{corollary}
\begin{proof}
We know from~\cref{lem:redcase2} that there exists some $s_0 \in \N$ such that
\begin{align}
\esup_{Q\bra*{\nu_0 2^{-1}\omega^{1-m}\bra*{\frac{R}{2}}^2, \frac{R}{2}}} \rho\leq \mu^+ - \frac{\omega}{2^{s_0}} \, .
\end{align}
Thus
\begin{align}
\eosc_{Q\bra*{\nu_0 2^{-1}\omega^{1-m}\bra*{\frac{R}{2}}^2, \frac{R}{2}}} \rho \leq & \mu^+ - \frac{\omega}{2^{s_0}} - \mu^- \\
\leq & \bra*{1- \frac{1}{2^{s_0}}}\omega \, .
\end{align}
Thus, for $\sigma_2=\bra*{1- \frac{1}{2^{s_0}}}$ the result follows.
\end{proof}
We combine the two cases into one:
\begin{lemma}[Total reduction of oscillation]\label{osc}
Fix some $0<R<L$ such that $Q(4R^{2-\eps}, 2R) \subset \Omega_T$. Assume that $\eosc_{Q(4R^{2-\eps}, 2R)} \rho \leq \omega$ and $\alpha \omega^{m-1}> R^{\eps}$ and that $\mu^{-}> \omega/4$. Then there exists a constant $\sigma \in (0,1)$, depending only on the data (and continuously on $\beta>0$), and independent of $\omega$ and $R$, such that
\begin{align}
\eosc_{Q\bra*{\nu_0 2^{-1}\omega^{1-m}\bra*{\frac{R}{2}}^2, \frac{R}{2}}} \rho \leq \sigma \omega \, .
\end{align}
\end{lemma}
\begin{proof}
The proof follows from the fact that $Q\bra*{\nu_0 2^{-1}\omega^{1-m}\bra*{\frac{R}{2}}^2, \frac{R}{2}} \subset Q\bra*{\omega^{1-m}\bra*{\frac{R}{2}}^2, \frac{R}{2}}$ and setting $\sigma= \max\set{\sigma_1,\sigma_2}$.
\end{proof}
We can now complete the proof of~\cref{holder}:
\begin{proof}[Proof of~\cref{holder}]
We now show that there exist constants $\gamma>1$, $a \in (0,1)$, depending only on the data ($W$, $\beta$, $m$, $d$, $M$), such that for all $0 \leq r \leq L$ we have
\begin{align}\label{controlosc}
\eosc_{Q\bra*{\omega^{1-m}r^2, r}} \rho \leq \gamma \omega \bra*{\frac{2r}{L}}^{a} \, .
\end{align}
where $\omega= c_1 M$ and $c_1$ is chosen to be large enough so that $\alpha \omega^{m-1} >L^{\eps}$. We choose as our starting point the cylinder $Q(4(L/2)^{2-\eps},L) \subset \Omega_T$. We start by defining
\begin{align}
R_k= c_0^k L/2 \, , \qquad c_0= \frac{1}{2}\sigma^{(m-1)/\eps}\frac{\nu_0}{2} <\frac{1}{2} \, , \qquad \omega_k= \sigma^k \omega \, ,
\end{align}
for $k =0,1, \dots$ and $\eps \leq (m-1)$. We  already have that $\alpha \omega^{m-1}>R_0^\eps$ for all $0 \leq r \leq R$. This implies that
\begin{align}
\omega_k^{1-m} R_k^\eps=& \sigma^{k(1-m)} c_0^{k\eps} \omega^{1-m} R_0^\eps \\
< & \alpha \bra*{\frac{\nu_0}{4}}^{k \eps} < \alpha. 
\end{align}
Additionally, we also have that
\begin{align}
\sigma =& \sigma^{1+ \frac{1-m}{\eps}} \sigma^{\frac{m-1}{\eps}} \\
& > c_0 \, .
\end{align}
It follows that
\begin{align}
\eosc_{Q(\omega^{1-m}R_0^2,R_0)} \rho \leq\eosc_{Q(4R_0^{2- \eps},R_0)} \rho \leq M \leq c_1 M=\omega\, . 
\end{align}
Furthermore, we have
\begin{align}
\eosc_{Q(\omega^{1-m}R_1^2,R_1)} \rho \leq\eosc_{Q(\omega^{1-m} \nu_0 2^{-1}(R/2)^2,R/2)}\rho \leq \sigma \omega \, ,
\end{align}
where we have applied~\cref{osc}. We can repeat the procedure starting at $R_k$ with $\omega_k= \sigma^k \omega$ and $\mu^-_k:= \einf_{Q(\omega^{1-m}R_k^2,R_k)} \rho$ assumed to be smaller than $\omega_k/4$.  If this is not the case, then the equation  is uniformly parabolic in $Q(\omega^{1-m}R_k^2,R_k)$ and by parabolic regularity theory (cf.~\cite{LSU68}),~\eqref{controlosc} holds for some constants $\gamma'>1,a' \in(0,1)$, depending only on the data. The dependence of the constants on $\beta>0$ is continuous. 

Assuming $\mu^-_k>\omega_k/4$ and applying the results of \cref{osc}  to $R_{k+1}$ we obtain
\begin{align}
\eosc_{Q(\omega^{1-m}R_{k+1}^2,R_{k+1})} \rho =&\eosc_{Q(\sigma^{k(1-m)}\omega^{1-m}\sigma^{(m-1)(k+2/\eps)}\nu_0^2 2^{-2}(R_k/2)^2,R_k/2)}\rho \\
\leq & \eosc_{Q(\omega_k^{1-m} \nu_0 2^{-1}(R_k/2)^2,R_k/2)}\rho \leq\sigma  \omega_k\, .
\end{align}
By induction it follows that
\begin{align}
\eosc_{Q(\omega^{1-m}R_{k}^2,R_k)} \rho \leq \sigma^k \omega \, .
\end{align}
Additionally, for all $0 \leq r \leq L$ we have that
\begin{align}
c_0^{k+1}(L/2) \leq r \leq c_0^{k}(L/2) \, ,
\end{align}
for some $k$. Picking $a= \log_{c_0} \sigma>0$, we derive
\begin{align}
\sigma^{k+1} \leq \bra*{\frac{2r}{L}}^a \, .
\end{align}
Thus, we have
\begin{align}
\eosc_{Q\bra*{\omega^{1-m}r^2, r}} \rho \leq \gamma \omega \bra*{\frac{2r}{L}}^{a} \, ,
\end{align}
 where $\gamma= \max\set{\sigma^{-1},\gamma'} >1$ and $a= \min\set{\log_{c_0} \sigma,a'} \in(0,1)$ since $\sigma>\sigma_1 >1/2>c_0$. Note that~\eqref{controlosc} implies that $\rho$ is continuous. One can see this by mollifying with some standard mollifier $\varphi^\eps$ and applying Arzel\`a--Ascoli to show that the limit as $\eps \to 0$ is continuous.

 Now that we have control on the oscillation of the solution we can proceed to the proof of H\"older regularity. Consider a weak solution $\rho(x,t)$ defined on $\Omega_T$. We would like the H\"older regularity to be uniform in space and time so we consider only those points such that $(x,t) +Q(4(L/2)^{2-\eps},L) \subset \Omega_T^\circ$. The local regularity near $\partial_p \Omega_T$ can be derived in a similar manner. Fix two points $(x,t)$ and $(y,t)$ for some $t$ large enough,  and consider the recursive scheme starting from $K:=(x,t)+Q(4(L/2)^{2-\eps},L) \subset \Omega_T$.  Setting  $r=d_{\T^d}\bra{x,y}$ and applying~\eqref{controlosc}, we obtain
\begin{align}\label{1}
\abs{\rho(x,t)-\rho(y,t)} \leq \eosc_{(x,t)+Q\bra*{\omega^{1-m}r^2, r}} \rho \leq \gamma \omega \bra*{\frac{2r}{L}}^{a} \leq \gamma 2^a c_1 M L^{-a} d_{\T^d}(x,y)^a \, .
\end{align}
For the time regularity we consider two points $(x,t_1),(x,t_2) \in \Omega_T, t_1>t_2 $ assuming that $\abs{t_1-t_2}^{1/2} \leq \omega^{1-m}(L/2)^2$. We consider the recursive scheme starting from $K:=(x,t_1)+Q(4(L/2)^{2-\eps},L) \subset \Omega_T$. Setting $r= \omega^{(m-1)/2}\abs{t_1-t_2}^{1/2}$, we obtain
\begin{align}\label{2}
\abs{\rho(x,t_2)-\rho(x,t_1)} \leq \eosc_{(x,t_1)+Q\bra*{\omega^{1-m}r^2, r}} \rho \leq \gamma \omega \bra*{\frac{2r}{L}}^{a} \leq \gamma 2^a (c_1 M)^{(2 + a(m-1))/2} L^{-a}  \abs{t_1-t_2}^{a/2} \, .
\end{align} 
For $\abs{t_1-t_2}^{1/2} > \omega^{1-m}(L/2)^2$, the proof is easier since
\begin{align}\label{3}
\abs{\rho(x,t_2)-\rho(x,t_1)} \leq 2M \leq 2M \abs{t_1-t_2}^{a/2} (L/2)^{-a} (c_1 M)^{(m-1)/2} \, .
\end{align}
Combining~\eqref{1},~\eqref{2}, and~\eqref{3}  together we have the required H\"older regularity away from the boundary:
\begin{align}\label{holdernorm}
\abs{\rho(x,t_1)- \rho(y,t_2)} \leq & C_{h} \bra{d_{\T^d}(x,y)^a +\abs{t_1-t_2}^{a/2}} \\
\leq &  C_h \bra{d_{\T^d}(x,y) +\abs{t_1-t_2}^{1/2}}^a \, ,
\end{align}
where $a\in (0,1)$ depends continuously on $\beta>0$ and $C_h$ depends on $M$, $L$, $m$, $\gamma$, and $d$. The regularity near the parabolic boundary can be derived in a similar manner.
\end{proof}
\begin{remark}
\cb{We note that the proof of~\cref{cor:holder} follows from the fact that the constant $C_h$ is uniform in time as long as we are far enough from the initial data $\rho_0$, i.e. if $0<C<t_1<t_2<\infty$ for some constant $C>0$.}
\end{remark}
\subsection*{Acknowledgements}
The authors would like to thank Felix Otto and Yao Yao for useful discussions during the course of this work. We are also grateful to the reviewers for their careful reading of the manuscript and their useful suggestions.
\begin{appendix}
\section{Some useful results}
In this section we state some useful lemmas and inequalities which we will use in the proof of~\cref{holder}.
\begin{lemma}[Geometric convergence lemma]\label{convergence}
Let $\set{X_n}_{n \in \N}$ be a sequence of nonnegative real numbers satisfying the recurrence inequality
\begin{align}
X_{n+1} \leq C b^n X_n^{1+a} \, ,
\end{align}
for some $C,b>1$ and $a>0$. If $X_0 \leq C^{-1/a}b^{-1/a^2}$, then $\lim\limits_{n \to \infty}X_n=0$.
\end{lemma}

Let $\Omega \subset \T^d$ be a smooth, convex, open subdomain. Then we have the following lemma due to De Giorgi~\cite{DeG57}:
\begin{lemma}\label{degiorgi}
Given a  function $v \in W^{1,1}(\Omega)$ and real numbers $k_1<k_2$ we define
\begin{align}
\pra*{v \gtrless k_i}&:= \set*{x \in \Omega: v(x) \gtrless k_i} \\
\pra*{k_1<v <k_2}&:= \set*{x \in \Omega: k_1 <v(x) < k_2} \, .
\end{align}
Then there exists a constant $C=C(d)$ such that
\begin{align}
(k_2-k_1)\abs*{\pra*{v >k_2}} \leq C \frac{R^{d+1}}{\abs*{\pra*{v <k_1}}} \int_{\pra*{k_1<v <k_2}} \abs*{\nabla v} \dx{x} \, ,
\end{align}
where $R= \mathrm{diam}(\Omega).$
\end{lemma}
Consider now the parabolic space $V^2(\Omega_T)$, equipped with the norm
\begin{align}
\norm{\rho}_{V^2(\Omega_T)}^2:= \esup_{0 \leq t\leq T} \norm{\rho}_{\Leb^2(\Omega)}^2(t) + \norm{\nabla \rho}_{\Leb^2(\Omega_T)}^2 \, .
\end{align}
We then have the following embedding~\cite[page 9]{DiB93}:
\begin{lemma}\label{embedding}
Let $\rho \in V^2(\Omega_T)$. Then there exists a constant $C_d$ depending only on $d$ such that
\begin{align}
\norm{\rho}_{\Leb^2(\Omega_T)}^2 \leq C_d \abs*{\set*{\abs{\rho}>0}}^{2/(2+d)} \norm{\rho}_{V^2(\Omega_T)}^2 \, .
\end{align}
\end{lemma}
\section{Bifurcation theory}
 We state here the Crandall--Rabinowitz theorem (cf. ~\cite{nirenberg2001topics,kielhofer2006bifurcation}) for bifurcations with a one-dimensional kernel.
\begin{theorem}\label{thm:cr}
	Consider a separable Hilbert space $X$ with $U \subset X$ an open neighbourhood of 0, and a nonlinear $C^2$ map, $F: U \times V \to X$, where $V$ is an open subset of $\R_+$ such that $F(0,\kappa)=0$ for all $\kappa \in V$. Assume the following conditions are satisfied for some $\kappa_* \in V$:
	\begin{enumerate}
		\item $D_x(0,\kappa_*)F$ is a Fredholm operator with index zero and has a one-dimensional kernel.
		\item $D^2_{x\kappa}(0,\kappa_*)F[\hat{v_0}] \notin \Ima(D_x(0,\kappa_*))$, where $\hat{v_0} \in \ker (D_x(0,\kappa_*)), \norm*{\hat{v_0}}
			=1$ \, .
	\end{enumerate}
	Then, there exists a nontrivial $C^1$ curve through $(0,\kappa_*)$ such that for some $\delta>0$,
	\begin{align}
           \{(x(s),\kappa(s))\ : s\in (-\delta,\delta), x(0)=0, \kappa(0)=\kappa_* \} \ ,
	\end{align}
and $F(x(s),\kappa(s))=0$. Additionally, for some neighbourhood of $(0,\kappa_*)$,
 this is the only such solution (apart from the trivial solution) and it has the following form:
\begin{equation}
x(s)=s\hat{v_0} +\Psi(s\hat{v_0},\psi(s)) \,, \qquad \kappa(s)=\psi(s) \ ,
\end{equation}
where $\Psi: \ker (D_x(0,\kappa_*)) \times \R_+ \to\bra*{\ker (D_x(0,\kappa_*)) }^\perp$ is a $C^1$ map and $\psi:(-\delta,\delta ) \to V$ is a $C^1$ function such that $\psi(0)=\kappa_*$. Furthermore if $D_\kappa \Psi(v_0,\kappa_*)=0$, we obtain a simplified expression of the form
\begin{align}
x(s)=s\hat{v_0} + r_1(s\hat{v_0},\psi(s)) \, ,
\end{align} 
such that $\lim\limits_{|s| +|\psi(s)-\kappa_*| \to 0} \frac{\norm*{r_1(s\hat{v_0},\psi(s))}}{|s| +|\psi(s)-\kappa_*|}=0$. 
\end{theorem}
\end{appendix}

\vspace{-0.5cm}

\bibliographystyle{myalpha}
\bibliography{biblio}
\end{document}